\documentclass[11pt,reqno]{amsart}

\usepackage[T1]{fontenc}
\usepackage[utf8]{inputenc}
\usepackage{lmodern}
\usepackage{microtype}
\usepackage{amsmath,amssymb,amsfonts,amsthm,mathtools,mathrsfs}
\usepackage[a4paper,margin=1in]{geometry}
\usepackage{enumitem}
\usepackage[dvipsnames]{xcolor}
\usepackage[hypertexnames=false]{hyperref}

\allowdisplaybreaks
\numberwithin{equation}{section}

\hypersetup{
  hidelinks,
  pdftitle={The Gromov-Ros conjecture for rank-one symmetric spaces of noncompact type},
  pdfauthor={David Kalaj}
}

\newtheorem{theorem}{Theorem}[section]
\newtheorem{proposition}[theorem]{Proposition}
\newtheorem{lemma}[theorem]{Lemma}
\newtheorem{corollary}[theorem]{Corollary}

\theoremstyle{definition}
\newtheorem{definition}[theorem]{Definition}

\theoremstyle{remark}
\newtheorem{remark}[theorem]{Remark}
\newtheorem{warning}[theorem]{Warning}

\newcommand{\CH}{\mathbb{C}\mathrm{H}}
\newcommand{\HH}{\mathbb{H}\mathrm{H}}
\newcommand{\OH}{\mathbb{O}\mathrm{H}}
\newcommand{\Area}{\operatorname{Area}}
\newcommand{\Vol}{\operatorname{Vol}}
\newcommand{\Per}{\operatorname{Per}}
\newcommand{\Jac}{\operatorname{Jac}}
\newcommand{\diver}{\operatorname{div}}
\newcommand{\divSigma}{\operatorname{div}_{\Sigma}}
\newcommand{\Ric}{\operatorname{Ric}}

\newcommand{\sphere}{\mathbb S_o^{n-1}}

\newcommand{\eps}{\varepsilon}

\newcommand{\distg}{d_g}
\newcommand{\Id}{\mathrm{Id}}
\newcommand{\cof}{\operatorname{cof}}
\newcommand{\ithead}[1]{%
  \par\medskip\noindent\textit{#1}\enspace
}

\title[Gromov--Ros conjecture in rank-one symmetric spaces]{The Gromov--Ros conjecture for rank-one symmetric spaces of noncompact type}

\author{David Kalaj} 
\address{Faculty of Natural Sciences and Mathematics, University of Montenegro, 81000 Podgorica, Montenegro}
\email{davidk@ucg.ac.me}
\urladdr{https://orcid.org/0000-0002-0017-6562}

\date{}

\subjclass[2020]{Primary 53C42, 49Q10; Secondary 53C40, 58E12}
\keywords{Gromov--Ros conjecture, rank-one symmetric spaces,
complex hyperbolic space, quaternionic hyperbolic space, Cayley hyperbolic plane,
isoperimetric inequality, finite-perimeter sets, volume radius,
geometric median, cofactor trace, BV radiality, stable constant mean curvature hypersurface,
weighted Bergman spaces, contractive embeddings, Lieb--Wehrl entropy,
Faber--Krahn inequality, Bergman--Fock limit}

\begin{document}
\begin{abstract}
We prove the Gromov--Ros conjecture for every rank-one symmetric space of
noncompact type: finite-perimeter isoperimetric regions are precisely the
geodesic balls, up to ambient isometry and null sets.  The proof is organized
uniformly in the root multiplicities $(p,q)$.  We introduce the volume radius
\[
 R(r)^n=n\int_0^r\sinh^{n-1}s\,\cosh^q s\,ds,
\]
which converts the polar volume form into
$R^{n-1}dR\,d\sigma$.  Hence radial--angular stretches
$R\mapsto e^{t\varphi(\theta)}R$ have the exact Jacobian
$e^{nt\varphi}$.  Starting from a volume geometric median, a log-partition
correction produces an exactly volume-preserving family of global
bi-Lipschitz stretches.  The reduced-boundary area formula and the ambient
cofactor yield a universal pointwise trace identity depending only on
$(p,q)$ and the horizontal and vertical components of the
measure-theoretic normal.  Its specializations for $q=1,3,7$ admit explicit
strict sign certificates, forcing the normal of an isoperimetric region to
be radial almost everywhere.  Isotropy invariance in BV and a
one-dimensional weighted endpoint comparison then force a single ball.
For complex hyperbolic space we additionally prove that every smooth bounded
fixed-volume stable critical domain is a geodesic ball.  The resulting
complex-hyperbolic isoperimetric theorem removes the geometric hypothesis in
several sharp weighted-Bergman contraction, Faber--Krahn, and Lieb--Wehrl
inequalities, and their high-weight scaling limit recovers holomorphic
Gaussian hypercontractivity.
\end{abstract}

\maketitle

\section{Introduction}

Let $M$ be a rank-one symmetric space of noncompact type and let
\[
 I_M(V)=\inf\{\Per(E):\Vol(E)=V\},\qquad V>0,
\]
be its isoperimetric profile.  The assertion that geodesic balls realize this profile, with equality only
for geodesic balls, is commonly called the \emph{Gromov--Ros conjecture}.
Gromov explicitly singled out the sharp isoperimetric problem in symmetric
spaces, observing that it was unknown even in real rank one; see
\cite[\S6.28(1/2+), pp.~331--332]{GromovMetricStructures}.
Ros later recorded the complex-hyperbolic case among the natural open
isoperimetric conjectures \cite[p.~177]{Ros2005}.
Following Silini \cite[p.~24]{SiliniThesis2024}, we use the name
\emph{Gromov--Ros conjecture} for the resulting rank-one problem. The real-hyperbolic
case is classical; see, for example, \cite{Ros2005,Ritore2023}.  The remaining
rank-one spaces are the complex hyperbolic spaces $\CH^m$, the quaternionic
hyperbolic spaces $\HH^m$, and the Cayley hyperbolic plane $\OH^2$.

Their polar geometry is anisotropic.  In the standard normalization the
possible radial curvature eigenvalues are $-1$ and $-4$, so the two angular
root spaces have different Jacobi scale factors.  This is the obstruction to
a direct use of the usual real-hyperbolic symmetrization.  The point of the
present paper is that the variational mechanism needed for isoperimetry is
nevertheless uniform across all rank-one spaces once the radial variable is
chosen from volume rather than from curvature.

\ithead{Main theorem.}
Our principal result is the full finite-perimeter classification.

\begin{theorem}[Gromov--Ros conjecture for rank-one symmetric spaces]
\label{thm:intro-rank-one}
Let $M$ be a rank-one symmetric space of noncompact type.  If a
finite-perimeter set $E\subset M$ and a geodesic ball $B\subset M$ satisfy
\[
 0<\Vol(E)=\Vol(B)<\infty,
\]
then
\[
 \Per(E)\ge \Per(B),
\]
with equality if and only if $E$ is a geodesic ball up to an ambient isometry
and a null set.
\end{theorem}

The non-real cases are proved simultaneously.  If $(p,q)$ are the short- and
long-root multiplicities and $n=p+q+1$, the polar volume density is
\[
 \mathfrak j_{p,q}(r)=\sinh^{n-1}r\,\cosh^q r.
\]
We introduce the \emph{volume radius}
\[
 R(r)^n=n\int_0^r\mathfrak j_{p,q}(s)\,ds.
\]
It is characterized by
\[
 R^{n-1}dR=\mathfrak j_{p,q}(r)\,dr,
 \qquad
 dV=R^{n-1}dR\,d\sigma.
\]
Thus, with $y=\log R$,
\[
 dV=e^{ny}\,dy\,d\sigma.
\]
For every spherical function $\varphi$ the radial--angular stretch
\[
 F_{t,\varphi}(y,\theta)=(y+t\varphi(\theta),\theta)
\]
has the exact ambient Jacobian
\[
 \Jac F_{t,\varphi}=e^{nt\varphi}.
\]
This is the structural replacement for the special complex-hyperbolic
identity obtained from $\rho=\sinh r$; indeed $R=\sinh r$ when $q=1$.

Starting from a volume geometric median, the coordinate functions
$\varphi_j(\theta)=\langle\theta,e_j\rangle$ are balanced.  A log-partition
correction then produces an $n$-parameter family of exactly volume-preserving
global bi-Lipschitz stretches.  The reduced-boundary area formula expresses
the perimeter of the deformed set through the ambient cofactor.  Writing
\[
 g=\mathfrak a(y)^2dy^2+b(y)^2g_{\mathcal H}+c(y)^2g_{\mathcal V},
 \qquad
 \mathfrak a b^pc^q=e^{ny},
\]
and decomposing a unit normal as
\[
 \nu=\alpha N+v+w,
 \qquad
 \beta=|v|^2,
 \qquad
 \gamma=|w|^2,
\]
we obtain a single root-multiplicity trace
\[
 \mathscr T_{p,q}
 =\sum_{j=1}^nJ_\Pi''(0,\varphi_j)-nJ_\Pi'(0,1).
\]
All mixed angular terms disappear after summation by the completeness
relations for the coordinate functions on the direction sphere.

The three non-real rank-one geometries now enter only through their
multiplicities
\[
 (p,q)=(2m-2,1),\qquad (4m-4,3),\qquad (8,7).
\]
For $q=1,3,7$ the universal trace becomes the negative of an explicit
nonnegative polynomial expression, and in every case equality occurs exactly
when
\[
 \beta=\gamma=0,
\]
that is, when the measure-theoretic normal is radial.  Global isoperimetric
minimality therefore forces radiality of the reduced boundary.  Invariance
under the isotropy group then makes the set radial modulo null sets, and a
one-dimensional weighted endpoint comparison leaves only a centered ball.
This proves Theorem~\ref{thm:intro-rank-one} without any regularity assumption
on the topological boundary.

The complex-hyperbolic case may be recorded separately as follows.

\begin{theorem}[Complex-hyperbolic Gromov--Ros theorem]
\label{thm:intro-gromov-ros}
Let $m\ge2$.  Every finite-perimeter isoperimetric region in $\CH^m$ agrees,
up to an ambient isometry and a null set, with a geodesic ball.  Equivalently,
if $E\subset\CH^m$ and a geodesic ball $B$ have the same positive finite
volume, then
\[
 \Per(E)\ge\Per(B),
\]
with equality if and only if $E$ is a geodesic ball up to an ambient isometry
and a null set.
\end{theorem}

\ithead{A separate smooth consequence in the complex case.}
The universal finite-perimeter proof is the main geometric argument of the
paper.  Complex hyperbolic space has an additional feature: in the coordinate
$\rho=\sinh r$ the radial--angular fields admit a particularly transparent
smooth second-variation analysis.  Retaining that simpler $q=1$ description,
we also prove the following rigidity statement for stable critical domains.

\begin{theorem}\label{thm:intro-main}
Let $m\ge2$ and let $\Omega\Subset\CH^m$ have smooth compact boundary.  If
$\partial\Omega$ is critical and stable for perimeter under a fixed-volume
constraint, then $\Omega$ is a geodesic ball.
\end{theorem}

This smooth theorem is logically independent of the existence theory for
isoperimetric regions.  Its proof uses the same $q=1$ trace sign, together
with cutoff arguments at the polar center and the Jacobi form for the
fixed-volume constraint.

\ithead{Relation with previous work.}
Ros recorded the complex-hyperbolic problem among the natural open
isoperimetric conjectures \cite[p.~2]{Ros2005}.  More recent work established
the conjecture in symmetric classes and obtained quantitative or local
stability results under additional hypotheses
\cite{Silini2024,Silini2025,Kalaj2026}.  In particular, the all-volume theorem
of Silini \cite[Corollary~1.6]{Silini2024} is restricted to the
Hopf-symmetric class.  None of these partial results is used as an input in
the proof below.

\ithead{Analytic consequences.}
The complex theorem supplies the geometric input in conditional sharp
weighted-Bergman contraction, Faber--Krahn, and Lieb--Wehrl results of Li--Su
and Singh \cite{LiSu2025,Singh2026}.  Through the resulting convex Bergman
inequality it also removes the corresponding hypothesis in the
higher-dimensional stability theorem of Melentijevi\'c
\cite[Theorem~4]{Melentijevic2025}.  These deductions are specifically
complex analytic and are not asserted for the quaternionic or Cayley spaces.

\ithead{Organization.}
Section~\ref{sec:rank-one-framework} contains the main proof: the volume
radius, global exact-volume stretches, the universal cofactor trace, the
three sign certificates, BV radiality, and the endpoint comparison.  It ends
with the full rank-one classification and the explicit complex-hyperbolic
profile.  Sections~\ref{sec:geometric-setup}--\ref{sec:integrated-trace} then
give the independent smooth stability argument in $\CH^m$, using the simpler
coordinate $\rho=\sinh r$.  Finally,
Section~\ref{sec:analytic-consequences} records the sharp complex-analytic
consequences and the Bergman--Fock scaling limit.

\section{The volume-radius framework for rank-one spaces}
\label{sec:rank-one-framework}

The proof of the finite-perimeter theorem is organized entirely in the
root-multiplicity variables $(p,q)$.  The complex, quaternionic, and Cayley
cases differ only at the final sign certificate.  We use the standard polar
root-space description of rank-one symmetric spaces; see, for example,
\cite{Silini2024}.

\subsection{Root multiplicities and the volume radius}

Let
\[
M\in\{\CH^m,\HH^m,\OH^2\}
\]
with the normalization in which the possible radial curvature eigenvalues are
$-1$ and $-4$, with multiplicities $p$ and $q$, respectively.  If one of the
multiplicities is zero, only the other eigenvalue occurs.  Fix $o\in M$.  At every
$\theta\in\mathbb S_o^{n-1}$ there is an orthogonal splitting
\begin{equation}
T_\theta\mathbb S_o^{n-1}
=
\mathcal H_\theta\oplus\mathcal V_\theta,
\qquad
\dim\mathcal H_\theta=p,
\qquad
\dim\mathcal V_\theta=q,
\label{eq:rank-one-angular-splitting}
\end{equation}
where
\[
\begin{array}{c|c|c|c}
M&p&q&n\\ \hline
\CH^m&2m-2&1&2m\\
\HH^m&4m-4&3&4m\\
\OH^2&8&7&16.
\end{array}
\]
Thus $p+q=n-1$.

\begin{remark}[Normalization and the case $\HH^1$]
\label{rem:rank-one-normalization}
The curvature normalization used in this section is harmless for the
isoperimetric classification.  A constant homothety of the metric multiplies
$n$-dimensional volume and $(n-1)$-dimensional perimeter by fixed positive
powers of the scale factor and therefore does not change the minimizing sets.
For $\HH^1$ one has $p=0$ and $q=3$, so the $-1$ radial eigenspace is absent;
with the present normalization the polar metric is the constant-curvature
real-hyperbolic metric of curvature $-4$.  The formulas below remain valid
without modification.
\end{remark}

The radial Jacobi equation gives
\[
g
=
dr^2
+
\sinh^2r\,g_{\mathcal H}
+
\sinh^2r\cosh^2r\,g_{\mathcal V},
\]
and hence
\[
dV
=
\mathfrak j_{p,q}(r)\,dr\,d\sigma,
\qquad
\mathfrak j_{p,q}(r)
=
\sinh^{n-1}r\,\cosh^q r.
\]

Define the \emph{volume radius} $R=R(r)$ by
\begin{equation}
R(r)^n
=
n\int_0^r\mathfrak j_{p,q}(s)\,ds.
\label{eq:volume-radius-definition}
\end{equation}
Then $R(0)=0$, $R'(r)>0$ for $r>0$, and
\[
R^{n-1}\,dR
=
\mathfrak j_{p,q}(r)\,dr.
\]
Consequently
\[
dV=R^{n-1}\,dR\,d\sigma
\]
and
\[
\Vol(B_r(o))
=
\frac{\omega_{n-1}}{n}R(r)^n.
\]

Put
\[
x=\sinh^2r.
\]
Since
\[
dx=2\sinh r\cosh r\,dr,
\]
we also have
\[
R^n
=
\frac{n}{2}
\int_0^x
t^{(n-2)/2}(1+t)^{(q-1)/2}\,dt.
\]
For $q=1$ this reduces to
\begin{equation}
R=\sinh r,
\label{eq:complex-volume-radius-recovery}
\end{equation}
so in complex hyperbolic space the volume radius is exactly the familiar
coordinate $\rho=\sinh r$.

For quaternionic hyperbolic space, where $q=3$ and $n=4m$,
\begin{equation}
R^n
=
x^{n/2}
\left(1+\frac{n}{n+2}x\right).
\label{eq:quaternionic-volume-radius}
\end{equation}
For the Cayley hyperbolic plane, where $q=7$ and $n=16$,
\begin{equation}
R^{16}
=
x^8
+\frac{8}{3}x^9
+\frac{12}{5}x^{10}
+\frac{8}{11}x^{11}.
\label{eq:Cayley-volume-radius}
\end{equation}

\subsection{Volume-radius stretches and exact-volume correction}

Put
\[
y=\log R.
\]
Then
\begin{equation}
dV=e^{ny}\,dy\,d\sigma.
\label{eq:rank-one-y-volume-form}
\end{equation}
For $\varphi\in C^1(\mathbb S_o^{n-1})$ define
\[
F_{t,\varphi}(y,\theta)
=
(y+t\varphi(\theta),\theta),
\]
or equivalently
\[
F_{t,\varphi}(R,\theta)
=
(e^{t\varphi(\theta)}R,\theta).
\]
Equation \eqref{eq:rank-one-y-volume-form} immediately gives
\[
\Jac F_{t,\varphi}=e^{nt\varphi}.
\]
The infinitesimal field is
\[
X_\varphi
=
\varphi R\partial_R,
\]
and
\[
\diver X_\varphi=n\varphi.
\]

\begin{lemma}
\label{lem:rank-one-median-balance}
Let $E\subset M$ be bounded and measurable with $0<\Vol(E)<\infty$.  The
functional
\[
\Phi_E(p)=\int_E d_g(p,z)\,dV_z
\]
attains a minimum.  If $o$ is any minimizer and
$\theta_o(z)\in T_oM$ is the initial unit tangent of the geodesic from $o$ to
$z\ne o$, then
\begin{equation}
\int_E\theta_o(z)\,dV_z=0.
\label{eq:rank-one-median-vector-balance}
\end{equation}
Equivalently, for every orthonormal basis $e_1,\ldots,e_n$ of $T_oM$ and
$\varphi_j(\theta)=\langle\theta,e_j\rangle$,
\begin{equation}
\int_E\varphi_j(\theta_o(z))\,dV_z=0,
\qquad j=1,\ldots,n.
\label{eq:rank-one-median-coordinate-balance}
\end{equation}
\end{lemma}

\begin{proof}
Every rank-one symmetric space of noncompact type is complete, simply
connected, and has nonpositive sectional curvature; in particular it is a
Hadamard manifold and has no cut locus.  Properness and continuity of
$\Phi_E$ are proved exactly as in the complex case by the triangle inequality
and Hopf--Rinow.  For $X\in T_pM$, the first variation formula for distance and
the bound
\[
|d_g(\exp_p(sX),z)-d_g(p,z)|\le d_g(\exp_p(sX),p)
\]
permit differentiation under the integral, giving
\[
D\Phi_E(p)[X]
=-\int_E\langle X,\theta_p(z)\rangle\,dV_z.
\]
At a minimizer this vanishes for every $X$, which proves
\eqref{eq:rank-one-median-vector-balance} and hence
\eqref{eq:rank-one-median-coordinate-balance}.
\end{proof}

Thus the infinitesimal fields $X_j=\varphi_jR\partial_R$ are balanced in the
same sense as in the complex case.

It is useful to record the metric in the $y$-coordinate.  Write
\[
g
=
\mathfrak a(y)^2dy^2
+b(y)^2g_{\mathcal H}
+c(y)^2g_{\mathcal V},
\]
where
\[
\mathfrak a(y)=\frac{dr}{dy},
\qquad
b(y)=\sinh r(y),
\qquad
c(y)=\sinh r(y)\cosh r(y).
\]
Comparison with \eqref{eq:rank-one-y-volume-form} gives the fundamental
identity
\begin{equation}
\mathfrak a\,b^pc^q=e^{ny}.
\label{eq:rank-one-abc-volume-identity}
\end{equation}

\begin{lemma}
\label{lem:rank-one-global-stretches}
For every $\psi\in C^1(\mathbb S_o^{n-1})$, the map
\[
F_\psi(o)=o,
\qquad
F_\psi(R,\theta)=(e^{\psi(\theta)}R,\theta)
\]
is an orientation-preserving global bi-Lipschitz homeomorphism with inverse
$F_{-\psi}$, and
\[
\Jac F_\psi=e^{n\psi}
\qquad\text{a.e. on }M.
\]
Its bi-Lipschitz constants depend only on
$\|\psi\|_{C^1(\mathbb S_o^{n-1})}$.

If $\psi=\psi(t,\cdot)$ is $C^2$ in $t$ with uniformly bounded $C^1$
spatial norm together with its first two $t$-derivatives, then for
\[
\mathcal J(t,x,\nu)
=
|\cof DF_{\psi(t)}(x)\nu|
\]
the quantities $\partial_t^k\mathcal J$, $k=0,1,2$, are uniformly bounded for
$x\ne o$ and $|\nu|=1$ on compact $t$-intervals.
\end{lemma}

\begin{proof}
The map preserves every radial ray and multiplies its volume radius by the
positive factor $e^{\psi(\theta)}$, hence it is globally bijective with inverse
$F_{-\psi}$.  We record the uniform radial estimates needed below.

Near $r=0$, the Taylor expansion of \eqref{eq:volume-radius-definition} gives
\[
R=r\bigl(1+O(r^2)\bigr),
\qquad
y=\log r+O(r^2).
\]
Consequently, as $y\to-\infty$,
\[
r=e^y\bigl(1+O(e^{2y})\bigr),
\qquad
\mathfrak a=e^y\bigl(1+O(e^{2y})\bigr),
\qquad
b=e^y\bigl(1+O(e^{2y})\bigr),
\qquad
c=e^y\bigl(1+O(e^{2y})\bigr).
\]
At the other end, if
\[
\kappa=\frac{n+q-1}{n},
\]
then
\[
R(r)=C_{p,q}e^{\kappa r}\bigl(1+O(e^{-2r})\bigr)
\]
for a positive constant $C_{p,q}$.  Hence, as $y\to+\infty$,
\[
r=\kappa^{-1}y+O(1)+O(e^{-2y/\kappa}),
\qquad
\mathfrak a=\kappa^{-1}+O(e^{-2y/\kappa}),
\]
while
\[
b=C_1e^{y/\kappa}\bigl(1+O(e^{-2y/\kappa})\bigr),
\qquad
c=C_2e^{2y/\kappa}\bigl(1+O(e^{-2y/\kappa})\bigr)
\]
for positive constants $C_1,C_2$.

We next justify the derivative bounds directly, without differentiating the
asymptotic remainders.  Put
\[
H_{p,q}(r)
=
\frac{d}{dr}\log\mathfrak j_{p,q}(r)
=
(n-1)\coth r+q\tanh r.
\]
Since
\[
\mathfrak a=\frac{dr}{dy}=\frac{R^n}{\mathfrak j_{p,q}(r)},
\]
one has
\[
\partial_y\log\mathfrak a=n-\mathfrak a H_{p,q},
\qquad
\partial_y\log b=\mathfrak a\coth r,
\qquad
\partial_y\log c=\mathfrak a(\coth r+\tanh r).
\]
Writing
\[
H_{p,q}'(r)=-(n-1)\operatorname{csch}^2r
+q\operatorname{sech}^2r,
\]
and using
$\partial_y\mathfrak a=\mathfrak a(n-\mathfrak a H_{p,q})$, a second
differentiation gives
\[
\begin{aligned}
\partial_y^2\log\mathfrak a
&=-\mathfrak a(n-\mathfrak a H_{p,q})H_{p,q}
  -\mathfrak a^2H_{p,q}',\\
\partial_y^2\log b
&=\mathfrak a(n-\mathfrak a H_{p,q})\coth r
  -\mathfrak a^2\operatorname{csch}^2r,\\
\partial_y^2\log c
&=\mathfrak a(n-\mathfrak a H_{p,q})(\coth r+\tanh r)\\
&\hspace{2.2cm}
  +\mathfrak a^2(-\operatorname{csch}^2r+\operatorname{sech}^2r).
\end{aligned}
\]
Near the pole, $\mathfrak a=r(1+O(r^2))$, while
$\coth r=r^{-1}+O(r)$,
$\operatorname{csch}^2r=r^{-2}+O(1)$, and $\tanh r=O(r)$;
hence every expression displayed above is bounded.  As $r\to\infty$,
$\mathfrak a\to\kappa^{-1}$, while $\coth r,\tanh r$ remain bounded and
$\operatorname{csch}^2r,\operatorname{sech}^2r$ decay exponentially.
Smoothness on a compact intervening interval therefore yields
\begin{equation}
\sup_{y\in\mathbb R}
\left(
|\partial_y\log U(y)|+|\partial_y^2\log U(y)|
\right)<\infty,
\qquad
U\in\{\mathfrak a,b,c\}.
\label{eq:rank-one-log-derivative-bounds}
\end{equation}
The endpoint expansions also give
$\mathfrak a/b\to1$ and $\mathfrak a/c\to1$ as $r\downarrow0$, whereas
both ratios tend to $0$ as $r\to\infty$; hence
$\mathfrak a/b$ and $\mathfrak a/c$ are globally bounded.

For a general exponent $\psi$, split
$\nabla_{\mathbb S}\psi=h+\zeta$ and put $y_+=y+\psi(\theta)$.  In source and
target orthonormal polar frames,
\[
DF_\psi=
\begin{pmatrix}
\dfrac{\mathfrak a(y_+)}{\mathfrak a(y)}
&\dfrac{\mathfrak a(y_+)}{b(y)}h^{\mathsf T}
&\dfrac{\mathfrak a(y_+)}{c(y)}\zeta^{\mathsf T}\\[0.8em]
0&\dfrac{b(y_+)}{b(y)}I_p&0\\[0.8em]
0&0&\dfrac{c(y_+)}{c(y)}I_q
\end{pmatrix}.
\]
The mean-value theorem applied to the logarithms, together with
\eqref{eq:rank-one-log-derivative-bounds}, bounds all diagonal ratios whenever
$\|\psi\|_{\infty}$ is bounded; the off-diagonal blocks are controlled by the
same estimates, the bounds for $\mathfrak a/b,\mathfrak a/c$, and
$\|d\psi\|_\infty$.  Applying the same estimate to $F_{-\psi}$ gives a uniform
bound for the inverse differential on $M\setminus\{o\}$.  The local expansion
$R=r(1+O(r^2))$ shows that $F_\psi$ and $F_{-\psi}$ extend continuously across
$o$.  To pass from the differential bounds to global metric bounds, join two
points by a minimizing geodesic.  Apart from at most the single parameter value
at which the geodesic meets $o$, the differential estimate bounds the length of
its image by a fixed multiple of the original length.  Hence
\[
d_g(F_\psi(x),F_\psi(x'))\le C\,d_g(x,x')
\]
for all $x,x'\in M$.  The same argument for $F_{-\psi}$ gives the reverse
Lipschitz estimate.  Thus $F_\psi$ is globally bi-Lipschitz.

The determinant of the displayed matrix is $e^{n\psi}$ by
\eqref{eq:rank-one-abc-volume-identity}, proving the Jacobian formula.  If
$\psi=\psi(t,\cdot)$, differentiating the matrix in $t$ up to order two uses
only the first two $y$-derivatives in
\eqref{eq:rank-one-log-derivative-bounds} and the assumed $C^1$ bounds on
$\partial_t^k\psi$, $k=0,1,2$.  Thus
$DF_{\psi(t)}$, its first two $t$-derivatives, and the inverse matrices are
uniformly bounded.  Since the cofactor is polynomial in the matrix entries and
$|\cof DF_{\psi(t)}\nu|$ is uniformly separated from zero for $|\nu|=1$,
differentiating its norm proves the asserted uniform bounds for
$\partial_t^k\mathcal J$, $k=0,1,2$.
\end{proof}

\begin{lemma}
\label{lem:rank-one-finite-perimeter-transport}
Let $E\subset M$ have finite perimeter and let
$t\mapsto\psi(t,\cdot)$ satisfy the hypotheses of
Lemma~\ref{lem:rank-one-global-stretches}.  Then
\begin{equation}
\Per(F_{\psi(t)}E)
=
\int_{\partial^*E}
\left|\cof DF_{\psi(t)}\,\nu_E\right|
\,d\mathcal H^{n-1}.
\label{eq:rank-one-finite-perimeter-transport}
\end{equation}
Moreover the left-hand side is $C^2$ in $t$, and its first two derivatives are
obtained by differentiating the integrand in
\eqref{eq:rank-one-finite-perimeter-transport}.
\end{lemma}

\begin{proof}
On $M\setminus\{o\}$, the maps $F_{\psi(t)}$ and their inverses are $C^1$
bi-Lipschitz diffeomorphisms.  We recall explicitly the finite-perimeter
transformation used here.  If $G$ is a $C^1$ bi-Lipschitz diffeomorphism between
coordinate domains and $E$ has finite perimeter, then, up to
$\mathcal H^{n-1}$-null sets,
\[
G(\partial^*E)=\partial^*G(E),
\qquad
\nu_{G(E)}(G(x))
=
\frac{\cof DG(x)\,\nu_E(x)}{|\cof DG(x)\,\nu_E(x)|},
\]
and the tangential area factor at $x\in\partial^*E$ is
$|\cof DG(x)\nu_E(x)|$; this is the standard reduced-boundary transformation
and rectifiable area formula for sets of finite perimeter
\cite{AFP2000,Maggi2012}.  Applying this statement in a countable smooth atlas
of $M\setminus\{o\}$, the coordinate metric factors combine exactly into the
Riemannian cofactor computed in orthonormal frames.  Since a singleton has zero
$\mathcal H^{n-1}$ measure and the perimeter measure of a finite-perimeter set
is $\mathcal H^{n-1}\llcorner\partial^*E$, neither side acquires a contribution
at $o$.  Summing the local formulas proves
\eqref{eq:rank-one-finite-perimeter-transport} on all of $M$.

By Lemma~\ref{lem:rank-one-global-stretches}, the integrand and its first two
$t$-derivatives are bounded uniformly in the boundary point on every compact
$t$-interval.  Since
$\mathcal H^{n-1}(\partial^*E)=\Per(E)<\infty$, dominated convergence may be
applied twice.  This proves the $C^2$ assertion and the differentiation
formula.
\end{proof}

Let $E$ be a bounded finite-perimeter set of positive volume and let $o$ be a
volume geometric median.  Define the angular probability measure $\mu_E$ on
$\mathbb S_o^{n-1}$ by
\[
\int_{\mathbb S_o^{n-1}}g(\theta)\,d\mu_E(\theta)
=
\frac{1}{\Vol(E)}
\int_E g(\theta_o(z))\,dV_z.
\]
The median balance is
\[
\int_{\mathbb S_o^{n-1}}\theta\,d\mu_E(\theta)=0.
\]
For $\mathbf a\in T_oM$ put
\[
Z_E(\mathbf a)
=
\int_{\mathbb S_o^{n-1}}
e^{n\langle\mathbf a,\theta\rangle}\,d\mu_E(\theta),
\qquad
s_E(\mathbf a)
=
-\frac{1}n\log Z_E(\mathbf a),
\]
and
\begin{equation}
\psi_{\mathbf a}(\theta)
=
\langle\mathbf a,\theta\rangle+s_E(\mathbf a).
\label{eq:rank-one-exact-volume-exponent}
\end{equation}
Then the Jacobian formula and the definition of $\mu_E$ give
\[
\Vol(F_{\psi_{\mathbf a}}E)
=
\int_Ee^{n\psi_{\mathbf a}(\theta_o(z))}\,dV_z
=
\Vol(E)e^{ns_E(\mathbf a)}Z_E(\mathbf a)
=
\Vol(E),
\]
that is,
\[
\Vol(F_{\psi_{\mathbf a}}E)=\Vol(E)
\qquad(\mathbf a\in T_oM).
\]
Differentiating the logarithmic moment-generating function, and using
Lemma~\ref{lem:rank-one-median-balance}, gives
\[
Ds_E(0)=0,
\qquad
D^2s_E(0)=-n\operatorname{Cov}_{\mu_E}(\Theta).
\]
For later use, we record explicitly the trace of this Hessian.  Since
$\Theta$ takes values in the unit direction sphere and the median balance gives
$\int \Theta\,d\mu_E=0$,
\[
\begin{aligned}
\operatorname{tr}\operatorname{Cov}_{\mu_E}(\Theta)
&=
\int_{\mathbb S_o^{n-1}}|\Theta|^2\,d\mu_E
-
\left|\int_{\mathbb S_o^{n-1}}\Theta\,d\mu_E\right|^2\\
&=1.
\end{aligned}
\]
Consequently
\begin{equation}
\operatorname{tr}D^2s_E(0)=-n.
\label{eq:rank-one-log-partition-trace}
\end{equation}


\begin{lemma}
\label{lem:cofactor}
Let $L:V\to W$ be an invertible linear map between oriented Euclidean
$n$-spaces, and let $\nu\in V$ be a unit vector.  Then
\begin{equation}
 J_{n-1}(L|_{\nu^\perp})
 =|\cof L\,\nu|
 =|\det L|\,|L^{-T}\nu|.
 \label{eq:cofactor}
\end{equation}
\end{lemma}

\begin{proof}
Choose an oriented orthonormal basis $e_1,\ldots,e_{n-1}$ of $\nu^\perp$ and
complete it by $e_n=\nu$.  The Hodge-star identity gives
\[
 *\bigl(Le_1\wedge\cdots\wedge Le_{n-1}\bigr)=(\cof L)\nu.
\]
Taking norms proves the first equality, while
$\cof L=(\det L)L^{-T}$ gives the second.
\end{proof}

\subsection{The universal cofactor trace}

Fix $(y,\theta)$ and abbreviate
\[
\mathfrak a=\mathfrak a(y),
\qquad
b=b(y),
\qquad
c=c(y).
\]
For $y_t=y+t\varphi(\theta)$ write
\[
\mathfrak a_t=\mathfrak a(y_t),
\qquad
b_t=b(y_t),
\qquad
c_t=c(y_t).
\]
Split the spherical gradient as
\[
\nabla_{\mathbb S}\varphi
=
h+\zeta,
\qquad
h\in\mathcal H_\theta,
\qquad
\zeta\in\mathcal V_\theta.
\]
In the source and target orthonormal polar frames,
\[
DF_{t,\varphi}
=
\begin{pmatrix}
\dfrac{\mathfrak a_t}{\mathfrak a}
&
\dfrac{t\mathfrak a_t}{b}h^{\mathsf T}
&
\dfrac{t\mathfrak a_t}{c}\zeta^{\mathsf T}
\\[0.8em]
0&\dfrac{b_t}{b}I_p&0
\\[0.8em]
0&0&\dfrac{c_t}{c}I_q
\end{pmatrix}.
\]
Indeed, the first row records the radial component generated by the angular
derivative of $y+t\varphi$, while the two diagonal angular blocks record the
change of the two Jacobi scale factors.  By
\eqref{eq:rank-one-abc-volume-identity},
\[
\det DF_{t,\varphi}
=
\frac{\mathfrak a_tb_t^pc_t^q}{\mathfrak a b^pc^q}
=
e^{nt\varphi}.
\]

Let $\Pi=\nu^\perp$ and decompose the unit normal as
\begin{equation}
\nu
=
\alpha N+v+w,
\qquad
v\in\mathcal H_\theta,
\qquad
w\in\mathcal V_\theta.
\label{eq:rank-one-normal-decomposition}
\end{equation}
Put
\[
\beta=|v|^2,
\qquad
\gamma=|w|^2,
\qquad
\alpha^2+\beta+\gamma=1.
\]
Solving $DF_{t,\varphi}^{T}\xi=\nu$ and using the cofactor identity gives
\[
J_\Pi(t,\varphi)
=
e^{nt\varphi}\sqrt{K(t)},
\]
where
\begin{equation}
\begin{aligned}
K(t)
={}&
\alpha^2
\left(\frac{\mathfrak a}{\mathfrak a_t}\right)^2
+
\left(\frac{b}{b_t}\right)^2
\left|v-t\alpha\frac{\mathfrak a}{b}h\right|^2
\\
&+
\left(\frac{c}{c_t}\right)^2
\left|w-t\alpha\frac{\mathfrak a}{c}\zeta\right|^2.
\end{aligned}
\label{eq:rank-one-K}
\end{equation}

\begin{theorem}[Universal rank-one trace identity]
\label{thm:universal-rank-one-trace}
Let
\[
A=\log\mathfrak a,
\qquad
B=\log b,
\qquad
C=\log c,
\]
and let a dot denote differentiation with respect to $y=\log R$.  Set
\[
\eta_H=\frac{\mathfrak a}{b},
\qquad
\eta_V=\frac{\mathfrak a}{c},
\]
and
\[
Z
=
\dot A\,\alpha^2
+
\dot B\,\beta
+
\dot C\,\gamma.
\]
Then
\begin{equation}
\begin{aligned}
\mathscr T_{p,q}
:={}&
\sum_{j=1}^nJ_\Pi''(0,\varphi_j)
-nJ_\Pi'(0,1)
\\
={}&
-nZ-Z^2
\\
&+
\alpha^2
\left(
-\ddot A+2\dot A^2
+\eta_H^2(p-\beta)
+\eta_V^2(q-\gamma)
\right)
\\
&+
\beta(-\ddot B+2\dot B^2)
+
\gamma(-\ddot C+2\dot C^2).
\end{aligned}
\label{eq:universal-rank-one-trace}
\end{equation}
\end{theorem}

\begin{proof}
For the coordinate functions
$\varphi_j(\theta)=\langle\theta,e_j\rangle$, write
\[
\nabla_{\mathbb S}\varphi_j=h_j+\zeta_j.
\]
The orthogonal splitting \eqref{eq:rank-one-angular-splitting} gives
\begin{align}
\sum_j\varphi_j^2&=1,
&
\sum_j\varphi_jh_j&=0,
&
\sum_j\varphi_j\zeta_j&=0,
\label{eq:rank-one-trace-identities-1}
\\
\sum_jh_j\otimes h_j&=I_{\mathcal H},
&
\sum_j\zeta_j\otimes\zeta_j&=I_{\mathcal V},
&
\sum_jh_j\otimes\zeta_j&=0.
\label{eq:rank-one-trace-identities-2}
\end{align}
In particular,
\begin{equation}
\sum_j\langle v,h_j\rangle^2=\beta,
\qquad
\sum_j\langle w,\zeta_j\rangle^2=\gamma.
\label{eq:rank-one-normal-traces}
\end{equation}

Expanding \eqref{eq:rank-one-K} as
\[
K(t)=1+k_1t+k_2t^2+O(t^3)
\]
gives
\[
k_1
=
-2\varphi Z
-2\alpha\eta_H\langle v,h\rangle
-2\alpha\eta_V\langle w,\zeta\rangle
\]
and
\begin{align}
k_2
={}&
\varphi^2
\left[
\alpha^2(2\dot A^2-\ddot A)
+\beta(2\dot B^2-\ddot B)
+\gamma(2\dot C^2-\ddot C)
\right]
\nonumber\\
&+
4\alpha\varphi
\left[
\dot B\eta_H\langle v,h\rangle
+\dot C\eta_V\langle w,\zeta\rangle
\right]
\nonumber\\
&+
\alpha^2
\left(
\eta_H^2|h|^2
+\eta_V^2|\zeta|^2
\right).
\label{eq:rank-one-k2}
\end{align}
Since
\[
J_\Pi(t,\varphi)
=
e^{nt\varphi}
\sqrt{1+k_1t+k_2t^2+O(t^3)},
\]
we have
\begin{equation}
J_\Pi''(0,\varphi)
=
n^2\varphi^2+n\varphi k_1+k_2-\frac{1}{4}k_1^2.
\label{eq:rank-one-J-second}
\end{equation}
For the uniform radial variation,
\[
J_\Pi'(0,1)=n-Z.
\]
Summing \eqref{eq:rank-one-J-second}, using
\eqref{eq:rank-one-trace-identities-1}--\eqref{eq:rank-one-normal-traces}, and
subtracting $nJ_\Pi'(0,1)$ gives
\eqref{eq:universal-rank-one-trace}.
\end{proof}

\subsection{The three trace sign certificates}

The universal formula now separates the argument into a geometric part,
which is common to all rank-one spaces, and a purely algebraic sign check for
$q=1,3,7$.  The first case is the complex trace polynomial.

\begin{theorem}[Complex trace certificate]
\label{thm:intrinsic-trace}
Let $q=1$, $p=n-2$, $n\ge4$, and put $x=\sinh^2r$.  For every hyperplane
$\Pi=\nu^\perp$ with normal decomposition
\eqref{eq:rank-one-normal-decomposition},
\begin{equation}
 \mathscr T_{n-2,1}
 =-\frac{\mathcal R_{n,x}(\beta,\gamma)}{(1+x)^2},
 \label{eq:intrinsic-trace}
\end{equation}
where
\begin{equation}
\begin{aligned}
\mathcal R_{n,x}(\beta,\gamma)
={}&(x^2-x-1)\beta^2
 +(4x^2-x-2)\beta\gamma
 +(4x^2-1)\gamma^2
\\
&+\bigl((n-2)x^2+(2n-1)x+n\bigr)\beta
\\
&+\bigl((2n-8)x^2+(3n-2)x+n\bigr)\gamma.
\end{aligned}
\label{eq:R-polynomial}
\end{equation}
\end{theorem}

\begin{proof}
For $q=1$ one has $R=\sinh r$ by
\eqref{eq:complex-volume-radius-recovery}.  Substitution of
\[
 \mathfrak a=\tanh r,\qquad b=\sinh r,\qquad
 c=\sinh r\cosh r
\]
into Theorem~\ref{thm:universal-rank-one-trace}, followed by
$\alpha^2=1-\beta-\gamma$ and $x=\sinh^2r$, gives
\eqref{eq:intrinsic-trace} after collecting coefficients.
\end{proof}

\begin{corollary}
\label{cor:tangent-plane-trace}
In the setting of Theorem~\ref{thm:intrinsic-trace}, if
\[
 J_\Pi(t,\varphi)=|\cof DF_{t,\varphi}\nu|,
\]
then
\[
 \sum_{j=1}^nJ_\Pi''(0,\varphi_j)-nJ_\Pi'(0,1)
 =-\frac{\mathcal R_{n,x}(\beta,\gamma)}{(1+x)^2}.
\]
In particular the formula applies at almost every approximate tangent
hyperplane of a reduced boundary.
\end{corollary}

\begin{theorem}[Strict positivity of the complex certificate]
\label{thm:strict-positivity}
Let
\[
 n\ge4,\qquad x\ge0,\qquad
 \beta,\gamma\ge0,\qquad \beta+\gamma\le1.
\]
Then
\[
 \mathcal R_{n,x}(\beta,\gamma)\ge0,
\]
and equality holds if and only if $\beta=\gamma=0$.
\end{theorem}

\begin{proof}
Put
\[
 \delta=1-\beta-\gamma\in[0,1],
 \qquad
 L_T=(n-1)+(3n-2)x+(2n-5)x^2.
\]
A direct rearrangement of \eqref{eq:R-polynomial} gives
\[
\begin{aligned}
\mathcal R_{n,x}
={}&(n-1)(1+x)^2\beta
 +\delta(1+x-x^2)\beta
 +L_T\gamma\\
&+\delta(1-3x^2)\gamma+x^2\gamma^2.
\end{aligned}
\]
For $0\le\delta\le1$ and real $c$ one has
$\delta c\ge\min\{0,c\}$.  Therefore
\[
 \mathcal R_{n,x}\ge c_H\beta+c_T\gamma+x^2\gamma^2,
\]
where
\[
 c_H=(n-1)(1+x)^2+\min\{0,1+x-x^2\}>0
\]
and
\[
 c_T=L_T+\min\{0,1-3x^2\}>0.
\]
In the only nontrivial branches these become, respectively,
\[
 n+(2n-1)x+(n-2)x^2
\]
and
\[
 n+(3n-2)x+(2n-8)x^2,
\]
which are strictly positive for $n\ge4$ and $x\ge0$.  Hence equality forces
$\beta=\gamma=0$, and the converse follows immediately from
\eqref{eq:R-polynomial}.
\end{proof}

\begin{theorem}[Quaternionic trace sign]
\label{thm:quaternionic-trace-sign}
Let $M=\HH^m$, $m\ge1$, so that
\[
n=4m\ge4,
\qquad
p=n-4,
\qquad
q=3.
\]
Put $x=\sinh^2r$ and $\delta=1-\beta-\gamma=\alpha^2$.  Then
\[
\mathscr T_{n-4,3}
=
-\frac{\mathcal P^{\mathbb H}_{n,x}(\beta,\gamma)}{(n+2)^2(1+x)^4},
\]
where
\begin{equation}
\mathcal P^{\mathbb H}_{n,x}
=
A_{20}\beta^2
+A_{11}\beta\gamma
+A_{02}\gamma^2
+A_{10}\beta\delta
+A_{01}\gamma\delta,
\label{eq:quaternionic-positive-decomposition}
\end{equation}
and
\begin{align}
A_{20}
={}&
(1+x)(nx+n+2)
\Bigl[(n^2+n)(1+x)^2-6x-2\Bigr],
\label{eq:quat-A20}
\\
A_{11}
={}&
(nx+n+2)
\Bigl[
3n^2x^3+(8n^2+6n-24)x^2
\nonumber\\
&\hspace{3.8cm}
+(7n^2+8n-16)x+2n^2+2n-4
\Bigr],
\label{eq:quat-A11}
\\
A_{02}
={}&
(nx+n+2)
\Bigl[
2n^2x^3+(5n^2+4n-16)x^2
\nonumber\\
&\hspace{3.8cm}
+(4n^2+5n-8)x+n^2+n-2
\Bigr],
\label{eq:quat-A02}
\\
A_{10}
={}&
n^3x^4
+(4n^3+5n^2-16n)x^3
\nonumber\\
&+(6n^3+14n^2-16n-48)x^2
\nonumber\\
&+(4n^3+13n^2+4n-12)x
+n^3+4n^2+4n,
\label{eq:quat-A10}
\\
A_{01}
={}&
(2n^3-4n^2)x^4
+(7n^3-48n)x^3
\nonumber\\
&+(9n^3+16n^2-48n-96)x^2
\nonumber\\
&+(5n^3+16n^2+4n-16)x
+n^3+4n^2+4n.
\label{eq:quat-A01}
\end{align}
Consequently
\[
\mathscr T_{n-4,3}\le0,
\qquad
\mathscr T_{n-4,3}=0
\quad\Longleftrightarrow\quad
\beta=\gamma=0.
\]
\end{theorem}

\begin{proof}
Substitute \eqref{eq:quaternionic-volume-radius} into
\eqref{eq:universal-rank-one-trace}, use
$\alpha^2=1-\beta-\gamma$, and clear the positive denominator
$(n+2)^2(1+x)^4$.  Direct expansion gives
\eqref{eq:quaternionic-positive-decomposition}--\eqref{eq:quat-A01}.
For $n\ge4$ and $x\ge0$, every displayed coefficient is strictly positive.
Since $\beta,\gamma,\delta\ge0$, the polynomial is nonnegative, and it can
vanish only when $\beta=\gamma=0$.  The converse is immediate.
\end{proof}

\begin{theorem}[Cayley trace sign]
\label{thm:Cayley-trace-sign}
Let $M=\OH^2$, so that
\[
n=16,
\qquad
p=8,
\qquad
q=7.
\]
Put $x=\sinh^2r$ and $\delta=1-\beta-\gamma$.  Then
\[
\mathscr T_{8,7}
=
-\frac{\mathcal P^{\mathbb O}_{x}(\beta,\gamma)}{27225(1+x)^8},
\]
where
\begin{equation}
\mathcal P^{\mathbb O}_{x}
=
B_{20}\beta^2
+B_{11}\beta\gamma
+B_{02}\gamma^2
+B_{10}\beta\delta
+B_{01}\gamma\delta.
\label{eq:Cayley-positive-decomposition}
\end{equation}
Set
\[
L(x)=120x^3+396x^2+440x+165.
\]
Then
\begin{align}
B_{20}
={}&
(1+x)L(x)
(2520x^4+9972x^3+14828x^2+9845x+2475),
\label{eq:Cayley-B20}
\\
B_{11}
={}&
2L(x)
(3600x^5+16944x^4+31752x^3+29568x^2
\nonumber\\
&\hspace{4.5cm}+13640x+2475),
\label{eq:Cayley-B11}
\\
B_{02}
={}&
L(x)
(4800x^5+21792x^4+39144x^3+34628x^2
\nonumber\\
&\hspace{4.5cm}+14960x+2475),
\label{eq:Cayley-B02}
\\
B_{10}
={}&
288000x^8
+2399040x^7
+8774208x^6
+18408720x^5
\nonumber\\
&+24235728x^4
+20497400x^3
+10865800x^2
+3294225x
+435600,
\label{eq:Cayley-B10}
\\
B_{01}
={}&
8\bigl(
64800x^8
+512640x^7
+1774992x^6
+3511992x^5
\nonumber\\
&\hspace{1.0cm}
+4339236x^4
+3421880x^3
+1675850x^2
+462825x
+54450
\bigr).
\label{eq:Cayley-B01}
\end{align}
Consequently
\[
\mathscr T_{8,7}\le0,
\qquad
\mathscr T_{8,7}=0
\quad\Longleftrightarrow\quad
\beta=\gamma=0.
\]
\end{theorem}

\begin{proof}
Substitute \eqref{eq:Cayley-volume-radius} into
\eqref{eq:universal-rank-one-trace} and clear the positive denominator
$27225(1+x)^8$.  Direct expansion gives
\eqref{eq:Cayley-positive-decomposition}--\eqref{eq:Cayley-B01}.  Every
coefficient of every polynomial displayed there is strictly positive.  Since
$\beta,\gamma,\delta\ge0$, the equality statement follows immediately.
\end{proof}


\subsection{Finite-perimeter rigidity from the trace sign}

For the exact-volume family \eqref{eq:rank-one-exact-volume-exponent}, define
\[
 \mathscr P_E(\mathbf a)=\Per(F_{\psi_{\mathbf a}}E).
\]

\begin{proposition}
\label{prop:rank-one-exact-volume-trace}
Let $E\subset M$ be a bounded finite-perimeter set of positive volume and let
$o$ be a volume geometric median.  Then $\mathscr P_E$ is $C^2$ near the
origin and
\begin{equation}
 \operatorname{tr}D^2\mathscr P_E(0)
 =\int_{\partial^*E}\mathscr T_{p,q}(r,\nu_E)\,d\mathcal H^{n-1}.
\label{eq:rank-one-perimeter-Hessian-trace}
\end{equation}
\end{proposition}

\begin{proof}
The map $\mathbf a\mapsto\psi_{\mathbf a}$ is smooth with values in
$C^1(\mathbb S_o^{n-1})$, and on a small closed ball its derivatives through
order two have uniformly bounded $C^1$ norm.  The cofactor integrand depends
smoothly on the finite-dimensional data $\psi_{\mathbf a}(\theta)$ and
$d\psi_{\mathbf a}(\theta)$.  The bounds from
Lemma~\ref{lem:rank-one-global-stretches} therefore dominate, uniformly on a
small parameter ball, every first and second partial derivative with respect
to $\mathbf a$.  Since $\mathcal H^{n-1}(\partial^*E)=\Per(E)<\infty$,
dominated convergence gives
\[
 \mathscr P_E(\mathbf a)
 =\int_{\partial^*E}|\cof DF_{\psi_{\mathbf a}}\nu_E|\,d\mathcal H^{n-1}
\]
and shows directly that $\mathscr P_E$ is $C^2$ near the origin, with all
first and second parameter derivatives obtained by differentiating under the
integral.

Choose an orthonormal basis $e_1,\ldots,e_n$ of $T_oM$ and put
$\varphi_j(\theta)=\langle\theta,e_j\rangle$.  Along the $j$-th coordinate
axis, write
\[
 \zeta_j(t)=t\varphi_j+s_E(te_j)\mathbf1.
\]
The median balance and \eqref{eq:rank-one-log-partition-trace} give
\[
 \zeta_j'(0)=\varphi_j,
 \qquad
 \zeta_j''(0)=D^2s_E(0)[e_j,e_j]\mathbf1.
\]
To make the second-order chain rule explicit, fix a reduced-boundary point
$x$ and set $\Pi_x=\nu_E(x)^\perp$.  Regard the pointwise tangential Jacobian
as the smooth functional
\[
 \mathcal F_x(\psi)
 :=|\cof DF_\psi(x)\nu_E(x)|.
\]
Then
\[
\begin{aligned}
\left.\frac{d^2}{dt^2}\right|_{t=0}\mathcal F_x(\zeta_j(t))
={}&D^2\mathcal F_x(0)[\varphi_j,\varphi_j]\\
&+D\mathcal F_x(0)
 \bigl[D^2s_E(0)[e_j,e_j]\mathbf1\bigr].
\end{aligned}
\]
By the notation used in the universal trace calculation,
\[
D^2\mathcal F_x(0)[\varphi_j,\varphi_j]
=J_{\Pi_x}''(0,\varphi_j),
\qquad
D\mathcal F_x(0)[\mathbf1]=J_{\Pi_x}'(0,1).
\]
There is no mixed first-order term involving $s_E$ because $Ds_E(0)=0$.
Consequently
\[
\begin{aligned}
 D^2\mathscr P_E(0)[e_j,e_j]
 =\int_{\partial^*E}\bigl(&J_{\Pi_x}''(0,\varphi_j)\\
 &+D^2s_E(0)[e_j,e_j]J_{\Pi_x}'(0,1)\bigr)
 \,d\mathcal H^{n-1}.
\end{aligned}
\]
Summing over $j$ and using
$\operatorname{tr}D^2s_E(0)=-n$ gives
\[
\begin{aligned}
\operatorname{tr}D^2\mathscr P_E(0)
&=\int_{\partial^*E}
\left(
\sum_{j=1}^nJ_{\Pi_x}''(0,\varphi_j)
-nJ_{\Pi_x}'(0,1)
\right)d\mathcal H^{n-1}\\
&=\int_{\partial^*E}\mathscr T_{p,q}(r,\nu_E)\,d\mathcal H^{n-1},
\end{aligned}
\]
which is \eqref{eq:rank-one-perimeter-Hessian-trace}.
\end{proof}

\begin{proposition}
\label{prop:rank-one-reduced-boundary-radiality}
Let
\[
 M=\CH^m\ (m\ge2),\qquad M=\HH^m\ (m\ge1),
 \qquad\text{or}\qquad M=\OH^2.
\]
If $E\subset M$ is a bounded finite-perimeter isoperimetric region of
positive volume and $o$ is a volume geometric median, then
\begin{equation}
 \nu_E=\pm N
 \qquad
 \mathcal H^{n-1}\text{-a.e. on }\partial^*E.
 \label{eq:rank-one-radial-normal}
\end{equation}
\end{proposition}

\begin{proof}
Every $F_{\psi_{\mathbf a}}E$ has the same volume as $E$.  Hence global
minimality gives $D^2\mathscr P_E(0)\ge0$ and therefore
\[
 0\le\operatorname{tr}D^2\mathscr P_E(0).
\]
If $M=\CH^m$, Theorems~\ref{thm:intrinsic-trace} and
\ref{thm:strict-positivity} imply
$\mathscr T_{n-2,1}\le0$, with equality exactly when
$\beta=\gamma=0$.  If $M=\HH^m$ or $M=\OH^2$, the corresponding statement is
Theorem~\ref{thm:quaternionic-trace-sign} or
Theorem~\ref{thm:Cayley-trace-sign}.  Proposition~\ref{prop:rank-one-exact-volume-trace}
therefore gives
\[
 0\le \operatorname{tr}D^2\mathscr P_E(0)
 =\int_{\partial^*E}\mathscr T_{p,q}\,d\mathcal H^{n-1}\le0.
\]
The strict equality statements force $\beta=\gamma=0$ almost everywhere,
which is equivalent to \eqref{eq:rank-one-radial-normal}.
\end{proof}

\begin{lemma}
\label{lem:rank-one-BV-radiality}
Under the assumptions of Proposition~\ref{prop:rank-one-reduced-boundary-radiality},
there exists a measurable set $\mathscr A\subset(0,\infty)$ such that, modulo
volume zero,
\begin{equation}
 E=\{(R,\theta):R\in\mathscr A,\ \theta\in\mathbb S_o^{n-1}\}.
 \label{eq:rank-one-radial-set}
\end{equation}
On every compact subinterval of $(0,\infty)$, the set $\mathscr A$ has finite
perimeter.
\end{lemma}

\begin{proof}
Put $u=\chi_E$ and let $K$ be the isotropy group fixing $o$.  In the three
cases under consideration,
\[
 K=U(m),\qquad K=\operatorname{Sp}(m)\operatorname{Sp}(1),
 \qquad K=\operatorname{Spin}(9),
\]
respectively.  Each $K$ is compact and connected and acts transitively on the
unit direction sphere.

The measure $Du$ is concentrated on $\partial^*E$ and has polar vector
$-\nu_E$ under the standard convention; see \cite{AFP2000,Maggi2012}.  Every
fundamental Killing field $Z$ of the $K$-action is tangent to centered
geodesic spheres.  Hence \eqref{eq:rank-one-radial-normal} gives
\[
 Z\cdot Du=0
\]
as a Radon measure.  Since $Z$ is Killing, $\operatorname{div}Z=0$, and thus
for every $\zeta\in C_c^\infty(M)$,
\[
 \int_MuZ\zeta\,dV
 =\int_Mu\operatorname{div}(\zeta Z)\,dV
 =-\int_M\zeta\,Z\cdot Du=0.
\]
Let $\Phi_t$ be the flow of $Z$ and, for
$\zeta\in C_c^\infty(M)$, set
\[
 \mathcal I(t)=\int_Mu(x)\zeta(\Phi_{-t}x)\,dV_x.
\]
Because $Z$ generates its own flow,
$Z(\zeta\circ\Phi_{-t})=(Z\zeta)\circ\Phi_{-t}$, and therefore the preceding
distributional identity gives
\[
 \mathcal I'(t)
 =-\int_Mu\,Z(\zeta\circ\Phi_{-t})\,dV
 =0.
\]
The flow of a Killing field preserves $dV$, so changing variables yields
\[
 \int_M (u\circ\Phi_t)\zeta\,dV
 =\mathcal I(t)=\mathcal I(0)
 =\int_Mu\zeta\,dV.
\]
Thus $u\circ\Phi_t=u$ in $L^1_{\mathrm{loc}}(M)$ for every $t$.  Since each of
the groups $K$ above is connected, it is generated by its one-parameter
subgroups; finite products of the corresponding flow invariances therefore give
$K$-invariance of $u$ in $L^1_{\mathrm{loc}}$.

Choose a Borel representative and average it with normalized Haar measure:
\[
 \bar u(x)=\int_Ku(kx)\,d\mu_K(k).
\]
We justify explicitly that the averaging does not change $u$ modulo null sets.
For every compact set $C\subset M$, Fubini's theorem and the already proved
$L^1_{\mathrm{loc}}$ invariance under every $k\in K$ give
\[
\begin{aligned}
\int_C|\bar u(x)-u(x)|\,dV_x
&\le
\int_K\int_C|u(kx)-u(x)|\,dV_x\,d\mu_K(k)\\
&=0.
\end{aligned}
\]
Hence $\bar u=u$ almost everywhere, while $\bar u$ is pointwise $K$-invariant
by Haar invariance.  Fix $\theta_0\in\mathbb S_o^{n-1}$ and put
$f(R)=\bar u(R,\theta_0)$.  Transitivity of $K$ on each centered sphere gives
$\bar u(R,\theta)=f(R)$.  Since $\bar u=u\in\{0,1\}$ almost everywhere,
Fubini's theorem in the polar measure
$R^{n-1}dR\,d\sigma$ implies that $f(R)\in\{0,1\}$ for almost every $R$.
Taking $\mathscr A=\{R:f(R)=1\}$ proves \eqref{eq:rank-one-radial-set}.

Finally, on every annulus $a<R<b$, polar coordinates are smooth and uniformly
nondegenerate, while $R^{n-1}$ is bounded above and away from zero.  The local
BV slicing theorem \cite[Theorem~3.103 and Remark~3.104]{AFP2000} applied in
product coordinates shows that $\chi_{\mathscr A}\in BV(a,b)$.
\end{proof}

\begin{lemma}
\label{lem:rank-one-endpoint-comparison}
Let $E_{\mathscr A}$ be the radial finite-perimeter set in
\eqref{eq:rank-one-radial-set}, assume that it is bounded and has positive
volume, and let $B_{r_0}(o)$ be the centered ball with
$\Vol(B_{r_0}(o))=\Vol(E_{\mathscr A})$.  Then
\[
 \Per(E_{\mathscr A})\ge\Per(B_{r_0}(o)),
\]
with equality if and only if
\begin{equation}
 \mathscr A=(0,R(r_0))
 \qquad\text{modulo one-dimensional measure}.
 \label{eq:rank-one-endpoint-equality}
\end{equation}
\end{lemma}

\begin{proof}
Fix $0<\varepsilon<L<\infty$.  Since
\[
 |dR|_g=\frac{dR}{dr}
 =\frac{\mathfrak j_{p,q}(r(R))}{R^{n-1}},
 \qquad
 dV=R^{n-1}dR\,d\sigma,
\]
strict one-dimensional BV approximation on $(\varepsilon,L)$
\cite[Theorem~3.9]{AFP2000}, together with lower semicontinuity of total
variation and the product structure of the radial representative, gives
\[
 \Per(E_{\mathscr A};\{\varepsilon<R<L\})
 \le
 \omega_{n-1}\int_{(\varepsilon,L)}
 \mathfrak j_{p,q}(r(R))\,d|D\chi_{\mathscr A}|(R).
\]
For the reverse inequality take a radial vector field $X=f(R)N$ with
$f\in C_c^1((\varepsilon,L))$ and $|f|\le1$.  Since
$N=(dR/dr)\partial_R$ and
$R^{n-1}(dR/dr)=\mathfrak j_{p,q}(r(R))$,
\[
 \operatorname{div}X
 =\frac{1}{R^{n-1}}\frac{d}{dR}
 \left(\mathfrak j_{p,q}(r(R))f(R)\right).
\]
The polar volume formula and one-dimensional BV integration by parts therefore
give
\[
\begin{aligned}
\int_{E_{\mathscr A}\cap\{\varepsilon<R<L\}}
\operatorname{div}X\,dV
&=\omega_{n-1}
\int_{(\varepsilon,L)}
\chi_{\mathscr A}(R)
\frac{d}{dR}
\left(\mathfrak j_{p,q}(r(R))f(R)\right)dR\\
&=-\omega_{n-1}
\int_{(\varepsilon,L)}
\mathfrak j_{p,q}(r(R))f(R)\,dD\chi_{\mathscr A}(R).
\end{aligned}
\]
Because $|X|=|f|\le1$, the dual definition of perimeter yields
\[
\Per(E_{\mathscr A};\{\varepsilon<R<L\})
\ge
-\omega_{n-1}
\int_{(\varepsilon,L)}
\mathfrak j_{p,q}(r(R))f(R)\,dD\chi_{\mathscr A}(R).
\]
Taking the supremum over such $f$ gives the weighted total variation,
\[
\Per(E_{\mathscr A};\{\varepsilon<R<L\})
\ge
\omega_{n-1}\int_{(\varepsilon,L)}
\mathfrak j_{p,q}(r(R))\,d|D\chi_{\mathscr A}|(R).
\]
Together with the preceding upper bound, this proves equality on every annulus.
Letting $\varepsilon\downarrow0$ and $L\uparrow\infty$ gives
\[
\begin{aligned}
 \Per(E_{\mathscr A})
 & =\omega_{n-1}\int_{(0,\infty)}
 \mathfrak j_{p,q}(r(R))\,d|D\chi_{\mathscr A}|(R)\\
 & =\omega_{n-1}\sum_{R\in\partial^*\mathscr A}
 \mathfrak j_{p,q}(r(R)).
\end{aligned}
\]
On every compact interval of $(0,\infty)$, a one-dimensional
finite-perimeter set is, modulo null sets, a finite union of intervals; hence
$|D\chi_{\mathscr A}|$ is the counting measure on its reduced endpoints there.
This gives the second equality above.  The limit
$\varepsilon\downarrow0$, $L\uparrow\infty$ is justified by monotone
convergence.  No boundary term is created at the pole: a singleton has zero
$\mathcal H^{n-1}$ measure, and moreover
$\mathfrak j_{p,q}(r(R))\to0$ as $R\downarrow0$.

Set
\[
 s=\frac{R^n}{n},\qquad
 \mathscr B=\left\{\frac{R^n}{n}:R\in\mathscr A\right\},\qquad
 \ell=|\mathscr B|=\frac{\Vol(E_{\mathscr A})}{\omega_{n-1}}.
\]
The increasing $C^1$ diffeomorphism $R\mapsto R^n/n$ carries reduced
endpoints of $\mathscr A$ to those of $\mathscr B$, and
\[
 \Per(E_{\mathscr A})
 =\omega_{n-1}\sum_{s\in\partial^*\mathscr B}\mathfrak q_{p,q}(s),
 \qquad
 \mathfrak q_{p,q}(s)
 =\mathfrak j_{p,q}\left(r((ns)^{1/n})\right).
\]
Both $r\mapsto\mathfrak j_{p,q}(r)$ and $r\mapsto R(r)$ are strictly
increasing, so $\mathfrak q_{p,q}$ is strictly increasing and
$\mathfrak q_{p,q}(0)=0$.

Let $b=\operatorname*{ess\,sup}\mathscr B$.  Boundedness and positive volume
give $0<\ell\le b<\infty$.  Choose $\delta>0$ so small that
$\mathfrak q_{p,q}\ge c_\delta>0$ on $(b-\delta,b+\delta)\cap(0,\infty)$.
The finite weighted endpoint sum then implies that
$\partial^*\mathscr B\cap(b-\delta,b+\delta)$ is finite.  The
one-dimensional structure theorem consequently represents $\mathscr B$, modulo
null sets, as a finite union of intervals in this neighborhood.  Since $b$ is
the essential supremum, one of those intervals has right endpoint $b$; thus
$b\in\partial^*\mathscr B$.  Consequently
\[
 \Per(E_{\mathscr A})
 \ge\omega_{n-1}\mathfrak q_{p,q}(b)
 \ge\omega_{n-1}\mathfrak q_{p,q}(\ell)
 =\Per(B_{r_0}(o)).
\]
If equality holds, strict monotonicity gives $b=\ell$.  Since
$\mathscr B\subset(0,b)$ modulo null sets and $|\mathscr B|=b$, one has
$\mathscr B=(0,b)$ modulo null sets.  Pulling back by $R\mapsto R^n/n$ yields
\eqref{eq:rank-one-endpoint-equality}.  The converse is immediate.
\end{proof}

\begin{theorem}[Non-real rank-one isoperimetric classification]
\label{thm:nonreal-rank-one-rigidity}
Let
\[
 M=\CH^m\ (m\ge2),\qquad M=\HH^m\ (m\ge1),
 \qquad\text{or}\qquad M=\OH^2.
\]
Every bounded finite-perimeter isoperimetric region of positive volume in
$M$ agrees, up to a null set and an ambient isometry, with a geodesic ball.
\end{theorem}

\begin{proof}
Let $E$ be such a region and choose a volume geometric median $o$.
Proposition~\ref{prop:rank-one-reduced-boundary-radiality} and
Lemma~\ref{lem:rank-one-BV-radiality} give a radial representative
$E_{\mathscr A}$ modulo null sets.  Since $E$ has a bounded representative,
there is $R_*>0$ with $\Vol(E\cap\{R>R_*\})=0$.  Because
$dV=R^{n-1}dR\,d\sigma$, after modifying $\mathscr A$ on a
one-dimensional null set we may assume $\mathscr A\subset(0,R_*)$.
Thus Lemma~\ref{lem:rank-one-endpoint-comparison} applies.  Minimality gives
$\Per(E_{\mathscr A})\le\Per(B_{r_0}(o))$, while the endpoint lemma gives the
reverse inequality.  Its equality statement forces
$\mathscr A=(0,R(r_0))$ modulo null sets.
\end{proof}

\begin{theorem}[Gromov--Ros conjecture for rank-one symmetric spaces]
\label{thm:rank-one-gromov-ros}
Let $M$ be a rank-one symmetric space of noncompact type.  If a
finite-perimeter set $E\subset M$ and a geodesic ball $B\subset M$ have the
same positive finite volume, then
\[
 \Per(E)\ge\Per(B),
\]
with equality if and only if $E$ is a geodesic ball up to an ambient isometry
and a null set.
\end{theorem}

\begin{proof}
The real-hyperbolic case is classical.  The complex line $\CH^1$ is a
real-hyperbolic surface up to the normalization used here.  Every remaining
case is covered by Theorem~\ref{thm:nonreal-rank-one-rigidity} once one knows
that isoperimetric regions exist and are bounded.

All rank-one symmetric spaces of noncompact type are complete and homogeneous.
Their full isometry groups act transitively, so the corresponding quotient is a
single point; in particular the action is cocompact.  Thus the hypotheses of
Ritor\'e's cocompact existence theorem are satisfied: at every positive volume
there is an isoperimetric region.  The accompanying boundedness result gives
every such minimizer a bounded representative
\cite[Section~4.4, Theorem~4.25 and Lemmas~4.26--4.27]{Ritore2023}.  Therefore
every minimizer in the non-real cases falls within
Theorem~\ref{thm:nonreal-rank-one-rigidity} and is a geodesic ball.

The inequality now follows from the definition of the isoperimetric profile.
If equality holds for an arbitrary finite-perimeter set $E$, then $E$ itself
is an isoperimetric region; the same boundedness theorem and the preceding
classification make it a geodesic ball modulo a null set.  Conversely every
geodesic ball realizes the profile.
\end{proof}

This proves Theorem~\ref{thm:intro-rank-one}.

\begin{corollary}
\label{thm:gromov-ros-all-m}
Let $m\ge2$, put $n=2m$, and let $V>0$.  Then
\begin{equation}
 I_{\CH^m}(V)
 =\omega_{n-1}
 \left(\frac{nV}{\omega_{n-1}}\right)^{(n-1)/n}
 \sqrt{1+\left(\frac{nV}{\omega_{n-1}}\right)^{2/n}}.
 \label{eq:all-dimensional-profile}
\end{equation}
Equivalently, every finite-perimeter set $E\subset\CH^m$ of positive finite
volume satisfies the corresponding sharp inequality, with equality exactly
for geodesic balls up to ambient isometry and null sets.
\end{corollary}

\begin{proof}
For $q=1$, $R=\sinh r$.  Hence
\[
 \Vol(B_r(o))=\frac{\omega_{n-1}}nR^n,
 \qquad
 \Area(\partial B_r(o))=\omega_{n-1}R^{n-1}\sqrt{1+R^2}.
\]
Eliminating $R$ and applying Theorem~\ref{thm:rank-one-gromov-ros} gives
\eqref{eq:all-dimensional-profile} and the equality statement.
\end{proof}

This also proves Theorem~\ref{thm:intro-gromov-ros}.

\section{A direct complex-hyperbolic smooth stability argument}\label{sec:geometric-setup}

The main finite-perimeter theorem has already been proved in the universal
volume-radius framework.  We now specialize to $M=\CH^m$ and retain the
simpler coordinate $\rho=\sinh r$ in order to prove the additional smooth
stability theorem.  This part is logically independent of the classification
of arbitrary finite-perimeter minimizers.


\subsection{Polar geometry}

Let
\[
M=\CH^m,
\qquad
n=2m\ge4,
\]
and let \(g\) denote the complex-hyperbolic metric normalized to have
holomorphic sectional curvature \(-4\).  Let \(J\) denote the canonical
parallel complex structure of the K\"ahler manifold \((M,g)\).  Thus
\[
J^2=-\Id,
\qquad
g(JX,JY)=g(X,Y),
\qquad
\nabla J=0.
\] Here \(\nabla=\nabla^g\) denotes the Levi--Civita connection of \(g\); in particular,
\(J\) is parallel.
We write \(\distg\) for the Riemannian distance induced by \(g\).  In the
unit-ball model \(\mathbb B^m\subset\mathbb C^m\), where
\(\langle z,w\rangle=\sum_{k=1}^m z_k\overline{w_k}\), the normalized
Bergman metric is
\begin{equation}
g_z(X,Y)
=
\operatorname{Re}
\frac{
(1-|z|^2)\langle X,Y\rangle
+\langle X,z\rangle\langle z,Y\rangle
}{(1-|z|^2)^2},
\label{eq:normalized-Bergman-metric}
\end{equation}
and its distance satisfies
\[
\cosh^2\!\distg(z,w)
=
\frac{|1-\langle z,w\rangle|^2}{(1-|z|^2)(1-|w|^2)}.
\]
In particular,
\[
\distg(0,z)=\operatorname{arctanh}|z|.
\]
See \cite[Chapter~1]{Zhu2005} and \cite{Goldman1999}.

Fix \(o\in M\) and set
\[
r=\distg(o,\cdot),
\qquad
\rho=\sinh r,
\qquad
x=\rho^2,
\qquad
a=1+x=\cosh^2r.
\]

The unit sphere of directions at \(o\) is denoted by
\[
\sphere\subset T_oM.
\]

At \(\theta\in\sphere\), define the Reeb direction
\[
\bar T=J\theta
\]
and the horizontal space
\[
\mathcal H_\theta
=
\{\theta,\bar T\}^{\perp}
\subset T_oM.
\]

Thus
\[
T_\theta\sphere
=
\mathcal H_\theta\oplus\mathbb R\bar T,
\qquad
\dim \mathcal H_\theta=n-2.
\]

Let \(g_{\mathbb S}\) be the round metric on the unit direction sphere.
We write \(g_{\mathcal H}=g_{\mathbb S}|_{\mathcal H_\theta}\), and \(\eta\) denotes the
\(g_{\mathbb S}\)-dual one-form of \(\bar T\). Thus
\(g_{\mathbb S}=g_{\mathcal H}+\eta^2\).

\ithead{Polar metric.}

For completeness, we briefly derive the radial scale factors. Along the
unit-speed radial geodesic \(\gamma_\theta\), with
\(N=\dot\gamma_\theta\), the curvature operator satisfies
\[
R(\,\cdot\,,N)N=-\Id
\quad\text{on the parallel translate of }\mathcal H_\theta,
\qquad
R(JN,N)N=-4JN.
\]
The normal Jacobi equation, with initial data \(Y(0)=0\) and
\(Y'(0)=E\), therefore gives the scale factor \(\sinh r\) in each
horizontal direction and
\[
\frac{1}{2}\sinh(2r)=\sinh r\cosh r
\]
in the Reeb direction. Consequently the exponential map pulls the ambient
metric back to the following polar form.

With the curvature normalization above, the polar metric is
\begin{equation}
g
=
dr^2
+
\sinh^2r\,g_{\mathcal H}
+
\sinh^2r\cosh^2r\,\eta^2.
\label{eq:polar-metric-r}
\end{equation}

Since
\[
d\rho=\cosh r\,dr=\sqrt a\,dr,
\]
equation \eqref{eq:polar-metric-r} becomes
\[
g
=
\frac{d\rho^2}{a}
+
\rho^2g_{\mathcal H}
+
\rho^2a\,\eta^2.
\]

\ithead{Volume density.}

The \(n-2\) horizontal directions contribute \(\rho^{n-2}\), the Reeb
direction contributes \(\rho\sqrt a\), and the radial direction contributes
\(a^{-1/2}\,d\rho\). Therefore
\begin{equation}
dV
=
\rho^{n-1}\,d\rho\,d\sigma.
\label{eq:polar-volume}
\end{equation}

\begin{warning}[Curvature normalization]
Every subsequent radial formula depends on the normalization
\(\operatorname{HolSec}=-4\). A change to holomorphic sectional curvature
\(-1\) rescales the radial variable and all ball formulas.
\end{warning}

\begin{remark}[Why the volume density becomes Euclidean in $\rho$]
The horizontal and Reeb scale factors are different, but their extra powers
of $\cosh r$ cancel the factor coming from the radial change of variables
$d\rho=\cosh r\,dr$.  Consequently the polar volume density is exactly
$\rho^{n-1}d\rho\,d\sigma$.  This cancellation is the reason that the
radial-angular deformation used later has the simple ambient Jacobian
$e^{nt\varphi}$.
\end{remark}

\ithead{Geodesic balls.}

Let
\[
\omega_{n-1}=|\mathbb S^{n-1}|.
\]
If \(R_0=\sinh r_0\), then
\begin{equation}
\Vol(B_{r_0}(o))
=
\frac{\omega_{n-1}}{n}R_0^n
\label{eq:ball-volume}
\end{equation}
and
\begin{equation}
\Area(\partial B_{r_0}(o))
=
\omega_{n-1}R_0^{n-1}\sqrt{1+R_0^2}.
\label{eq:ball-area}
\end{equation}

\ithead{The ball profile.}

Eliminating \(R_0\) between \eqref{eq:ball-volume} and
\eqref{eq:ball-area} gives, in every real dimension \(n=2m\ge4\),
\begin{equation}
I_n(V)
=
\omega_{n-1}
\left(\frac{nV}{\omega_{n-1}}\right)^{(n-1)/n}
\sqrt{
1+
\left(\frac{nV}{\omega_{n-1}}\right)^{2/n}
}.
\label{eq:ball-profile}
\end{equation}

Formula \eqref{eq:ball-profile} is the candidate profile supplied by
geodesic balls. The classification theorem below proves that it is the
isoperimetric profile.

\subsection{The volume geometric median}

\ithead{Definition and existence.}

For a bounded measurable set \(\Omega\Subset M\) of positive volume,
define
\[
\Phi_\Omega(p)
=
\int_\Omega \distg(p,y)\,dV_y.
\]

Fix \(p_0\in M\). By the triangle inequality,
\[
\distg(p,y)
\ge
\distg(p,p_0)-\distg(p_0,y).
\]
Hence
\[
\Phi_\Omega(p)
\ge
|\Omega|\,\distg(p,p_0)
-
\int_\Omega \distg(p_0,y)\,dV_y.
\]
Thus \(\Phi_\Omega\) is proper. Complex hyperbolic space is complete and
finite-dimensional, hence proper by Hopf--Rinow. Since \(\Phi_\Omega\) is
continuous, its sublevel sets are compact and it attains a minimum.

\begin{definition}
A minimizer \(o\) of \(\Phi_\Omega\) is called a
\emph{volume geometric median} of \(\Omega\).
\end{definition}

\ithead{Differentiability and balance.}

For \(p\in M\) and \(y\ne p\), let
\[
\theta_p(y)\in T_pM
\]
be the initial unit tangent of the geodesic from \(p\) to \(y\).

\begin{lemma}
\label{lem:median-differentiability}
For every \(p\in M\) and \(X\in T_pM\),
\begin{equation}
D\Phi_\Omega(p)[X]
=
-
\int_\Omega
\langle X,\theta_p(y)\rangle\,dV_y.
\label{eq:median-first-variation}
\end{equation}
Equivalently,
\begin{equation}
\nabla\Phi_\Omega(p)
=
-
\int_\Omega\theta_p(y)\,dV_y.
\label{eq:median-gradient}
\end{equation}
\end{lemma}

\begin{proof}
Let
\[
p_s=\exp_p(sX).
\]
Complex hyperbolic space is a Hadamard manifold, so there is no cut locus.
For every \(y\ne p\), the first variation formula for distance gives
\[
\left.
\frac{d}{ds}
\right|_{s=0}
\distg(p_s,y)
=
-
\langle X,\theta_p(y)\rangle.
\]

Moreover,
\[
|\distg(p_s,y)-\distg(p,y)|
\le
\distg(p_s,p).
\]
Therefore, for small nonzero \(s\),
\[
\left|
\frac{\distg(p_s,y)-\distg(p,y)}{s}
\right|
\le
|X|+o(1),
\]
uniformly in \(y\). The exceptional point \(y=p\) has zero volume.
Dominated convergence proves \eqref{eq:median-first-variation}.  In a smooth
local orthonormal frame, the same argument shows that
\[
p\longmapsto \int_\Omega\theta_p(y)\,dV_y
\]
is continuous: for a convergent sequence of base points, the direction vectors
converge for almost every \(y\) and have norm one.  Thus the G\^ateaux
differential of \(\Phi_\Omega\) is represented by a continuous covector field.
Since \(M\) is finite-dimensional, the continuous-G\^ateaux-differential
criterion gives \(\Phi_\Omega\in C^1\), and
\eqref{eq:median-gradient} follows.
\end{proof}

\begin{corollary}
\label{cor:median-balance}
If \(o\) is a volume geometric median, then
\begin{equation}
\int_\Omega\theta_o(y)\,dV_y=0.
\label{eq:median-balance}
\end{equation}
\end{corollary}

\begin{remark}[Interpretation of the balance identity]
The vector $\theta_o(y)$ has unit length and records only the direction from
$o$ to $y$, not its distance.  Thus the median condition is different from a
Riemannian center-of-mass condition.  It is precisely adapted to the vector
fields $X_j=\varphi_j\rho\partial_\rho$, whose divergence is the constant
multiple $n\varphi_j$.  This compatibility is what converts the interior
balance into zero boundary mean for the normal speeds.
\end{remark}

\ithead{Coordinate functions.}

Choose an orthonormal basis
\[
e_1,\dots,e_n
\]
of \(T_oM\), and define
\[
\varphi_j(\theta)
=
\langle\theta,e_j\rangle,
\qquad
j=1,\dots,n.
\]

Equation \eqref{eq:median-balance} is equivalent to
\[
\int_\Omega\varphi_j\,dV=0,
\qquad
j=1,\dots,n.
\]

\subsection{Balanced polar vector fields}

Away from the polar center \(o\), define
\[
X_j
=
\varphi_j(\theta)\,\rho\,\partial_\rho,
\qquad
j=1,\ldots,n.
\]
These are the infinitesimal generators of the radial-angular deformations
\[
(\rho,\theta)
\longmapsto
\bigl(e^{t\varphi_j(\theta)}\rho,\theta\bigr).
\]

Using the polar volume form \eqref{eq:polar-volume}, we compute
\[
\begin{aligned}
\diver X_j
&=
\rho^{1-n}
\partial_\rho
\left(
\rho^{n-1}\varphi_j(\theta)\rho
\right)
\\
&=
n\varphi_j(\theta).
\end{aligned}
\]

Let
\[
\Sigma=\partial\Omega
\]
be smooth, and let \(\nu\) denote its outer unit normal. Define
\[
f_j
=
\langle X_j,\nu\rangle.
\]
Thus \(f_j\) is the initial normal speed of the deformation generated by
\(X_j\).

Assume first that
\[
o\notin\overline{\Omega}.
\]
Then \(X_j\) is smooth on a neighborhood of \(\overline{\Omega}\), and the
divergence theorem gives
\[
\int_\Sigma f_j\,dA
=
\int_\Omega \diver X_j\,dV
=
n\int_\Omega\varphi_j\,dV.
\]
If \(o\) is a volume geometric median of \(\Omega\), then the median-balance
identity yields
\[
\int_\Omega\varphi_j\,dV=0.
\]
Consequently,
\begin{equation}
\int_\Sigma f_j\,dA=0.
\label{eq:fj-mean-zero-direct}
\end{equation}

When \(o\in\Omega\) or \(o\in\Sigma\), the field \(X_j\) is not smooth at
the polar center, so the preceding application of the divergence theorem
must be justified by cutoff arguments. These two cases are treated in
Sections~\ref{sec:interior-cutoff} and~\ref{sec:boundary-cutoff},
respectively.
\ithead{Local behavior near the center.}

Let \(z=(z_1,\dots,z_n)\) be geodesic normal coordinates centered at \(o\).
Then
\[
r=|z|,
\qquad
\theta=\frac{z}{r},
\qquad
\varphi_j=\frac{z_j}{r}.
\]

Since
\[
\rho\partial_\rho
=
\tanh r\,\partial_r
\]
and
\[
\partial_r
=
\frac{z}{r},
\]
we obtain
\[
X_j
=
\frac{z_j}{r}\tanh r\,\partial_r.
\]

Because
\[
\tanh r=r+O(r^3),
\]
we have
\[
|X_j|\le Cr
\]
near the pole. To control the first derivatives of \(X_j\), let \(z\) be
normal coordinates centered at \(o\). In the associated
normal-coordinate frame,
\[
X_j(z)
=
\frac{\tanh r}{r}\,
\frac{z_jz}{r}.
\]
The factor \(\tanh r/r\) extends smoothly across \(r=0\), while the map
\[
z\longmapsto \frac{z_jz}{|z|}
\]
is homogeneous of degree one and has uniformly bounded first derivatives
for \(z\ne0\). Consequently, after defining \(X_j(o)=0\), the vector field
\(X_j\) is locally Lipschitz near \(o\), and its first derivatives are
bounded almost everywhere.

If \(o\in\Sigma\), the smoothness of the unit normal \(\nu\) then implies
that
\[
f_j=\langle X_j,\nu\rangle
\]
is locally Lipschitz on \(\Sigma\) and satisfies
\[
|f_j|\le Cr,
\qquad
|\nabla_\Sigma f_j|\le C
\]
almost everywhere near \(o\). In particular, \(f_j\in H^1(\Sigma)\), where
\(\nabla_\Sigma f_j\) denotes the weak tangential gradient and
\[
H^1(\Sigma)
=
\{f\in L^2(\Sigma):\nabla_\Sigma f\in L^2(\Sigma)\}
\]
is equipped with its standard norm.

\subsection{Radial-angular deformations}

\ithead{Definition.}

For a smooth function \(\varphi\) on \(\sphere\), define
\[
F_{t,\varphi}(\rho,\theta)
=
\left(
e^{t\varphi(\theta)}\rho,
\theta
\right).
\]

Set
\[
u=e^{t\varphi},
\qquad
a_t=1+xu^2.
\]

At \(t=0\),
\[
\left.
\partial_tF_{t,\varphi}
\right|_{t=0}
=
\varphi\rho\partial_\rho.
\]

For \(\varphi=\varphi_j\), this velocity is \(X_j\).

\ithead{Orthonormal polar frames.}

At the source use
\[
N,E_1,\dots,E_{n-2},T_\rho,
\]
where
\[
N=\partial_r=\sqrt a\,\partial_\rho,
\]
the \(E_i\) are horizontal ambient unit vectors, and
\[
T_\rho
=
\frac{1}{\rho\sqrt a}\bar T
\]
is the ambient unit Reeb vector.

Use the corresponding primed orthonormal frame at the target radius
\(u\rho\).

Split the spherical derivative of \(\varphi\) as
\[
h=\nabla_{\mathcal H}\varphi,
\qquad
q=\bar T\varphi.
\]

Here \(h\) and \(q\) are derivatives on the unit direction sphere.

\ithead{Differential matrix.}

\begin{proposition}
\label{prop:basic-matrix}
At every point $(\rho,\theta)$ with $\rho>0$, set
\[
\rho'=u(\theta)\rho,
\qquad
u=e^{t\varphi(\theta)},
\qquad
a=1+\rho^{2},
\qquad
a_t=1+\rho'^2=1+\rho^{2}u^{2}.
\]
Let
\[
N,E_{1},\ldots,E_{n-2},T_{\rho}
\]
and
\[
N',E_{1}',\ldots,E_{n-2}',T_{\rho'}'
\]
denote the source and target orthonormal polar frames, respectively.
Then
\[
DF_{t,\varphi}
=
\begin{pmatrix}
u\sqrt{\dfrac{a}{a_t}}
&
\dfrac{ut}{\sqrt{a_t}}h^{\mathsf T}
&
\dfrac{ut}{\sqrt{aa_t}}q
\\[8pt]
0
&
uI_{n-2}
&
0
\\
0
&
0
&
u\sqrt{\dfrac{a_t}{a}}
\end{pmatrix},
\]
where
\[
h=\nabla_{\mathcal H}\varphi,
\qquad
q=\bar T\varphi.
\]
In particular,
\begin{equation}
\Jac F_{t,\varphi}
=
\det DF_{t,\varphi}
=
u^n
=
e^{nt\varphi}.
\label{eq:basic-ambient-jacobian}
\end{equation}
\end{proposition}

\begin{proof}
The deformation is
\[
F_{t,\varphi}(\rho,\theta)
=
\bigl(\rho',\theta\bigr),
\qquad
\rho'=u(\theta)\rho.
\]
Hence, for a coordinate tangent vector
\[
(s,Y)
\in
\mathbb R\,\partial_\rho
\oplus
T_\theta\mathbb S_o^{n-1},
\]
we have
\[
DF_{t,\varphi}(s,Y)
=
\left(
us+\rho\,du(Y),\,Y
\right).
\]
Since
\[
u=e^{t\varphi},
\qquad
du=ut\,d\varphi,
\]
it follows that
\begin{equation}
DF_{t,\varphi}(s,Y)
=
\left(
us+\rho ut\,d\varphi(Y),\,Y
\right).
\label{eq:master-differential}
\end{equation}

We now express \eqref{eq:master-differential} in the source and
target orthonormal polar frames.

At the source radius $\rho$, the radial unit vector is
\[
N=\sqrt a\,\partial_\rho.
\]
At the target radius $\rho'=u\rho$, it is
\[
N'=\sqrt{a_t}\,\partial_{\rho'}.
\]
Equivalently,
\[
\partial_{\rho'}=\frac{1}{\sqrt{a_t}}N'.
\]

For the radial input $N$, equation
\eqref{eq:master-differential} gives
\[
DF_{t,\varphi}(N)
=
DF_{t,\varphi}\bigl(\sqrt a\,\partial_\rho\bigr)
=
u\sqrt a\,\partial_{\rho'}.
\]
Therefore,
\begin{equation}
DF_{t,\varphi}(N)
=
u\sqrt{\frac{a}{a_t}}\,N'.
\label{eq:DF-radial}
\end{equation}

Next, choose a round-orthonormal horizontal basis
\[
\widehat E_1,\ldots,\widehat E_{n-2}
\in\mathcal H_\theta.
\]
Since the horizontal part of the ambient metric is
\[
\rho^2 g_{\mathcal H},
\]
the corresponding ambient unit vectors at the source are
\[
E_i=\frac{1}{\rho}\widehat E_i.
\]
At the target radius $\rho'=u\rho$, the corresponding ambient unit
vectors are
\[
E_i'=\frac{1}{u\rho}\widehat E_i.
\]

Applying \eqref{eq:master-differential} with
\[
s=0,
\qquad
Y=\frac{1}{\rho}\widehat E_i,
\]
the radial coordinate component is
\[
\rho ut\,
d\varphi\left(\frac{1}{\rho}\widehat E_i\right)
=
ut\,d\varphi(\widehat E_i).
\]
Its coefficient in the target radial unit direction is therefore
\[
\frac{ut}{\sqrt{a_t}}\,
d\varphi(\widehat E_i).
\]
Because $\widehat E_i$ is horizontal,
\[
d\varphi(\widehat E_i)
=
\left\langle
\nabla_{\mathcal H}\varphi,\widehat E_i
\right\rangle
=
\langle h,\widehat E_i\rangle.
\]

The angular component satisfies
\[
\frac{1}{\rho}\widehat E_i
=
u\left(\frac{1}{u\rho}\widehat E_i\right)
=
uE_i'.
\]
Consequently,
\begin{equation}
DF_{t,\varphi}(E_i)
=
\frac{ut}{\sqrt{a_t}}
\langle h,\widehat E_i\rangle N'
+
uE_i'.
\label{eq:DF-horizontal}
\end{equation}

Finally, the source unit Reeb vector is
\[
T_\rho
=
\frac{1}{\rho\sqrt a}\bar T,
\]
whereas at the target radius it is
\[
T_{\rho'}'
=
\frac{1}{u\rho\sqrt{a_t}}\bar T.
\]
Applying \eqref{eq:master-differential} with
\[
s=0,
\qquad
Y=\frac{1}{\rho\sqrt a}\bar T,
\]
and using
\[
q=\bar T\varphi,
\]
the radial coordinate component becomes
\[
\rho ut\,
d\varphi\left(
\frac{1}{\rho\sqrt a}\bar T
\right)
=
\frac{ut}{\sqrt a}\,q.
\]
After converting $\partial_{\rho'}$ into the target radial unit
vector $N'$, this gives
\[
\frac{ut}{\sqrt{aa_t}}q\,N'.
\]

The angular component satisfies
\[
\frac{1}{\rho\sqrt a}\bar T
=
u\sqrt{\frac{a_t}{a}}\,
\frac{1}{u\rho\sqrt{a_t}}\bar T
=
u\sqrt{\frac{a_t}{a}}\,T_{\rho'}'.
\]
Therefore,
\begin{equation}
DF_{t,\varphi}(T_\rho)
=
\frac{ut}{\sqrt{aa_t}}q\,N'
+
u\sqrt{\frac{a_t}{a}}\,T_{\rho'}'.
\label{eq:DF-Reeb}
\end{equation}

Equations \eqref{eq:DF-radial}, \eqref{eq:DF-horizontal}, and
\eqref{eq:DF-Reeb} give the three block columns of the asserted
matrix.

Since this matrix is upper triangular,
\[
\det DF_{t,\varphi}
=
u\sqrt{\frac{a}{a_t}}\,
u^{n-2}\,
u\sqrt{\frac{a_t}{a}}
=
u^n
=
e^{nt\varphi}.
\]
This concludes the proof.
\end{proof}

\section{Complex surface Jacobians and local second-order formulas}\label{sec:surface-trace}

\subsection{The intrinsic surface Jacobian}

At a point of \(\Sigma\), decompose the unit normal as
\[
\nu
=
\alpha N+v+\tau T_\rho,
\qquad
v\in\mathcal H_\theta.
\]
Here \(N=\nabla r\) is the radial unit direction, \(T_\rho=JN\) is the unit
Reeb direction, and \(\mathcal H_\theta=\{N,T_\rho\}^{\perp}\) is the
horizontal subspace tangent to the geodesic sphere centered at \(o\).

Set
\[
\beta=|v|^2,
\qquad
\gamma=\tau^2,
\qquad
\alpha^2+\beta+\gamma=1.
\]

No sign assumption is imposed on \(\alpha\).

\ithead{The cofactor identity.}
We use the tangential cofactor formula from
Lemma~\ref{lem:cofactor}.  It packages the restriction of the ambient
differential to an arbitrary tangent hyperplane and avoids any radial-graph
parametrization.

\ithead{Explicit inversion.}

Let \(L=DF_{t,\varphi}\). Solving
\[
L^Tw=\nu
\]
gives
\[
w_N
=
u^{-1}\alpha\sqrt{\frac{a_t}{a}},
\]
\[
w_H
=
u^{-1}
\left(
v-\frac{\alpha t}{\sqrt a}h
\right),
\]
and
\[
w_T
=
u^{-1}
\frac{\tau a-\alpha tq}{\sqrt{aa_t}}.
\]

Therefore
\begin{equation}
|L^{-T}\nu|^2
=
u^{-2}K(t),
\label{eq:Linverse}
\end{equation}
where
\begin{equation}
K(t)
=
\alpha^2\frac{a_t}{a}
+
\left|
v-\frac{\alpha t}{\sqrt a}h
\right|^2
+
\frac{(\tau a-\alpha tq)^2}{aa_t}.
\label{eq:K}
\end{equation}

Combining
\eqref{eq:basic-ambient-jacobian},
\eqref{eq:cofactor}, and
\eqref{eq:Linverse}, we obtain the following result.

\begin{theorem}[Exact intrinsic surface Jacobian]
\label{thm:surface-jacobian}
For every smooth hypersurface and every point away from the polar center,
let
\[
J_\Sigma(t,\varphi)
:=
J_{n-1}\!\left(
DF_{t,\varphi}|_{T\Sigma}
\right)
=
\left|\cof DF_{t,\varphi}\,\nu\right|
\]
denote the tangential metric Jacobian. Then
\[
J_\Sigma(t,\varphi)
=
u^{n-1}\sqrt{K(t)}.
\]
\end{theorem}

No radial-graph representation has been used.

\subsection{Second-order expansion}

Since
\[
u=e^{t\varphi},
\]
we have
\[
a_t
=
1+xe^{2t\varphi}
=
a+2x\varphi t+2x\varphi^2t^2+O(t^3).
\]

Write
\[
K(t)
=
1+k_1t+k_2t^2+O(t^3).
\]

Let
\[
s=\langle v,h\rangle.
\]

A direct expansion of \eqref{eq:K} gives
\[
k_1
=
\frac{2x\varphi}{a}(\alpha^2-\gamma)
-
\frac{2\alpha}{\sqrt a}s
-
\frac{2\alpha\tau}{a}q
\]
and
\[
\begin{aligned}
k_2
={}&
\varphi^2
\left(
\frac{2\alpha^2x}{a}
+
\frac{2\gamma x(x-1)}{a^2}
\right)
+
\frac{\alpha^2}{a}|h|^2
\\
&+
\frac{4\alpha\tau x}{a^2}\varphi q
+
\frac{\alpha^2}{a^2}q^2.
\end{aligned}
\]

Put
\[
b=n-1.
\]

From
\[
J_\Sigma(t,\varphi)
=
e^{bt\varphi}
\sqrt{
1+k_1t+k_2t^2+O(t^3)
},
\]
we obtain
\begin{equation}
J_\Sigma''(0,\varphi)
=
b^2\varphi^2
+
b\varphi k_1
+
k_2
-
\frac{1}{4}k_1^2.
\label{eq:J-second}
\end{equation}

For the uniform deformation \(\varphi\equiv1\),
\begin{equation}
J_\Sigma'(0,1)
=
b+\frac{x}{a}(\alpha^2-\gamma).
\label{eq:J-radial-first}
\end{equation}

\subsection{Complex specialization of the universal trace}

For $q=1$ the volume radius is $R=\rho=\sinh r$.  Therefore the universal
rank-one trace in Theorem~\ref{thm:universal-rank-one-trace} specializes to
Theorem~\ref{thm:intrinsic-trace} and its strictly positive polynomial
certificate in Theorem~\ref{thm:strict-positivity}.  In particular, at every
fixed tangent hyperplane,
\[
 \sum_{j=1}^nJ_\Pi''(0,\varphi_j)-nJ_\Pi'(0,1)
 =-\frac{\mathcal R_{n,x}(\beta,\gamma)}{(1+x)^2},
\]
and the right-hand side vanishes exactly for a radial normal.  The explicit
formulas \eqref{eq:J-second} and \eqref{eq:J-radial-first} retained above are
used only for the cutoff estimates in the smooth argument.

\section{Cutoffs and variational justification}\label{sec:cutoff-variational}
\subsection{The interior-center cutoff}\label{sec:interior-cutoff}

The case
\[
o\in\Omega
\]
requires a smooth ambient approximation even though the boundary is separated
from the polar singularity.  Indeed, the surface deformation is already
smooth near $\Sigma$, but the enclosed-volume functional depends on an
ambient map on all of $\Omega$.  The cutoff makes that ambient map smooth at
$o$ while leaving it unchanged on a fixed neighborhood of the boundary.

\begin{lemma}
\label{lem:interior-cutoff}
Assume
\[
o\in\Omega,
\qquad
o\notin\Sigma.
\]
Let
\[
d_0=\distg(o,\Sigma)>0.
\]
Choose
\[
0<\eps<\frac{d_0}{4}
\]
and a smooth radial cutoff \(\chi_\eps\) satisfying
\[
\chi_\eps=0
\quad\text{on }[0,\eps],
\]
\[
\chi_\eps=1
\quad\text{on }[2\eps,\infty),
\]
and
\[
|\chi_\eps'|\le\frac{C}{\eps}.
\]

Define
\[
\zeta_{j,\eps}(r,\theta)
=
\chi_\eps(r)\varphi_j(\theta)
\]
and extend it smoothly by zero at \(o\). Set
\[
F^\eps_{t,j}(\rho,\theta)
=
\left(
e^{t\zeta_{j,\eps}(\rho,\theta)}\rho,
\theta
\right).
\]

Let \(A_j(t)\) and \(V_j(t)\) denote the area and enclosed volume under the
exact polar deformation \(F_{t,\varphi_j}\).
Although this map need not be differentiable at \(o\), extend it by
\(F_{t,\varphi_j}(o)=o\).  Lemma~\ref{lem:rank-one-global-stretches}
shows that it is bi-Lipschitz for small \(t\), with inverse
\(F_{-t,\varphi_j}\).  Hence the Lipschitz area formula and
\eqref{eq:basic-ambient-jacobian} give
\[
V_j(t)
=
\int_\Omega e^{nt\varphi_j}\,dV,
\]
where the value assigned to \(\varphi_j\) at \(o\) is irrelevant.

Then:

\begin{enumerate}[label=\textup{(\roman*)}]

\item
For all sufficiently small \(|t|\), with a time interval independent of
\(\eps\), \(F^\eps_{t,j}\) is a smooth ambient diffeomorphism.

\item
The cutoff deformation agrees with the exact polar deformation on a
neighborhood of \(\Sigma\).

\item
If \(A^\eps_j(t)\) and \(V^\eps_j(t)\) are the area and enclosed volume
under \(F^\eps_{t,j}\), then
\begin{equation}
(A^\eps_j)'(0)=A_j'(0),
\qquad
(A^\eps_j)''(0)=A_j''(0),
\label{eq:interior-area-equality}
\end{equation}
and
\begin{equation}
(V^\eps_j)'(0)\to V_j'(0),
\qquad
(V^\eps_j)''(0)\to V_j''(0).
\label{eq:interior-volume-convergence}
\end{equation}

\item
If \(o\) is a volume geometric median, then
\begin{equation}
\int_\Sigma f_j\,dA=0.
\label{eq:interior-zero-mean}
\end{equation}

\item
If \(H\equiv\lambda\), then
\begin{equation}
Q_\Sigma(f_j)
=
A_j''(0)-\lambda V_j''(0).
\label{eq:interior-Hessian-identification}
\end{equation}

\end{enumerate}
\end{lemma}

\begin{proof}
Because \(2\eps<d_0/2\), one has
\[
\chi_\eps\equiv1
\]
on a fixed neighborhood of \(\Sigma\). Thus the cutoff and exact
deformations agree near the boundary, proving
\eqref{eq:interior-area-equality}.

For a general radial-dependent exponent \(\zeta\),
\[
\rho_t=e^{t\zeta}\rho
\]
and
\[
d\rho_t
=
e^{t\zeta}
\left[
(1+t\rho\partial_\rho\zeta)d\rho
+
t\rho\,d_{\mathbb S}\zeta
\right].
\]
The angular term wedges to zero against the full spherical volume form.
Therefore
\begin{equation}
\Jac F_t
=
e^{nt\zeta}
(1+t\rho\partial_\rho\zeta).
\label{eq:general-ambient-Jacobian}
\end{equation}

On the transition annulus
\[
\eps<r<2\eps,
\]
one has
\[
|\zeta_{j,\eps}|\le1
\]
and
\[
|\rho\partial_\rho\zeta_{j,\eps}|
\le C
\]
uniformly in \(\eps\). Indeed,
\[
\rho\partial_\rho\zeta_{j,\eps}
=
\rho\chi_\eps'(r)\frac{dr}{d\rho}\varphi_j,
\qquad
\frac{dr}{d\rho}
=
\frac{1}{\sqrt{1+\rho^2}}.
\]
On this annulus \(r\asymp\eps\), and
\(\rho=\sinh r=r+O(r^3)\) as \(r\to0\); hence
\[
\rho\asymp r\asymp\eps.
\]
Using \(|\chi_\eps'|\le C/\eps\) and
\(|\varphi_j|\le1\), we therefore obtain
\[
|\rho\partial_\rho\zeta_{j,\eps}|
\le
\frac{C\rho}{\eps\sqrt{1+\rho^2}}
\le C.
\]

Choose \(t_0>0\), independently of \(\eps\), so that
\[
1+t\rho\partial_\rho\zeta_{j,\eps}\ge\frac{1}{2}
\qquad (|t|\le t_0).
\]
For each fixed direction \(\theta\), the radial map
\[
\rho\longmapsto e^{t\zeta_{j,\eps}(\rho,\theta)}\rho
\]
is then strictly increasing. It is the identity near \(0\), and for large
\(\rho\) it is multiplication by the positive number
\(e^{t\varphi_j(\theta)}\). It is therefore a diffeomorphism of
\([0,\infty)\), and the full polar map is a smooth global diffeomorphism.
This proves the first assertion.

Hence the first two \(t\)-derivatives at \(t=0\) of
\eqref{eq:general-ambient-Jacobian} are uniformly bounded. The exact and
cutoff deformations differ only in \(B_{2\eps}(o)\), whose volume is
\(O(\eps^n)\). This proves
\eqref{eq:interior-volume-convergence}.

The first variation of volume and the agreement of the two deformations
near \(\Sigma\) give
\[
\int_\Sigma f_j\,dA=(V^\eps_j)'(0).
\]
If \(o\) is a volume geometric median, then
\[
(V^\eps_j)'(0)
\longrightarrow
V_j'(0)
=n\int_\Omega\varphi_j\,dV
=0,
\]
which proves \eqref{eq:interior-zero-mean}.

Along \(\Sigma\), the cutoff deformation has the same initial normal speed \(f_j\)
as the exact flow. Lemma~\ref{lem:Hessian-property} gives
\[
(A^\eps_j)''(0)-\lambda(V^\eps_j)''(0)
=
Q_\Sigma(f_j).
\]
Passing to the limit using
\eqref{eq:interior-area-equality} and
\eqref{eq:interior-volume-convergence} proves
\eqref{eq:interior-Hessian-identification}.
\end{proof}

\subsection{The boundary-center cutoff}\label{sec:boundary-cutoff}

Assume now that
\[
o\in\Sigma.
\]
This is the genuinely singular configuration.  Both the ambient polar field
and its normal trace meet the center.  The estimate $|f_j|=O(r)$ is crucial:
it compensates for the $O(\eps^{-1})$ derivative of the cutoff and yields
strong $H^1$ convergence on the $(n-1)$-dimensional boundary.  Two different
approximations will be used: one comes from an explicit ambient deformation, while
the other has exactly zero mean and is admissible for stability.

Let \(\chi_\eps\) be a smooth radial cutoff satisfying
\[
\chi_\eps=0
\quad\text{on }r\le\eps,
\]
\[
\chi_\eps=1
\quad\text{on }r\ge2\eps,
\]
and
\[
|\chi_\eps'|\le\frac{C}{\eps}.
\]

Set
\[
f_{j,\eps}
=
\chi_\eps f_j.
\]

\ithead{The \texorpdfstring{\(H^1\)}{H1} approximation.}

Since
\[
|f_j|\le Cr
\]
and
\[
|\nabla_\Sigma f_j|\le C
\]
near \(o\), while
\[
\Area(\Sigma\cap B_{2\eps}(o))
=
O(\eps^{n-1}),
\]
we obtain
\[
\|f_{j,\eps}-f_j\|_{L^2(\Sigma)}^2
\le
C\eps^{n+1}
\]
and
\[
\|
\nabla_\Sigma(f_{j,\eps}-f_j)
\|_{L^2(\Sigma)}^2
\le
C\eps^{n-1}.
\]

Therefore
\begin{equation}
f_{j,\eps}
\longrightarrow
f_j
\quad\text{strongly in }H^1(\Sigma).
\label{eq:fjeps-H1}
\end{equation}

\ithead{The weak divergence identity.}

\begin{lemma}
\label{lem:boundary-zero-mean}
If \(o\in\Sigma\), then
\[
\int_\Sigma f_j\,dA
=
n\int_\Omega\varphi_j\,dV.
\]
In particular, at a volume geometric median,
\[
\int_\Sigma f_j\,dA=0.
\]
\end{lemma}

\begin{proof}
Apply the divergence theorem to the smooth vector field
\[
\chi_\eps X_j.
\]
One obtains
\[
\int_\Sigma
\chi_\eps\langle X_j,\nu\rangle\,dA
=
\int_\Omega
\chi_\eps\diver X_j\,dV
+
\int_\Omega
\langle\nabla\chi_\eps,X_j\rangle\,dV.
\]

The first term converges to
\[
n\int_\Omega\varphi_j\,dV.
\]

The second term is supported on
\[
B_{2\eps}(o)\setminus B_\eps(o).
\]
On this annulus,
\[
|\nabla\chi_\eps|
\le
\frac{C}{\eps}
\]
and
\[
|X_j|\le Cr,
\]
so the integrand is uniformly bounded. Since the annular volume is
\(O(\eps^n)\), that term tends to zero.

The left-hand side converges to
\[
\int_\Sigma f_j\,dA
\]
by \eqref{eq:fjeps-H1}.
\end{proof}

\ithead{Exact mean correction.}

Assume for the remainder of this subsection that \(o\) is a volume geometric
median.  Fix
\[
\psi\in C_c^\infty(\Sigma\setminus\{o\})
\]
with
\[
\int_\Sigma\psi\,dA=1.
\]

Define
\[
m_{j,\eps}
=
\int_\Sigma f_{j,\eps}\,dA
\]
and
\begin{equation}
g_{j,\eps}
=
f_{j,\eps}-m_{j,\eps}\psi.
\label{eq:gjeps}
\end{equation}

Then
\[
\int_\Sigma g_{j,\eps}\,dA=0.
\]

By Lemma~\ref{lem:boundary-zero-mean},
\[
\int_\Sigma f_j\,dA=0,
\]
we have
\[
m_{j,\eps}
=
-\int_\Sigma(1-\chi_\eps)f_j\,dA
=
O(\eps^n).
\]

Consequently,
\begin{equation}
g_{j,\eps}
\longrightarrow
f_j
\quad\text{strongly in }H^1(\Sigma).
\label{eq:gjeps-H1}
\end{equation}

\subsection{Radial-dependent cutoff deformations}

Define
\[
\zeta_{j,\eps}(r,\theta)
=
\chi_\eps(r)\varphi_j(\theta)
\]
and extend it smoothly by zero near \(o\).

Consider
\[
F^\eps_{t,j}(\rho,\theta)
=
\left(
e^{t\zeta_{j,\eps}(\rho,\theta)}\rho,
\theta
\right).
\]

Write, temporarily,
\[
\zeta=\zeta_{j,\eps},
\qquad
u=e^{t\zeta},
\qquad
\rho_t=u\rho,
\qquad
a_t=1+\rho_t^2.
\]

Also define
\[
\zeta_\rho=\partial_\rho\zeta,
\qquad
h_\zeta=\nabla_{\mathcal H}\zeta,
\qquad
q_\zeta=\bar T\zeta.
\]

\begin{lemma}
\label{lem:cutoff-matrix}
In the source and target orthonormal polar frames,
\begin{equation}
DF^\eps_{t,j}
=
\begin{pmatrix}
u(1+t\rho\zeta_\rho)\sqrt{a/a_t}
&
\dfrac{ut}{\sqrt{a_t}}h_\zeta^{\mathsf T}
&
\dfrac{ut}{\sqrt{aa_t}}q_\zeta
\\[0.8em]
0&uI_{n-2}&0
\\
0&0&u\sqrt{a_t/a}
\end{pmatrix}.
\label{eq:cutoff-matrix}
\end{equation}

Consequently,
\[
\det DF^\eps_{t,j}
=
u^n(1+t\rho\zeta_\rho).
\]
\end{lemma}

\begin{proof}
Since
\[
\rho_t=e^{t\zeta(\rho,\theta)}\rho,
\]
we have
\[
\partial_\rho\rho_t
=
u(1+t\rho\zeta_\rho).
\]
Therefore
\[
DF(N)
=
u(1+t\rho\zeta_\rho)
\sqrt{\frac{a}{a_t}}N'.
\]

The horizontal and Reeb computations are the same as in
Proposition~\ref{prop:basic-matrix}, with \(\varphi\) replaced by \(\zeta\).
This gives the remaining matrix entries. The determinant follows from the
upper-triangular form.
\end{proof}

\ithead{Uniform annular bounds.}

On
\[
\mathcal A_\eps
=
\{\eps<r<2\eps\},
\]
one has
\begin{equation}
|\zeta|
+
|\rho\zeta_\rho|
+
|h_\zeta|
+
|q_\zeta|
\le C,
\label{eq:uniform-zeta}
\end{equation}
where \(C\) is independent of \(\eps\).

Indeed,
\[
\zeta=\chi_\eps\varphi_j,
\]
and the spherical derivatives of \(\varphi_j\) are uniformly bounded.
Moreover,
\[
\rho\zeta_\rho
=
\rho\chi_\eps'(r)
\frac{dr}{d\rho}
\varphi_j.
\]
Since
\[
\frac{dr}{d\rho}=\frac{1}{\sqrt{1+\rho^2}}
\]
and, on \(\mathcal A_\eps\),
\[
r\asymp\eps,
\qquad
\rho=\sinh r=r+O(r^3),
\]
we have
\[
\rho\asymp r\asymp\eps.
\]
Consequently, using
\(|\chi_\eps'|\le C/\eps\) and \(|\varphi_j|\le1\),
\[
|\rho\zeta_\rho|
\le
\frac{C\rho}{\eps\sqrt{1+\rho^2}}
\le C.
\]

\begin{lemma}
\label{lem:uniform-Jacobian}
There exist
\[
t_0>0
\quad\text{and}\quad
C>0,
\]
independent of \(\eps\), such that for \(|t|\le t_0\):

\begin{enumerate}[label=\textup{(\roman*)}]

\item
\(F^\eps_{t,j}\) is a smooth ambient diffeomorphism.

\item
On \(\mathcal A_\eps\),
\[
\|DF^\eps_{t,j}\|
+
\|(DF^\eps_{t,j})^{-1}\|
\le C.
\]

\item
For every unit vector \(\nu\), the tangential Jacobian
\[
J^\eps_\Sigma(t)
=
|\det DF^\eps_{t,j}|
\,
|(DF^\eps_{t,j})^{-T}\nu|
\]
satisfies
\begin{equation}
\left|
\partial_t^kJ^\eps_\Sigma(0)
\right|
\le C,
\qquad
k=0,1,2.
\label{eq:uniform-tangential-Jacobian}
\end{equation}

\end{enumerate}
\end{lemma}

\begin{proof}
By \eqref{eq:uniform-zeta}, every matrix entry in
\eqref{eq:cutoff-matrix}, together with its first two derivatives in \(t\)
at \(t=0\), is bounded uniformly in \(\eps\).

Choose \(t_0\) so small that
\[
1+t\rho\zeta_\rho
\ge\frac{1}{2}
\]
for every \(|t|\le t_0\). The diagonal entries are then bounded above and
away from zero.

For each fixed direction \(\theta\), the radial map
\[
\rho\longmapsto
e^{t\zeta(\rho,\theta)}\rho
\]
has positive derivative. It equals the identity near the origin and is a
positive linear scaling for large \(\rho\). Hence it is a diffeomorphism of
\([0,\infty)\), and the full map is a global diffeomorphism.

The matrix and its inverse range over a compact subset of \(GL(n)\).
The function
\[
L\longmapsto
|\det L|\,|L^{-T}\nu|
\]
has uniformly bounded first and second derivatives on this compact set.
The chain rule proves
\eqref{eq:uniform-tangential-Jacobian}.

Only
\[
\zeta,
\qquad
\rho\zeta_\rho,
\qquad
\nabla_{\mathbb S}\zeta
\]
occur. No second derivative of \(\chi_\eps\) appears in the first two
derivatives with respect to \(t\).
\end{proof}

\ithead{Convergence of the cutoff derivatives.}

Let
\[
A^\eps_j(t)
=
\Area\bigl(F^\eps_{t,j}(\Sigma)\bigr),
\qquad
V^\eps_j(t)
=
\Vol\bigl(F^\eps_{t,j}(\Omega)\bigr).
\]

Away from the center, define
\[
A_j^{(k)}(0)
=
\int_{\Sigma\setminus\{o\}}
\partial_t^kJ_\Sigma(0,\varphi_j)\,dA,
\qquad
k=1,2.
\]

Define similarly
\[
V_j^{(k)}(0)
=
\int_{\Omega\setminus\{o\}}
\left.
\partial_t^k
\right|_{t=0}
e^{nt\varphi_j}\,dV,
\qquad
k=1,2.
\]

Below we use the equivalent prime notation
\[
A_j'(0):=A_j^{(1)}(0),\qquad A_j''(0):=A_j^{(2)}(0),
\]
and likewise \(V_j'(0):=V_j^{(1)}(0)\) and
\(V_j''(0):=V_j^{(2)}(0)\).

Thus
\[
V_j'(0)
=
n\int_\Omega\varphi_j\,dV
\]
and
\[
V_j''(0)
=
n^2\int_\Omega\varphi_j^2\,dV.
\]

The values assigned to \(\varphi_j\) at \(o\) are irrelevant.

\begin{lemma}
\label{lem:cutoff-derivative-convergence}
Assume that \(o\in\Sigma\). Then, for \(k=1,2\),
\[
(A^\eps_j)^{(k)}(0)
\longrightarrow
A_j^{(k)}(0)
\]
and
\[
(V^\eps_j)^{(k)}(0)
\longrightarrow
V_j^{(k)}(0).
\]

More precisely,
\begin{equation}
\left|
(A^\eps_j)^{(k)}(0)-A_j^{(k)}(0)
\right|
\le
C\,
\Area\bigl(\Sigma\cap B_{2\eps}(o)\bigr)
=
O(\eps^{n-1}),
\label{eq:area-error-explicit}
\end{equation}
and
\begin{equation}
\left|
(V^\eps_j)^{(k)}(0)-V_j^{(k)}(0)
\right|
\le
C\,
\Vol\bigl(\Omega\cap B_{2\eps}(o)\bigr)
=
O(\eps^n).
\label{eq:volume-error-explicit}
\end{equation}
\end{lemma}

\begin{proof}
The cutoff and uncut deformations agree on
\[
M\setminus B_{2\eps}(o).
\]
Consequently, their area derivatives differ only on
\[
\Sigma\cap B_{2\eps}(o).
\]

On \(\Sigma\cap B_\eps(o)\) one has
\(\zeta_{j,\eps}=0\), so \(F^\eps_{t,j}\) is the identity there,
\(J^\eps_\Sigma(t)\equiv1\), and
\[
\partial_t^kJ^\eps_\Sigma(0)=0,
\qquad k=1,2.
\]
On the remaining transition annulus
\(\Sigma\cap(B_{2\eps}(o)\setminus B_\eps(o))\),
Lemma~\ref{lem:uniform-Jacobian} gives
\[
\left|
\partial_t^kJ^\eps_\Sigma(0)
\right|
\le C,
\qquad
k=1,2,
\]
with a constant independent of \(\eps\).

The explicit formulas
\[
J_\Sigma'(0,\varphi_j)
=
(n-1)\varphi_j+\frac{1}{2}k_{1,j}
\]
and
\[
J_\Sigma''(0,\varphi_j)
=
(n-1)^2\varphi_j^2
+
(n-1)\varphi_jk_{1,j}
+
k_{2,j}
-
\frac{1}{4}k_{1,j}^2
\]
also give
\[
\left|
\partial_t^kJ_\Sigma(0,\varphi_j)
\right|
\le C,
\qquad
k=1,2,
\]
near \(o\).

Indeed,
\[
a=1+\rho^2\ge1,
\]
the coordinate functions and their spherical first derivatives are
uniformly bounded, and
\[
|\alpha|+|v|+|\tau|\le C.
\]

Therefore
\[
\begin{aligned}
\left|
(A^\eps_j)^{(k)}(0)-A_j^{(k)}(0)
\right|
&\le
\int_{\Sigma\cap B_{2\eps}(o)}
\left(
\left|\partial_t^kJ^\eps_\Sigma(0)\right|
+
\left|\partial_t^kJ_\Sigma(0,\varphi_j)\right|
\right)dA
\\
&\le
C\,
\Area\bigl(\Sigma\cap B_{2\eps}(o)\bigr).
\end{aligned}
\]

Since \(\Sigma\) is smooth and has dimension \(n-1\),
\[
\Area\bigl(\Sigma\cap B_{2\eps}(o)\bigr)
=
O(\eps^{n-1}).
\]
This proves \eqref{eq:area-error-explicit}.

For the volume, the cutoff ambient Jacobian is
\[
\Jac F^\eps_{t,j}
=
e^{nt\zeta_{j,\eps}}
\left(
1+t\rho\partial_\rho\zeta_{j,\eps}
\right).
\]

The quantities
\[
|\zeta_{j,\eps}|
+
|\rho\partial_\rho\zeta_{j,\eps}|
\]
are uniformly bounded. Hence the first two \(t\)-derivatives of this
Jacobian at \(t=0\) are uniformly bounded.

For the uncut deformation,
\[
\Jac F_{t,\varphi_j}
=
e^{nt\varphi_j},
\]
whose first two derivatives are also uniformly bounded. The two
deformations agree outside \(B_{2\eps}(o)\). Thus
\[
\left|
(V^\eps_j)^{(k)}(0)-V_j^{(k)}(0)
\right|
\le
C\,
\Vol\bigl(\Omega\cap B_{2\eps}(o)\bigr).
\]

Smoothness of the ambient metric near \(o\) gives
\[
\Vol\bigl(\Omega\cap B_{2\eps}(o)\bigr)
=
O(\eps^n).
\]
This proves \eqref{eq:volume-error-explicit}.
\end{proof}

\subsection{The constrained variational framework}\label{sec:Hessian}

The distinction between a second-variation identity and a stability
inequality is essential.  At a critical hypersurface, the Hessian of
$A-\lambda V$ is defined for every smooth initial speed and depends only on
its normal component.  Stability, by contrast, asserts nonnegativity only
on the codimension-one subspace of zero-mean normal speeds.  The classical
framework may be found in
\cite{BarbosaDoCarmo1984,BarbosaDoCarmoEschenburg1988,Ritore2023}.

Let
\[
\Sigma=\partial\Omega
\]
be a smooth compact hypersurface, oriented by the outer unit normal
\(\nu\). We use the convention
\[
H=\divSigma\nu.
\]

For a variation with initial normal speed \(f\),
\[
A'(0)=\int_\Sigma Hf\,dA,
\qquad
V'(0)=\int_\Sigma f\,dA.
\]

\begin{lemma}
\label{lem:stationarity-implies-cmc}
Let $\Omega\Subset M$ have smooth compact boundary $\Sigma=\partial\Omega$.
If $\Omega$ is stationary for perimeter under all smooth variations that
preserve the enclosed volume, then there exists a constant
$\lambda\in\mathbb R$ such that
\[
H\equiv\lambda
\qquad\text{on }\Sigma.
\]
The same constant occurs on every connected component of $\Sigma$.
\end{lemma}

\begin{proof}
Let \(f\in C^\infty(\Sigma)\) satisfy
\[
\int_\Sigma f\,dA=0.
\]
Choose \(\psi\in C^\infty(\Sigma)\) such that
\[
\int_\Sigma \psi\,dA\neq 0,
\]
and consider the two-parameter family of normal graphs
\[
F_{t,s}(x)
=
\exp_x\bigl((tf(x)+s\psi(x))\nu(x)\bigr).
\]
For sufficiently small \((t,s)\), these are embedded hypersurfaces
enclosing a well-defined volume. Let \(V(t,s)\) denote this volume.
With the chosen orientation convention, the first variation of volume
gives
\[
\partial_tV(0,0)=\int_\Sigma f\,dA=0
\]
and
\[
\partial_sV(0,0)=\int_\Sigma\psi\,dA\neq 0.
\]
(The signs of both formulas may be reversed under the opposite
orientation convention, which does not affect the argument.)

By the implicit function theorem, there exist \(\varepsilon>0\) and a
smooth function \(s:(-\varepsilon,\varepsilon)\to\mathbb{R}\) such that
\[
s(0)=0
\qquad\text{and}\qquad
V(t,s(t))=V(0,0).
\]
Differentiating the volume constraint at \(t=0\), we obtain
\[
0
=
\frac{d}{dt}\bigg|_{t=0}V(t,s(t))
=
\partial_tV(0,0)+s'(0)\partial_sV(0,0).
\]
Since \(\partial_tV(0,0)=0\) and \(\partial_sV(0,0)\neq 0\), it follows
that
\[
s'(0)=0.
\]
Consequently, the exactly volume-preserving variation
\(t\mapsto F_{t,s(t)}\) has initial normal speed
\[
\frac{d}{dt}\bigg|_{t=0}\bigl(tf+s(t)\psi\bigr)=f.
\]

Stationarity under volume-preserving variations and the first variation
of area therefore imply
\[
\int_\Sigma Hf\,dA=0
\qquad
\text{for every } f\in C^\infty(\Sigma)
\text{ satisfying }
\int_\Sigma f\,dA=0.
\]
Define
\[
\overline H
=
\frac{1}{\operatorname{Area}(\Sigma)}
\int_\Sigma H\,dA
\]
and take
\[
f=H-\overline H.
\]
This function has zero integral over \(\Sigma\), so
\[
0
=
\int_\Sigma H(H-\overline H)\,dA
=
\int_\Sigma (H-\overline H)^2\,dA,
\]
where we used
\[
\int_\Sigma(H-\overline H)\,dA=0.
\]
It follows that
\[
H\equiv \overline H
\]
on \(\Sigma\). Since the admissible test functions are global, the same
constant occurs on every connected component of \(\Sigma\).
\end{proof}

Suppose
\[
H\equiv\lambda.
\]
Let \(\mathrm{II}_\Sigma\) denote the second fundamental form of \(\Sigma\).
Let \(A\) and \(V\) denote the area and enclosed-volume functionals. Then
\(\Sigma\) is a critical point of
\[
\mathcal L=A-\lambda V.
\]

Define the Jacobi quadratic form
\[
Q_\Sigma(f)
=
\int_\Sigma
\left(
|\nabla_\Sigma f|^2
-
\bigl(
|\mathrm{II}_\Sigma|^2+\Ric(\nu,\nu)
\bigr)f^2
\right)dA.
\]

\begin{definition}[Constrained criticality and stability]
\label{def:constrained-stability}
A smooth closed hypersurface is \emph{critical under the fixed-volume
constraint} if its mean curvature is constant. It is
\emph{stable under the fixed-volume constraint} if
\[
Q_\Sigma(f)\ge0
\qquad
\text{for every }f\in C^\infty(\Sigma)
\text{ with }\int_\Sigma f\,dA=0.
\]
The zero-mean condition is precisely the infinitesimal volume constraint.
The normal-graph correction constructed in the proof of
Lemma~\ref{lem:stationarity-implies-cmc} realizes every smooth zero-mean
function as the initial normal speed of an exactly volume-preserving
variation. Conversely, every such variation has zero-mean initial speed.
At a critical hypersurface, Lemma~\ref{lem:Hessian-property} below
identifies the corresponding second derivative of \(A-\lambda V\) with
\(Q_\Sigma(f)\).
\end{definition}

\ithead{The Hessian at a critical hypersurface.}

\begin{lemma}
\label{lem:Hessian-property}
Let
\[
F:(-\delta,\delta)\times\Sigma\longrightarrow M
\]
be a smooth variation with
\[
F_0=\iota_\Sigma,
\]
where \(\iota_\Sigma:\Sigma\hookrightarrow M\) denotes the inclusion map.
Write its initial velocity as
\[
Y
=
\left.\partial_tF_t\right|_{t=0}
=
f\nu+Y^\top.
\]

If \(H\equiv\lambda\), then
\[
\left.
\frac{d^2}{dt^2}
\right|_{t=0}
\mathcal L(F_t(\Sigma))
=
Q_\Sigma(f).
\]

In particular, the second derivative of \(A-\lambda V\) at a critical
hypersurface depends only on the initial normal speed \(f\), and not on
the tangential part \(Y^\top\) or on the acceleration of the variation.
\end{lemma}

\begin{proof}
Area and enclosed volume are invariant under reparametrization of
\(\Sigma\). Let \(\psi_t:\Sigma\to\Sigma\) be a local family of
diffeomorphisms satisfying
\[
\psi_0=\operatorname{Id}_\Sigma
\]
and
\[
\left.\partial_t\psi_t\right|_{t=0}
=
-Y^\top.
\]

Set
\[
\widetilde F_t
=
F_t\circ\psi_t.
\]
Then
\[
\widetilde F_t(\Sigma)=F_t(\Sigma)
\]
as unparametrized hypersurfaces, so
\[
\mathcal L(\widetilde F_t(\Sigma))
=
\mathcal L(F_t(\Sigma))
\]
for every sufficiently small \(t\). Moreover,
\[
\left.
\partial_t\widetilde F_t
\right|_{t=0}
=
Y-Y^\top
=
f\nu.
\]

It therefore suffices to consider a variation whose initial velocity is
purely normal.

{On the local space of unparametrized hypersurfaces, use the
normal-graph chart supplied by a fixed tubular neighborhood of \(\Sigma\).
For any \(C^2\) path \(u(t)\) in this chart,}
\[
\left.
\frac{d^2}{dt^2}
\right|_{t=0}
\mathcal L(u(t))
=
D^2\mathcal L_{\iota_\Sigma}
\bigl[u'(0),u'(0)\bigr]
+
 D\mathcal L_{\iota_\Sigma}
 \bigl[u''(0)\bigr].
\]
{This is the ordinary second-order chain rule in a Banach chart;
no geometric quantity depends on a chosen parametrization.}

Since
\[
H-\lambda=0,
\]
the first variation of \(\mathcal L\) vanishes in every direction:
\[
D\mathcal L_{\iota_\Sigma}=0.
\]
Hence the acceleration term vanishes, and the second derivative depends
only on the initial tangent vector \(f\nu\).

We may therefore replace the original path by the normal graph variation
\[
G_t(x)
=
\exp_x\bigl(tf(x)\nu(x)\bigr).
\]

For this normal variation, the standard evolution formula is
\[
\left.
\frac{d}{dt}
\right|_{t=0}
H_t
=
-\Delta_\Sigma f
-
\bigl(
|\mathrm{II}_\Sigma|^2+\Ric(\nu,\nu)
\bigr)f.
\]
See, for example, \cite[\S 1.3.4, pp.~27--30]{Ritore2023}, after translating
to the convention \(H=\divSigma\nu\).

Because \(H-\lambda=0\) at \(t=0\), differentiating
\[
\mathcal L'(t)
=
\int_{\Sigma_t}
(H_t-\lambda)f_t\,dA_t
\]
at \(t=0\) gives
\[
\mathcal L''(0)
=
\int_\Sigma
f
\left[
-\Delta_\Sigma f
-
\bigl(
|\mathrm{II}_\Sigma|^2+\Ric(\nu,\nu)
\bigr)f
\right]dA.
\]

Integration by parts on the closed hypersurface yields
\[
\mathcal L''(0)
=
\int_\Sigma
\left(
|\nabla_\Sigma f|^2
-
\bigl(
|\mathrm{II}_\Sigma|^2+\Ric(\nu,\nu)
\bigr)f^2
\right)dA.
\]

This is \(Q_\Sigma(f)\).
\end{proof}

\begin{remark}
Lemma~\ref{lem:Hessian-property} is an identity for every smooth normal
speed \(f\). Stability is a separate assertion: it gives
\[
Q_\Sigma(f)\ge0
\]
only when
\[
\int_\Sigma f\,dA=0.
\]
\end{remark}
Consequently, every smooth isoperimetric boundary is critical and stable
under the fixed-volume constraint. Indeed,
Lemma~\ref{lem:stationarity-implies-cmc} gives constant mean curvature,
while the normal-graph correction used in its proof realizes every smooth
zero-mean function \(f\) as the initial normal speed of an exactly
volume-preserving variation. Minimality and
Lemma~\ref{lem:Hessian-property} then give \(Q_\Sigma(f)\ge0\).

\ithead{Continuity of the Jacobi form.}

\begin{lemma}
\label{lem:Q-continuity}
The quadratic form \(Q_\Sigma\) extends continuously to \(H^1(\Sigma)\).
More precisely, there exists \(C_\Sigma>0\) such that
\[
\left|
Q_\Sigma(f)-Q_\Sigma(g)
\right|
\le
C_\Sigma
\|f-g\|_{H^1(\Sigma)}
\left(
\|f\|_{H^1(\Sigma)}
+
\|g\|_{H^1(\Sigma)}
\right)
\]
for all \(f,g\in H^1(\Sigma)\).
\end{lemma}

\begin{proof}
Since \(\Sigma\) is compact and smooth,
\[
\mathcal V
=
|\mathrm{II}_\Sigma|^2+\Ric(\nu,\nu)
\]
is bounded. Therefore
\[
\begin{aligned}
Q_\Sigma(f)-Q_\Sigma(g)
={}&
\int_\Sigma
\langle\nabla(f-g),\nabla(f+g)\rangle\,dA
\\
&-
\int_\Sigma
\mathcal V(f-g)(f+g)\,dA.
\end{aligned}
\]
The estimate follows from Cauchy--Schwarz and the boundedness of
\(\mathcal V\).
\end{proof}

\subsection{The boundary-center cutoff--Hessian bridge}\label{sec:boundary-Hessian-bridge}

Assume in this section that
\[
o\in\Sigma.
\]

Recall
\[
X_j
=
\varphi_j\rho\partial_\rho,
\qquad
f_j
=
\langle X_j,\nu\rangle.
\]

Let
\[
f_{j,\eps}
=
\chi_\eps f_j,
\]
and choose
\[
\psi\in C_c^\infty(\Sigma\setminus\{o\}),
\qquad
\int_\Sigma\psi\,dA=1.
\]

Define
\[
m_{j,\eps}
=
\int_\Sigma f_{j,\eps}\,dA
\]
and
\begin{equation}
g_{j,\eps}
=
f_{j,\eps}-m_{j,\eps}\psi.
\label{eq:mean-corrected-speed-rigorous}
\end{equation}

\begin{theorem}[Boundary-center cutoff--Hessian bridge]
\label{thm:boundary-Hessian-bridge}
Let \(m\ge2\) and let \(\Omega\Subset\CH^m\) have smooth compact boundary
\(\Sigma=\partial\Omega\). Suppose that \(o\in\Sigma\) is a volume geometric
median and that \(H\equiv\lambda\). Then:

\begin{enumerate}[label=\textup{(\roman*)}]

\item
\(f_j\in H^1(\Sigma)\) and
\[
f_{j,\eps}\longrightarrow f_j
\quad\text{strongly in }H^1(\Sigma).
\]

\item
The limiting speed has zero mean:
\[
\int_\Sigma f_j\,dA=0.
\]

\item
The corrected functions satisfy
\[
\int_\Sigma g_{j,\eps}\,dA=0
\]
and
\[
g_{j,\eps}\longrightarrow f_j
\quad\text{strongly in }H^1(\Sigma).
\]

\item
The punctured polar second derivative is identified with the Jacobi form:
\[
A_j''(0)-\lambda V_j''(0)
=
Q_\Sigma(f_j).
\]

\item
If \(\Sigma\) is stable under the fixed-volume constraint, then
\begin{equation}
Q_\Sigma(f_j)\ge0.
\label{eq:boundary-stability-rigorous}
\end{equation}

\end{enumerate}
\end{theorem}

\begin{proof}
Statement (i) is \eqref{eq:fjeps-H1}, and (ii) follows from
Lemma~\ref{lem:boundary-zero-mean}.  The same calculation as in
\eqref{eq:gjeps}--\eqref{eq:gjeps-H1} shows that the functions defined in
\eqref{eq:mean-corrected-speed-rigorous} have zero mean and converge strongly
to \(f_j\) in \(H^1(\Sigma)\), proving (iii).

The smooth cutoff ambient deformation
\[
F^\eps_{t,j}
\]
has initial normal speed \(f_{j,\eps}\). By
Lemma~\ref{lem:Hessian-property},
\begin{equation}
(A^\eps_j)''(0)
-
\lambda(V^\eps_j)''(0)
=
Q_\Sigma(f_{j,\eps}).
\label{eq:cutoff-Hessian-identity-rigorous}
\end{equation}

Lemma~\ref{lem:cutoff-derivative-convergence} gives
\[
(A^\eps_j)''(0)\to A_j''(0),
\qquad
(V^\eps_j)''(0)\to V_j''(0).
\]

Lemma~\ref{lem:Q-continuity} and (i) give
\[
Q_\Sigma(f_{j,\eps})
\longrightarrow
Q_\Sigma(f_j).
\]

Passing to the limit in
\eqref{eq:cutoff-Hessian-identity-rigorous} proves (iv).

Finally, if \(\Sigma\) is stable, every smooth zero-mean function is an
admissible first-order volume-preserving speed. Thus
\[
Q_\Sigma(g_{j,\eps})\ge0.
\]

By (iii) and Lemma~\ref{lem:Q-continuity},
\[
Q_\Sigma(g_{j,\eps})
\longrightarrow
Q_\Sigma(f_j).
\]
This proves (v).
\end{proof}

\begin{remark}
The two approximation sequences have different roles:
\[
f_{j,\eps}
\]
is the normal speed of the explicit cutoff ambient deformation and is used
to identify the Hessian, whereas
\[
g_{j,\eps}
\]
has exactly zero mean and is used to invoke stability.
\end{remark}

\section{The integrated trace identity and smooth rigidity}\label{sec:integrated-trace}
\subsection{The integrated constrained trace identity}

This section is the sign-changing heart of the proof.  Stability gives
$Q_\Sigma(f_j)\ge0$ for each balanced coordinate speed.  The geometric
calculation gives a formula for their sum as the negative integral of a
nonnegative function.  The only way both statements can hold is that the
nonradial part of the normal vanishes identically.  The uniform radial
variation appears only as a bookkeeping device that cancels the second
volume derivatives.

For the coordinate deformation generated by \(\varphi_j\), denote its
area and volume derivatives by
\[
A_j'(0),
\qquad
A_j''(0),
\qquad
V_j'(0),
\qquad
V_j''(0).
\]

For the uniform radial deformation
\[
F_{t,\mathrm{rad}}(\rho,\theta)
=
(e^t\rho,\theta),
\]
write
\[
A_{\mathrm{rad}}'(0),
\qquad
V_{\mathrm{rad}}'(0).
\]

When \(o\in\Sigma\), the derivatives \(A_j^{(k)}(0)\) and
\(V_j^{(k)}(0)\), \(k=1,2\), are understood in the punctured sense
introduced before Lemma~\ref{lem:cutoff-derivative-convergence}. The
uniform radial deformation itself extends smoothly through \(o\). Indeed,
in geodesic normal coordinates \(z=r\theta\),
\[
F_{t,\mathrm{rad}}(z)
=
\frac{\operatorname{arsinh}(e^t\sinh |z|)}{|z|}\,z
\qquad (z\ne0),
\]
with \(F_{t,\mathrm{rad}}(0)=0\). The scalar quotient is an even
real-analytic function jointly of \(t\) and \(|z|^2\), with value \(e^t\)
at the origin.
Thus \(A_{\mathrm{rad}}'(0)\) and \(V_{\mathrm{rad}}'(0)\) are ordinary
first variations for every possible location of \(o\).

\ithead{The volume cancellation.}

Because
\[
\Jac F_{t,\varphi_j}
=
e^{nt\varphi_j},
\]
one has
\[
V_j''(0)
=
n^2\int_\Omega\varphi_j^2\,dV.
\]

Since
\[
\sum_{j=1}^n\varphi_j^2=1,
\]
it follows that
\[
\sum_{j=1}^nV_j''(0)
=
n^2|\Omega|.
\]

For the uniform radial deformation,
\[
\Jac F_{t,\mathrm{rad}}=e^{nt},
\]
and therefore
\[
V_{\mathrm{rad}}'(0)
=
n|\Omega|.
\]

Since \(H\equiv\lambda\),
\[
A_{\mathrm{rad}}'(0)
=
\lambda V_{\mathrm{rad}}'(0)
=
\lambda n|\Omega|.
\]

Consequently,
\begin{equation}
\lambda
\sum_{j=1}^nV_j''(0)
=
nA_{\mathrm{rad}}'(0).
\label{eq:volume-cancellation}
\end{equation}

\begin{remark}[Why the volume terms cancel]
The coordinate deformations are not individually volume preserving to
second order.  Their summed second volume variation is nevertheless
explicit because $\sum_j\varphi_j^2=1$.  The constant-mean-curvature
identity for the uniform radial deformation converts this summed volume
term into the radial area term already present in the pointwise trace
formula.  This exact cancellation is what allows the final formula to be
expressed solely through the Jacobi forms and the nonnegative polynomial.
\end{remark}

\begin{theorem}[Integrated constrained trace identity]
\label{thm:integrated-trace}
Let \(m\ge 2\), and let \(\Omega\Subset\CH^m\) be a domain with smooth
compact boundary \(\Sigma\).  Let \(o\) be a volume geometric median of
\(\Omega\), which may lie in \(\Omega\), on \(\Sigma\), or in
\(\CH^m\setminus\overline{\Omega}\). Suppose that
\(H\equiv\lambda\). Then
\begin{equation}
\sum_{j=1}^n Q_\Sigma(f_j)
=
-\int_\Sigma
\frac{\mathcal R_{n,x}(\beta,\gamma)}{(1+x)^2}\,dA.
\label{eq:integrated-trace}
\end{equation}
When \(o\in\Sigma\), the polar quantities in the integrand are understood
on \(\Sigma\setminus\{o\}\), and their values at \(o\) are immaterial.
\end{theorem}

\begin{proof}
We distinguish the three possible locations of the median.

\medskip
\noindent
\textbf{Case 1: \(o\notin\overline\Omega\).}

The polar coordinate deformations are smooth on a neighborhood of
\(\overline\Omega\). Lemma~\ref{lem:Hessian-property} applies directly and
gives
\[
Q_\Sigma(f_j)
=
A_j''(0)-\lambda V_j''(0).
\]

\medskip
\noindent
\textbf{Case 2: \(o\in\Omega\).}

The boundary is separated from \(o\). Apply
Lemma~\ref{lem:interior-cutoff}. It gives
\[
Q_\Sigma(f_j)
=
A_j''(0)-\lambda V_j''(0).
\]

\medskip
\noindent
\textbf{Case 3: \(o\in\Sigma\).}

The boundary-center cutoff--Hessian bridge,
Theorem~\ref{thm:boundary-Hessian-bridge}, gives
\[
Q_\Sigma(f_j)
=
A_j''(0)-\lambda V_j''(0).
\]

Thus, in every case,
\begin{equation}
Q_\Sigma(f_j)
=
A_j''(0)-\lambda V_j''(0).
\label{eq:Q-area-volume-all-cases}
\end{equation}

We now integrate the pointwise trace identity
\[
\sum_{j=1}^n
J_\Sigma''(0,\varphi_j)
-
nJ_\Sigma'(0,1)
=
-
\frac{
\mathcal R_{n,x}(\beta,\gamma)
}{
(1+x)^2
}
\]
If \(o\notin\Sigma\), this is a direct integration over \(\Sigma\). If
\(o\in\Sigma\), the identity holds on \(\Sigma\setminus\{o\}\). The
explicit formulas \eqref{eq:J-second} and \eqref{eq:J-radial-first} show
that the densities on the left are bounded near \(o\), as also used in the
proof of Lemma~\ref{lem:cutoff-derivative-convergence}. The density on the
right is bounded there because \(x\to0\) and
\[
0\le\beta,\gamma,
\qquad
\beta+\gamma\le1.
\]
Thus all terms are integrable, and the omission of the single point does
not change the integral. In either case we obtain
\begin{equation}
\sum_{j=1}^nA_j''(0)
-
nA_{\mathrm{rad}}'(0)
=
-
\int_\Sigma
\frac{
\mathcal R_{n,x}(\beta,\gamma)
}{
(1+x)^2
}\,dA.
\label{eq:integrated-area-trace-rigorous}
\end{equation}

Summing \eqref{eq:Q-area-volume-all-cases} over \(j\) gives
\[
\sum_{j=1}^nQ_\Sigma(f_j)
=
\sum_{j=1}^nA_j''(0)
-
\lambda\sum_{j=1}^nV_j''(0).
\]

Using the volume cancellation
\eqref{eq:volume-cancellation},
\[
\sum_{j=1}^nQ_\Sigma(f_j)
=
\sum_{j=1}^nA_j''(0)
-
nA_{\mathrm{rad}}'(0).
\]

Equation \eqref{eq:integrated-area-trace-rigorous} now proves
\eqref{eq:integrated-trace}.
\end{proof}

\begin{corollary}
\label{cor:stability-radial-normal}
Let \(m\ge2\), let \(\Omega\Subset\CH^m\) have smooth compact boundary
\(\Sigma=\partial\Omega\), and let \(o\) be a volume geometric median of
\(\Omega\).
Suppose that \(H\equiv\lambda\) and that \(\Sigma\) is stable under the
fixed-volume constraint. Then
\[
\beta=\gamma=0
\qquad
\text{on }\Sigma\setminus\{o\}.
\]
Equivalently,
\[
\nu=\pm N
\qquad
\text{on }\Sigma\setminus\{o\}.
\]
\end{corollary}

\begin{proof}
Equation~\eqref{eq:fj-mean-zero-direct},
Lemma~\ref{lem:interior-cutoff}, and
Lemma~\ref{lem:boundary-zero-mean} show, respectively in the three possible
locations of the median, that
\[
\int_\Sigma f_j\,dA=0.
\]
When \(o\notin\Sigma\), the speeds are smooth on \(\Sigma\), and stability
gives
\[
Q_\Sigma(f_j)\ge0
\]
directly. When \(o\in\Sigma\), the same inequality is precisely
\eqref{eq:boundary-stability-rigorous} in
Theorem~\ref{thm:boundary-Hessian-bridge}. Thus, in every case,
\[
0
\le
\sum_{j=1}^nQ_\Sigma(f_j).
\]

The integrated trace identity, which now includes the
boundary-center case, gives
\[
\sum_{j=1}^nQ_\Sigma(f_j)
=
-
\int_\Sigma
\frac{
\mathcal R_{n,x}(\beta,\gamma)
}{
(1+x)^2
}\,dA
\le0.
\]

Therefore equality holds. Since
\[
\mathcal R_{n,x}(\beta,\gamma)\ge0,
\]
the integrand vanishes almost everywhere. The strict positivity theorem
then gives
\[
\beta=\gamma=0
\]
almost everywhere on \(\Sigma\setminus\{o\}\).

Smoothness of \(\Sigma\) and continuity of the normal decomposition upgrade
this to pointwise vanishing away from \(o\).
\end{proof}

\subsection{Rigidity of smooth stable critical domains}

The preceding trace argument shows that stability forces the normal to be
radial.  It remains only to identify the resulting boundary components and
to use their common oriented mean curvature to exclude all but one.

\begin{theorem}[Rigidity of stable critical domains]
\label{thm:smooth-rigidity}
Let
\[
m\ge2
\]
and let
\[
\Omega\Subset\CH^m
\]
have smooth compact boundary \(\Sigma\). Assume that \(\Sigma\) is critical
and stable for perimeter under a fixed-volume constraint.

Then \(\Omega\) is a geodesic ball.
\end{theorem}

\begin{proof}
Let \(o\) be a volume geometric median. By
Corollary~\ref{cor:stability-radial-normal},
\[
\beta=\gamma=0
\]
on \(\Sigma\setminus\{o\}\), and hence
\[
\nu=\pm N.
\]

For every tangent vector \(Y\in T\Sigma\),
\[
Y(r)
=
\langle\nabla r,Y\rangle
=
\langle N,Y\rangle
=
0.
\]
Thus \(r\) is constant on every connected component of
\(\Sigma\setminus\{o\}\).

We next rule out
\[
o\in\Sigma.
\]

Let \(\Sigma_0\) be the connected component of \(\Sigma\) containing \(o\).
A connected smooth manifold of dimension at least two remains connected
after removal of one point. Here
\[
\dim\Sigma_0=n-1\ge3.
\]
Therefore
\[
\Sigma_0\setminus\{o\}
\]
is connected.

The function \(r\) is constant on
\(\Sigma_0\setminus\{o\}\); write
\[
r\equiv c.
\]
Since points of \(\Sigma_0\setminus\{o\}\) approach \(o\), continuity gives
\[
c=0.
\]
This is impossible away from \(o\). Hence
\[
o\notin\Sigma.
\]

It follows that \(r\) is constant on every connected component of
\(\Sigma\). Therefore every connected component of \(\Sigma\) is a geodesic
sphere centered at \(o\).  More explicitly, if a component \(\Gamma\) lies
in \(S_c(o)\), then
\[
T_p\Gamma=T_pS_c(o)
\qquad(p\in\Gamma).
\]
Thus the inclusion \(\Gamma\hookrightarrow S_c(o)\) is open.  Since
\(\Gamma\) is compact it is also closed, and connectedness of the geodesic
sphere gives \(\Gamma=S_c(o)\).

It remains to exclude multiple concentric components. The area density of a
geodesic sphere is proportional to
\[
\sinh^{n-1}r\,\cosh r.
\]
Its scalar mean curvature with outward radial normal is
\[
H_0(r)
=
\frac{d}{dr}
\log\bigl(\sinh^{n-1}r\cosh r\bigr)
=
(n-1)\coth r+\tanh r
>0.
\]

Equivalently,
\[
H_0(r)
=
(n-2)\coth r+2\coth(2r).
\]

Order the distinct boundary radii as
\[
0<r_1<\cdots<r_k.
\]

The outermost component has outer normal \(+N\) and oriented mean curvature
\[
H_0(r_k)>0.
\]

If \(k\ge2\), the next component inward has outer normal \(-N\), and hence
oriented mean curvature
\[
-H_0(r_{k-1})<0.
\]
Indeed, the indicator of \(\Omega\) is constant on each connected annulus
between successive radii.  Crossing the outermost sphere from outside enters
\(\Omega\), so crossing the next sphere inward leaves \(\Omega\); its outer
normal is therefore \(-N\).  The orientations alternate at all successive
boundary spheres.

A volume-constrained critical boundary has one constant oriented mean
curvature \(\lambda\) on every component. The opposite signs are impossible.
Thus
\[
k=1.
\]

Since \(\Omega\) is bounded, it is the ball enclosed by the unique boundary
sphere.
\end{proof}

This proves Theorem~\ref{thm:intro-main}.

\section{Sharp analytic consequences}\label{sec:analytic-consequences}

We record several consequences of the isoperimetric classification in
$\CH^m$, \(m\ge2\).  These statements were previously known under the
assumption that geodesic balls are the isoperimetric regions of the complex
hyperbolic ball; see \cite{LiSu2025,Singh2026}.  Corollary~\ref{thm:gromov-ros-all-m} discharges that geometric assumption
in every complex dimension at least two.  The deductions below are therefore short but rigorous: we invoke
the conditional analytic theorems of Li and Su \cite{LiSu2025} and Singh
\cite{Singh2026}, and verify that their only unresolved geometric hypothesis
is exactly the isoperimetric classification proved in
Theorem~\ref{thm:gromov-ros-all-m}.
The resulting convex Bergman inequality also verifies the analytic
hypothesis in Melentijevi\'c's higher-dimensional stability theorem
\cite[Theorem~4]{Melentijevic2025}.
We also record the sharp constants and
equality cases in the normalization used here.

\ithead{The one-dimensional precursor.}
Kulikov proved in the unit disk that certain monotone functionals on Hardy
spaces and convex functionals on weighted Bergman spaces attain their
maximum, among functions of norm one, at normalized reproducing kernels
\cite{Kulikov2022}.  His theorem simultaneously yields sharp contractive
embeddings and the $SU(1,1)$ Wehrl-type entropy inequality
\cite{LiebSolovej2021}. It is therefore
the natural one-dimensional model for the results below.

The higher-dimensional difficulty is geometric.  The distribution-function
argument requires the sharp isoperimetric inequality for the invariant
metric on $\mathbb B^m$. Li and Su \cite{LiSu2025} proved that this
complex-hyperbolic isoperimetric statement implies the corresponding
Kulikov-type convex contraction principle in $\mathbb B^m$. Nicola,
Riccardi, and Tilli formulated the higher-dimensional convex Bergman
statement as Conjecture~5.1
\cite[Section~5]{NicolaRiccardiTilli2025}, while Singh \cite{Singh2026}
used the same geometric input for the $SU(N,1)$ Lieb--Wehrl and Bergman
Faber--Krahn inequalities.  Melentijevi\'c
\cite[Theorem~4]{Melentijevic2025} proved a higher-dimensional quadratic
stability theorem conditional on this convex Bergman inequality.
Theorem~\ref{thm:gromov-ros-all-m}
supplies the missing isoperimetric input for every \(m\ge2\), and hence also
makes that stability theorem unconditional.
Consequently, the theorems in this section are
unconditional higher-dimensional extensions of the one-dimensional
reproducing-kernel extremal principle.

{
Kalaj and Zhu \cite{KalajZhu2026} recently obtained a local quantitative
contraction theorem near the constant extremizer, under near-sphericity
hypotheses on the weighted level sets.  That stability result is
complementary to the global sharp consequences below and is not used in
their proof.
}

\subsection{Invariant measure and weighted Bergman spaces}

Let $m\ge2$ and let
\[
\mathbb B^m=\{z\in\mathbb C^m:|z|<1\}.
\]
Let $dv$ be Euclidean volume normalized by $v(\mathbb B^m)=1$, and let
$d\sigma$ be normalized surface measure on $\partial\mathbb B^m$.  We use
the invariant measure
\begin{equation}
 d\mu_m(z)=\frac{dv(z)}{(1-|z|^2)^{m+1}}.
 \label{eq:invariant-ball-measure}
\end{equation}
For the metric \eqref{eq:normalized-Bergman-metric},
\[
 dV_g(z)=\frac{d\lambda_{2m}(z)}{(1-|z|^2)^{m+1}},
 \qquad
 d\mu_m=\frac{m!}{\pi^m}\,dV_g.
\]
Thus the invariant measure in this section differs from Riemannian volume
only by a fixed dimensional constant.

\begin{remark}[Normalization and hypothesis transfer]
\label{rem:analytic-transfer}
The conditional results quoted below use the same unit-ball model and
invariant measure, up to the fixed normalization displayed above.  In the
notation of Li and Su \cite{LiSu2025}, their complex dimension is our \(m\),
and their parameters \(p\) and \(\alpha\) are unchanged.  In Singh's
notation \cite{Singh2026}, \(N=m\) and \(\lambda=(m+1)k\).  Under this
dictionary, the unresolved geometric hypothesis in those results is exactly
that finite-volume isoperimetric regions for the complex-hyperbolic metric
are geodesic balls.  Theorem~\ref{thm:gromov-ros-all-m} supplies this
hypothesis.  Multiplying volume and perimeter by fixed positive constants
does not change the minimizing sets or their equality classification; the
equality statements used below are the equality statements in the cited
theorems under this same dictionary.
\end{remark}

For $0<p<\infty$ and $\alpha>m$, define the weighted Bergman space
$\mathcal A^p_\alpha(\mathbb B^m)$ by
\begin{equation}
 \|f\|_{\mathcal A^p_\alpha}^{p}
 =c_{\alpha,m}\int_{\mathbb B^m}
 |f(z)|^p(1-|z|^2)^\alpha\,d\mu_m(z),
 \qquad
 c_{\alpha,m}=\frac{\Gamma(\alpha)}{m!\,\Gamma(\alpha-m)}.
 \label{eq:weighted-Bergman-norm}
\end{equation}
The constant is chosen so that $\|1\|_{\mathcal A^p_\alpha}=1$.  We also
write
\[
 \|f\|_{H^s}^{s}
 =\sup_{0<r<1}\int_{\partial\mathbb B^m}|f(r\zeta)|^s\,d\sigma(\zeta)
\]
for the Hardy quasi-norm.  Our Hermitian product is
$\langle z,w\rangle=\sum_{j=1}^m z_j\overline{w_j}$.

For a normalized $f\in\mathcal A^p_\alpha$, the standard point-evaluation
estimate \cite[(1.1)]{LiSu2025} gives
\[
 u_f(z):=|f(z)|^p(1-|z|^2)^\alpha\le1.
\]
Thus convex functions below need only be defined on $[0,1]$.

\subsection{The convex Bergman contraction principle}

The following is the higher-dimensional Kulikov-type contraction conjecture
discussed by Nicola, Riccardi, and Tilli
\cite[Section~5]{NicolaRiccardiTilli2025}.  Kulikov proved the corresponding
one-dimensional reproducing-kernel extremal principle in $\mathbb B^1$
\cite{Kulikov2022}.  For every \(m\ge2\), the higher-dimensional
statement is a theorem.

\begin{theorem}[Sharp convex Bergman contraction]
\label{thm:convex-Bergman-contraction}
Let $m\ge2$, $0<p<\infty$, and $\alpha>m$.  Let
$f\in\mathcal A^p_\alpha(\mathbb B^m)$ satisfy
$\|f\|_{\mathcal A^p_\alpha}=1$.  If
$\Phi:[0,1]\to\mathbb R$ is convex, $\Phi(0)=0$, and the right-hand side
below is finite, then
\begin{equation}
 \int_{\mathbb B^m}
 \Phi\!\left(|f(z)|^p(1-|z|^2)^\alpha\right)d\mu_m(z)
 \le
 \int_{\mathbb B^m}
 \Phi\!\left((1-|z|^2)^\alpha\right)d\mu_m(z).
 \label{eq:convex-Bergman-contraction}
\end{equation}
For every $w\in\mathbb B^m$ and $\theta\in\mathbb R$, equality is attained
by
\begin{equation}
 f_{w,\theta}(z)
 =e^{i\theta}
 \frac{(1-|w|^2)^{\alpha/p}}{(1-\langle z,w\rangle)^{2\alpha/p}}.
 \label{eq:Bergman-extremizer}
\end{equation}
If $\Phi$ is strictly convex, then these are the only equality cases.
\end{theorem}

\begin{proof}
Under the dictionary in Remark~\ref{rem:analytic-transfer}, Li and Su prove
\eqref{eq:convex-Bergman-contraction} under the assumption
that isoperimetric regions in the complex hyperbolic ball are geodesic
balls; see \cite[Theorem~1.7]{LiSu2025}.
Theorem~\ref{thm:gromov-ros-all-m} establishes precisely that assumption
for every \(m\ge2\). Their distribution-function
and rearrangement argument therefore applies without an additional
hypothesis and proves the inequality.

For \(w\ne0\), set
\[
P_wz=\frac{\langle z,w\rangle}{|w|^2}w,
\qquad
Q_w=I-P_w,
\qquad
s_w=\sqrt{1-|w|^2},
\]
and define the standard involutive ball automorphism
\[
\phi_w(z)
=
\frac{w-P_wz-s_wQ_wz}{1-\langle z,w\rangle};
\qquad
\phi_0(z)=-z.
\]
The automorphism identity
\[
 1-|\phi_w(z)|^2
 =\frac{(1-|w|^2)(1-|z|^2)}{|1-\langle z,w\rangle|^2}
\]
shows that the functions in \eqref{eq:Bergman-extremizer} are isometric
images of the constant function $1$.  Since $d\mu_m$ is invariant under
$\phi_w$, both sides of \eqref{eq:convex-Bergman-contraction} are unchanged,
so these functions attain equality.  Conversely, once the convex inequality
is known for the full admissible class of convex functions in a given
dimension, the equality theorem of Nicola, Riccardi, and Tilli
\cite[Theorem~5.2]{NicolaRiccardiTilli2025} shows that strict convexity
forces exactly the family \eqref{eq:Bergman-extremizer}.  Li and Su's theorem
supplies precisely that full class.  The
complex power is defined by the holomorphic logarithm of
$1-\langle z,w\rangle$, whose real part is positive on $\mathbb B^m$.
For the standard ball automorphisms and their basic identities, see
\cite[Chapter~2]{Zhu2005}.
\end{proof}

\begin{remark}[Quantitative stability of convex Bergman functionals]
Melentijevi\'c \cite[Theorem~4]{Melentijevic2025} established a
higher-dimensional quadratic stability estimate conditional on the full
convex inequality \eqref{eq:convex-Bergman-contraction}: the deficit in that
inequality controls the squared weighted-Bergman distance from the family of
normalized reproducing kernels.  Theorem~\ref{thm:convex-Bergman-contraction}
verifies his hypothesis for every complex dimension $m\ge2$.  Thus his
higher-dimensional stability conclusion now holds unconditionally (after
the parameter translation in Remark~\ref{rem:analytic-transfer}).
\end{remark}

\begin{corollary}
\label{cor:sharp-Bergman-contractivity}
Let $m\ge2$ and suppose
\[
 0<p<q<\infty,
 \qquad
 m<\alpha<\beta<\infty,
 \qquad
 \frac{p}{\alpha}=\frac{q}{\beta}.
\]
Then every holomorphic function on $\mathbb B^m$ satisfies
\begin{equation}
 \|f\|_{\mathcal A^q_\beta}
 \le
 \|f\|_{\mathcal A^p_\alpha}.
 \label{eq:sharp-Bergman-contractivity}
\end{equation}
For nonzero $f$, equality holds if and only if $f$ is a constant multiple of
one of the functions in \eqref{eq:Bergman-extremizer}.
\end{corollary}

\begin{proof}
The result is trivial when $f=0$.  Otherwise normalize
$g=f/\|f\|_{\mathcal A^p_\alpha}$ and apply
Theorem~\ref{thm:convex-Bergman-contraction} with
\[
 \Phi(t)=t^{q/p}.
\]
Since $q/p=\beta/\alpha$, we obtain
\[
 \int_{\mathbb B^m}|g|^q(1-|z|^2)^\beta\,d\mu_m
 \le
 \int_{\mathbb B^m}(1-|z|^2)^\beta\,d\mu_m
 =\frac{1}{c_{\beta,m}}.
\]
Multiplication by $c_{\beta,m}$ gives
$\|g\|_{\mathcal A^q_\beta}\le1$, which is
\eqref{eq:sharp-Bergman-contractivity}.  Since $q/p>1$, the function
$\Phi$ is strictly convex, and the equality statement follows from
Theorem~\ref{thm:convex-Bergman-contraction}.
\end{proof}

\begin{corollary}
\label{cor:Hardy-Bergman-contractivity}
Let $m\ge2$, $r>0$, and $\alpha>m$.  Then
\[
 \|f\|_{\mathcal A^{\alpha r}_\alpha}
 \le
 \|f\|_{H^{mr}}
\]
for every holomorphic $f$ for which the right-hand side is finite.  Constant
functions attain equality.
\end{corollary}

\begin{proof}
This is the Hardy endpoint of the contractive chain of Li and Su
\cite[Corollary~1.8]{LiSu2025}. Its sole geometric hypothesis is the
classification of isoperimetric regions as geodesic balls, supplied by
Theorem~\ref{thm:gromov-ros-all-m}; the normalizations are matched in
Remark~\ref{rem:analytic-transfer}.
\end{proof}

\subsection{The Gaussian--Fock scaling limit}

Define the probability measure
\begin{equation}
 d\nu_{\alpha,m}(z)
 =
 c_{\alpha,m}(1-|z|^2)^\alpha\,d\mu_m(z),
 \qquad \alpha>m.
 \label{eq:weighted-Bergman-probability-measure}
\end{equation}
Let \(S_\alpha(z)=\sqrt{\alpha}\,z\), let
\(d\lambda_{2m}\) denote Lebesgue measure on \(\mathbb C^m\), and put
\[
 d\gamma_m(\xi)
 =
 \pi^{-m}e^{-|\xi|^2}\,d\lambda_{2m}(\xi).
\]
Thus \(\gamma_m\) is the standard complex Gaussian probability measure.
Under the identification \(\mathbb C^m\simeq\mathbb R^{2m}\), its real
covariance is \(\tfrac12 I_{2m}\).  Hence a real standard Gaussian vector in \(\mathbb R^{2m}\) has the same law as
\(\sqrt2\,\xi\) when \(\xi\sim\gamma_m\).

\begin{proposition}
\label{prop:Bergman-to-Gaussian-scaling}
The pushforward measures
\[
 \widetilde\nu_{\alpha,m}
 =
 (S_\alpha)_\#\nu_{\alpha,m}
\]
have densities
\begin{equation}
 \frac{d\widetilde\nu_{\alpha,m}}{d\lambda_{2m}}(\xi)
 =
 \frac{1}{\pi^m}
 \frac{\Gamma(\alpha)}{\Gamma(\alpha-m)\alpha^m}
 \begin{cases}
 \left(1-\dfrac{|\xi|^2}{\alpha}\right)^{\alpha-m-1},
 & |\xi|<\sqrt\alpha,\\[4pt]
 0,& |\xi|\ge\sqrt\alpha.
 \end{cases}
 \label{eq:scaled-Bergman-density}
\end{equation}
As \(\alpha\to\infty\), these densities converge in
\(L^1(\mathbb C^m)\) to the density of \(\gamma_m\).  Equivalently,
\[
 \widetilde\nu_{\alpha,m}
 \longrightarrow
 \gamma_m
 \qquad\text{in total variation}.
\]
\end{proposition}

\begin{proof}
Since \(dv=(m!/\pi^m)\,d\lambda_{2m}\), definitions
\eqref{eq:invariant-ball-measure},
\eqref{eq:weighted-Bergman-norm}, and
\eqref{eq:weighted-Bergman-probability-measure} give
\[
 d\nu_{\alpha,m}(z)
 =
 \frac{1}{\pi^m}
 \frac{\Gamma(\alpha)}{\Gamma(\alpha-m)}
 (1-|z|^2)^{\alpha-m-1}\,d\lambda_{2m}(z).
\]
The change of variables \(\xi=\sqrt{\alpha}\,z\) proves
\eqref{eq:scaled-Bergman-density}.  For every fixed \(\xi\),
\[
 \frac{\Gamma(\alpha)}{\Gamma(\alpha-m)\alpha^m}
 \longrightarrow1.
\]
The piecewise factor in \eqref{eq:scaled-Bergman-density} converges to
\(e^{-|\xi|^2}\).
The densities in \eqref{eq:scaled-Bergman-density} and the limiting
Gaussian density all have integral one.  Scheff\'e's lemma therefore
upgrades pointwise convergence to \(L^1\) convergence.
\end{proof}

\begin{remark}[Scope of the scaling]
Proposition~\ref{prop:Bergman-to-Gaussian-scaling} is a fixed-dimensional
limit as \(\alpha\to\infty\).  It asserts neither a rate of convergence nor
uniformity when the dimension grows, and it is not an extreme-value limit.
\end{remark}

\begin{corollary}
\label{cor:Gaussian-Fock-hypercontractivity}
Let \(m\ge2\) and \(0<p<q<\infty\).  For every entire function
\(F:\mathbb C^m\to\mathbb C\) with \(F\in L^p(\gamma_m)\),
\begin{equation}
 \left(
 \int_{\mathbb C^m}
 \left|F\!\left(\sqrt{\frac{p}q}\,\xi\right)\right|^q
 \,d\gamma_m(\xi)
 \right)^{1/q}
 \le
 \left(
 \int_{\mathbb C^m}|F(\xi)|^p\,d\gamma_m(\xi)
 \right)^{1/p}.
 \label{eq:Janson-hypercontractivity}
\end{equation}
Thus the sharp holomorphic Gaussian hypercontractivity inequality of
Janson \cite{Janson1983} is the high-weight flat limit of
Corollary~\ref{cor:sharp-Bergman-contractivity}.
\end{corollary}

\begin{proof}
First let \(F\) be a holomorphic polynomial.  For \(\alpha>m\), set
\[
 \beta=\frac{q}p\,\alpha,
 \qquad
 f_\alpha(z)=F(\sqrt{\alpha}\,z).
\]
Corollary~\ref{cor:sharp-Bergman-contractivity} gives
\[
 \|f_\alpha\|_{\mathcal A^q_\beta}
 \le
 \|f_\alpha\|_{\mathcal A^p_\alpha}.
\]
By Proposition~\ref{prop:Bergman-to-Gaussian-scaling},
\[
 \|f_\alpha\|_{\mathcal A^p_\alpha}^p
 =
 \int_{\mathbb C^m}|F(\xi)|^p
 \,d\widetilde\nu_{\alpha,m}(\xi),
\]
whereas
\[
 \|f_\alpha\|_{\mathcal A^q_\beta}^q
 =
 \int_{\mathbb C^m}
 \left|F\!\left(\sqrt{\frac{\alpha}{\beta}}\,\eta\right)\right|^q
 \,d\widetilde\nu_{\beta,m}(\eta).
\]
The case \(F\equiv0\) is immediate. Otherwise, let \(d=\deg F\)
and choose an integer \(k\) such that
\(2k>\max\{dp,dq\}\).  Since
\[
 |F(\xi)|^s\le C_s(1+|\xi|^{ds}),
 \qquad s\in\{p,q\},
\]
the required uniform integrability follows from the radial moment identity
\[
 \int_{\mathbb C^m}|\xi|^{2k}\,
 d\widetilde\nu_{\alpha,m}(\xi)
 =
 \alpha^k
 \frac{\Gamma(m+k)\Gamma(\alpha)}{\Gamma(m)\Gamma(\alpha+k)}
 \qquad(k=0,1,\ldots),
\]
whose right-hand side is bounded for fixed \(k\) and all sufficiently large
\(\alpha\).  Passing to the limit and using \(\alpha/\beta=p/q\) proves
\eqref{eq:Janson-hypercontractivity} for polynomials.

For a general entire \(F\in L^p(\gamma_m)\), choose holomorphic polynomials
\(P_\ell\to F\) in \(L^p(\gamma_m)\), using polynomial density in Gaussian
Fock spaces \cite[Chapter~2]{ZhuFock2012}.  Applying the polynomial
inequality to \(P_\ell-P_j\) shows that
\(P_\ell(\sqrt{p/q}\,\cdot)\) is Cauchy in \(L^q(\gamma_m)\).  This remains
valid below exponent one: for every \(0<s<\infty\), \(L^s\) is complete for
the metric
\[
 d_s(u,v)=\|u-v\|_{L^s(\gamma_m)}^{\min\{1,s\}}.
\]
The submean-value estimate for entire functions shows that
\(L^p(\gamma_m)\)-convergence implies local uniform convergence.  Therefore
the \(L^q\)-limit is \(F(\sqrt{p/q}\,\cdot)\), and passage to the limit in
the polynomial inequality proves the claim.
\end{proof}

\begin{remark}[Extremizers and entropy in the flat limit]
For fixed \(\mathbf a\in\mathbb C^m\), take
\(w_\alpha=\mathbf a/\sqrt{\alpha}\) in
\eqref{eq:Bergman-extremizer}.  In the scaled variable
\(z=\xi/\sqrt{\alpha}\),
\[
 \frac{(1-|w_\alpha|^2)^{\alpha/p}}{(1-\langle z,w_\alpha\rangle)^{2\alpha/p}}
 \longrightarrow
 \exp\left(
 -\frac{|\mathbf a|^2}{p}
 +\frac{2}{p}\langle\xi,\mathbf a\rangle
 \right).
\]
The convergence is locally uniform.  For the limiting coherent state
\[
 F_{\mathbf a}(\xi)
 =
 \exp\left(
 -\frac{|\mathbf a|^2}{p}
 +\frac{2}{p}\langle\xi,\mathbf a\rangle
 \right),
\]
the complex Gaussian identity
\[
 \int_{\mathbb C^m}
 e^{2\operatorname{Re}\langle\xi,\mathbf b\rangle}\,d\gamma_m(\xi)
 =
 e^{|\mathbf b|^2}
\]
gives
\[
 \int|F_{\mathbf a}(\xi)|^p\,d\gamma_m(\xi)=1,
 \qquad
 \int\left|F_{\mathbf a}\!\left(\sqrt{\frac{p}q}\,\xi\right)\right|^q
 \,d\gamma_m(\xi)=1.
\]
Thus the Bergman reproducing-kernel extremizers converge to exponential
coherent states that attain equality in
\eqref{eq:Janson-hypercontractivity}.  Likewise, as
\(\lambda\to\infty\),
\[
 \lambda\sum_{j=1}^m\frac{1}{\lambda-j}\longrightarrow m,
\]
so the \(SU(m,1)\) entropy constant converges to its Gaussian--Fock value.
These are compatibility statements; the Gaussian inequalities themselves
are classical.
\end{remark}

\subsection{A Bergman-space Faber--Krahn inequality}

For $s>0$, let $B_s\subset\mathbb B^m$ be the Euclidean ball centered at the
origin whose invariant measure is $\mu_m(B_s)=s$.  If its Euclidean radius
is $R_s$, then
\begin{equation}
 s=\left(\frac{R_s^2}{1-R_s^2}\right)^m,
 \qquad
 R_s^2=\frac{s^{1/m}}{1+s^{1/m}}.
 \label{eq:invariant-ball-radius}
\end{equation}

\begin{theorem}[Sharp Bergman Faber--Krahn concentration]
\label{thm:Bergman-Faber-Krahn}
Let $m\ge2$, $\alpha>m$, and $s>0$.  If
$f\in\mathcal A^2_\alpha(\mathbb B^m)$,
$\|f\|_{\mathcal A^2_\alpha}=1$, and $E\subset\mathbb B^m$ is measurable
with $\mu_m(E)=s$, then
\begin{equation}
 \int_E |f(z)|^2(1-|z|^2)^\alpha\,d\mu_m(z)
 \le
 \mathfrak F_{m,\alpha}(s),
 \label{eq:Bergman-Faber-Krahn}
\end{equation}
where
\begin{equation}
 \begin{aligned}
 \mathfrak F_{m,\alpha}(s)
 &:=\int_{B_s}(1-|z|^2)^\alpha\,d\mu_m(z)\\
 &=m\int_0^{\frac{s^{1/m}}{1+s^{1/m}}}
 t^{m-1}(1-t)^{\alpha-m-1}\,dt.
 \end{aligned}
 \label{eq:Bergman-Faber-Krahn-profile}
\end{equation}
Equality holds if and only if, up to a null set, $E$ is a complex-hyperbolic
ball centered at some $w\in\mathbb B^m$ and
\[
 f(z)=e^{i\theta}
 \frac{(1-|w|^2)^{\alpha/2}}{(1-\langle z,w\rangle)^\alpha}
\]
for some $\theta\in\mathbb R$.
\end{theorem}

\begin{proof}
Under the dictionary in Remark~\ref{rem:analytic-transfer}, Singh proves
this statement in every complex dimension under the hypothesis
that the isoperimetric regions of the complex hyperbolic ball are geodesic
balls; see \cite[Theorem~C]{Singh2026}.  The hypothesis is supplied here by
Theorem~\ref{thm:gromov-ros-all-m}. This proves
\eqref{eq:Bergman-Faber-Krahn} and its equality characterization.

It remains only to verify the explicit expression
\eqref{eq:Bergman-Faber-Krahn-profile}.  With $t=|z|^2$ and normalized
Euclidean measure, radial integration gives
\[
 d\mu_m=m t^{m-1}(1-t)^{-m-1}\,dt.
\]
Consequently
\[
 \mu_m(B_R)=m\int_0^{R^2}t^{m-1}(1-t)^{-m-1}\,dt
 =\left(\frac{R^2}{1-R^2}\right)^m,
\]
which proves \eqref{eq:invariant-ball-radius}; inserting the additional
factor $(1-t)^\alpha$ gives
\eqref{eq:Bergman-Faber-Krahn-profile}.
\end{proof}

\subsection{The Lieb--Wehrl inequality for
\texorpdfstring{$SU(m,1)$}{SU(m,1)}}

The terminology originates in Lieb's proof of Wehrl's entropy conjecture
\cite{Lieb1978}. Fix \(m\ge2\) and an integer \(k\ge1\), and put
\[
 \lambda=(m+1)k.
\]
The weighted Hilbert space $\mathcal A^2_\lambda$ realizes the corresponding
holomorphic discrete-series representation of $SU(m,1)$; see
\cite{Singh2026}.  Its normalized
coherent vector centered at $w\in\mathbb B^m$ is
\[
 \varphi_w(z)=
 \frac{(1-|w|^2)^{\lambda/2}}{(1-\langle z,w\rangle)^\lambda}.
\]
Here a density operator means a positive trace-class operator \(\varrho\) on
\(\mathcal A^2_\lambda\), normalized by \(\operatorname{tr}\varrho=1\).  Its
trace is
\(\operatorname{tr}\varrho=\sum_j\langle e_j,\varrho e_j\rangle\) for any
orthonormal basis \((e_j)\) of \(\mathcal A^2_\lambda\); the series converges,
its sum is finite, and its value is independent of the chosen basis.  Define
the Husimi function of \(\varrho\) by
\[
 h_\varrho(z)=\langle\varphi_z,\varrho\varphi_z\rangle.
\]
It satisfies $0\le h_\varrho\le1$ and
\[
 c_{\lambda,m}\int_{\mathbb B^m}h_\varrho\,d\mu_m=1.
\]

\begin{theorem}[Generalized Lieb--Wehrl inequality]
\label{thm:Lieb-Wehrl-SUm1}
Let \(m\ge2\), let \(k\ge1\), and put \(\lambda=(m+1)k\).
Let \(\varrho\) be a density operator on \(\mathcal A^2_\lambda\).
Then, for every convex
\(\Phi:[0,1]\to\mathbb R\) with
{\(\Phi(0)=0\)} for which the integrals are finite,
\begin{equation}
 \int_{\mathbb B^m}\Phi(h_\varrho(z))\,d\mu_m(z)
 \le
 \frac{m}{\lambda}\int_0^1
 \Phi(s)s^{-m/\lambda-1}
 \left(1-s^{1/\lambda}\right)^{m-1}\,ds.
 \label{eq:generalized-Lieb-Wehrl}
\end{equation}
The right-hand side is the value for any rank-one coherent projector
$\varrho=|\varphi_w\rangle\langle\varphi_w|$.  If $\Phi$ is strictly
convex, these are the only equality cases.
\end{theorem}

\begin{proof}
Under the dictionary in Remark~\ref{rem:analytic-transfer}, Singh's
generalized Lieb--Wehrl theorem for $SU(N,1)$ is conditional only
on the assertion that complex-hyperbolic isoperimetric regions are
geodesic balls; see \cite[Theorems~A and B]{Singh2026}.  Setting $N=m$ and
using Theorem~\ref{thm:gromov-ros-all-m} removes that condition for every
\(m\ge2\) and gives
\eqref{eq:generalized-Lieb-Wehrl}, including its equality statement.
\end{proof}

\begin{corollary}
\label{cor:Lieb-Wehrl-entropy}
Under the assumptions of Theorem~\ref{thm:Lieb-Wehrl-SUm1}, define the
normalized Wehrl entropy
\begin{equation}
 \mathcal S_\lambda(\varrho)
 =-c_{\lambda,m}\int_{\mathbb B^m}
 h_\varrho(z)\log h_\varrho(z)\,d\mu_m(z).
 \label{eq:Wehrl-entropy-definition}
\end{equation}
Then
\begin{equation}
 \mathcal S_\lambda(\varrho)
 \ge
 \lambda\bigl(\psi(\lambda)-\psi(\lambda-m)\bigr)
 =\lambda\sum_{j=1}^{m}\frac{1}{\lambda-j},
 \label{eq:Wehrl-entropy-bound}
\end{equation}
where $\psi=\Gamma'/\Gamma$ is the digamma function.  Equality holds if and
only if $\varrho$ is a rank-one coherent projector.  Without the normalizing factor
$c_{\lambda,m}$ in \eqref{eq:Wehrl-entropy-definition}, the right-hand side is
divided by $c_{\lambda,m}$.
\end{corollary}
\begin{remark}[The one-dimensional case]
For \(m=1\), the preceding entropy inequality is the
Lieb--Wehrl inequality for \(SU(1,1)\). In this case,
\[
\mathcal S_\lambda(\varrho)
\ge
\lambda\bigl(\psi(\lambda)-\psi(\lambda-1)\bigr)
=
\frac{\lambda}{\lambda-1}.
\]
The corresponding convex-functional extremal principle for normalized
functions in weighted Bergman spaces was proved by Kulikov
\cite{Kulikov2022}, resolving the \(SU(1,1)\) conjecture formulated by
Lieb and Solovej \cite{LiebSolovej2021}. The extension from pure states
to arbitrary density operators follows from the spectral decomposition
of the density operator and Jensen's inequality. Equality holds precisely
for rank-one coherent projectors. At the weighted Bergman-space level,
the result is valid for every real parameter \(\lambda>1\).
\end{remark}
\begin{proof}
Let
\[
\Phi(s)=s\log s,\qquad s\geq 0,
\]
with the convention \(0\log 0=0\). Since \(\Phi\) is strictly convex,
Theorem~\ref{thm:Lieb-Wehrl-SUm1} implies that
\[
c_{\lambda,m}\int h\log h\,d\mu_m
\leq
c_{\lambda,m}\int
h_{\mathrm{coh}}\log h_{\mathrm{coh}}\,d\mu_m,
\]
where \(h_{\mathrm{coh}}\) is the Husimi function of a coherent-state
projector. Multiplying the preceding inequality by \(-1\), we obtain
\[
-c_{\lambda,m}\int h\log h\,d\mu_m
\geq
-c_{\lambda,m}\int
h_{\mathrm{coh}}\log h_{\mathrm{coh}}\,d\mu_m.
\]
It therefore remains to compute the coherent-state entropy.

By the \(SU(m,1)\)-invariance of the problem, we may take the coherent
projector centered at the origin. Its Husimi function is
\[
h_{\mathrm{coh}}(z)=(1-|z|^2)^\lambda.
\]
For
\[
q>\frac{m}{\lambda},
\]
the distribution formula in
Theorem~\ref{thm:Lieb-Wehrl-SUm1} gives
\[
\int h_{\mathrm{coh}}^q\,d\mu_m
=
\frac{m}{\lambda}
\int_0^1
s^{q-m/\lambda-1}
\left(1-s^{1/\lambda}\right)^{m-1}\,ds.
\]
The restriction \(q>m/\lambda\) guarantees convergence at \(s=0\).
Since \(\lambda>m\), the admissible range contains an open neighborhood
of \(q=1\).

We make the substitution
\[
t=s^{1/\lambda},
\qquad
s=t^\lambda,
\qquad
ds=\lambda t^{\lambda-1}\,dt.
\]
Then
\[
s^{q-m/\lambda-1}
=
t^{\lambda q-m-\lambda},
\]
and consequently
\[
\begin{aligned}
\frac{m}{\lambda}
\int_0^1
s^{q-m/\lambda-1}
\left(1-s^{1/\lambda}\right)^{m-1}\,ds
&=
m\int_0^1
t^{\lambda q-m-1}(1-t)^{m-1}\,dt \\
&=
m\,B(\lambda q-m,m).
\end{aligned}
\]
Using the beta-function identity
\[
B(x,y)
=
\frac{\Gamma(x)\Gamma(y)}{\Gamma(x+y)}
\]
and the identity
\[
\Gamma(m)=(m-1)!,
\]
we obtain
\[
\begin{aligned}
m\,B(\lambda q-m,m)
&=
m\,
\frac{\Gamma(\lambda q-m)\Gamma(m)}{\Gamma(\lambda q)} \\
&=
m!\,
\frac{\Gamma(\lambda q-m)}{\Gamma(\lambda q)}.
\end{aligned}
\]
Therefore
\[
\int h_{\mathrm{coh}}^q\,d\mu_m
=
m!\,
\frac{\Gamma(\lambda q-m)}{\Gamma(\lambda q)}.
\]

Since
\[
c_{\lambda,m}
=
\frac{\Gamma(\lambda)}{m!\Gamma(\lambda-m)},
\]
the normalized \(q\)-moment is
\[
\mathcal M(q)
:=
c_{\lambda,m}\int h_{\mathrm{coh}}^q\,d\mu_m
=
\frac{\Gamma(\lambda)}{\Gamma(\lambda-m)}
\frac{\Gamma(\lambda q-m)}{\Gamma(\lambda q)}.
\]
At \(q=1\), this gives
\[
\mathcal M(1)
=
\frac{\Gamma(\lambda)}{\Gamma(\lambda-m)}
\frac{\Gamma(\lambda-m)}{\Gamma(\lambda)}
=
1.
\]
Thus
\[
c_{\lambda,m}
\int h_{\mathrm{coh}}\,d\mu_m
=
1.
\]

We next differentiate \(\mathcal M(q)\) at \(q=1\). Since
\(q=1\) lies in the interior of the range \(q>m/\lambda\),
differentiation under the integral sign is justified. Hence
\[
\mathcal M'(q)
=
c_{\lambda,m}
\int
h_{\mathrm{coh}}^q
\log h_{\mathrm{coh}}\,d\mu_m,
\]
and therefore
\[
\mathcal M'(1)
=
c_{\lambda,m}
\int
h_{\mathrm{coh}}
\log h_{\mathrm{coh}}\,d\mu_m.
\]

Recall that the digamma function is defined by
\[
\psi(x)
=
\frac{\Gamma'(x)}{\Gamma(x)}
=
\frac{d}{dx}\log\Gamma(x).
\]
Taking the logarithmic derivative of \(\mathcal M(q)\), we obtain
\[
\frac{\mathcal M'(q)}{\mathcal M(q)}
=
\lambda\psi(\lambda q-m)
-
\lambda\psi(\lambda q).
\]
Since
\[
\mathcal M(1)=1,
\]
evaluation at \(q=1\) yields
\[
c_{\lambda,m}
\int
h_{\mathrm{coh}}
\log h_{\mathrm{coh}}\,d\mu_m
=
\lambda
\left(
\psi(\lambda-m)-\psi(\lambda)
\right).
\]
Consequently,
\[
-c_{\lambda,m}
\int
h_{\mathrm{coh}}
\log h_{\mathrm{coh}}\,d\mu_m
=
\lambda
\left(
\psi(\lambda)-\psi(\lambda-m)
\right).
\]
Combining this identity with the inequality obtained from
Theorem~\ref{thm:Lieb-Wehrl-SUm1}, we conclude that
\[
-c_{\lambda,m}
\int h\log h\,d\mu_m
\geq
\lambda
\left(
\psi(\lambda)-\psi(\lambda-m)
\right).
\]

Finally, since \(\lambda\) and \(\lambda-m\) are positive integers, the
recurrence relation
\[
\psi(x+1)-\psi(x)
=
\frac{1}{x}
\]
gives
\[
\begin{aligned}
\psi(\lambda)-\psi(\lambda-m)
&=
\sum_{k=\lambda-m}^{\lambda-1}\frac{1}{k} \\
&=
\sum_{j=1}^{m}\frac{1}{\lambda-j}.
\end{aligned}
\]
Hence
\[
-c_{\lambda,m}
\int h\log h\,d\mu_m
\geq
\lambda
\sum_{j=1}^{m}\frac{1}{\lambda-j},
\]
which proves the two equivalent expressions in
\eqref{eq:Wehrl-entropy-bound}.

Because \(\Phi(s)=s\log s\) is strictly convex, the equality statement in
Theorem~\ref{thm:Lieb-Wehrl-SUm1} applies. Therefore equality holds
precisely for coherent-state projectors, up to the natural
\(SU(m,1)\)-action.
\end{proof}
\begin{remark}[Related quantitative results]
The Faber--Krahn theorem of Nicola and Tilli \cite{NicolaTilli2022}
concerns Gaussian short-time Fourier transforms, equivalently the Fock
space. Quantitative stability in that problem was subsequently established
by G\'omez, Guerra, Ramos, and Tilli
\cite{GomezGuerraRamosTilli2024}. G\'omez, Kalaj, Melentijevi\'c, and Ramos
developed a uniform stability theory for related Cauchy-wavelet
concentration inequalities, including their Hardy-space limit
\cite{GomezKalajMelentijevicRamos2025}. These results are not used here, but
they suggest quantitative refinements of the Bergman concentration
inequalities proved above.  In the convex-functional setting, the
higher-dimensional stability theorem of Melentijevi\'c
\cite[Theorem~4]{Melentijevic2025} is a direct quantitative consequence of
the now-unconditional contraction principle
\eqref{eq:convex-Bergman-contraction}.
\end{remark}

\ithead{Final perspective.}
The volume-radius formulation shows that the geometric core of the argument
is rank-one rather than specifically complex: its cofactor Hessian detects
the nonradial part of the boundary normal, while the log-partition correction
supplies exact-volume variations in every root multiplicity.  In the complex
case the same log-partition Hessian is an angular covariance whose Gaussian
average gives the trace, and the high-weight flat limit carries the resulting
Bergman contractivity to the classical Gaussian--Fock hypercontractive
inequality.

\appendix
\section{Extended algebraic verification of Theorems~2.7--2.11}
\label{app:extended-trace-verification}

This appendix expands the algebra behind Theorems~\ref{thm:intrinsic-trace},
\ref{thm:strict-positivity}, \ref{thm:quaternionic-trace-sign}, and
\ref{thm:Cayley-trace-sign}, together with Corollary~\ref{cor:tangent-plane-trace}.
The purpose is not to introduce new geometry, but to make every specialization
of the universal trace formula reproducible by hand.

Throughout, the starting point is Theorem~\ref{thm:universal-rank-one-trace}:
\begin{equation}
\begin{aligned}
\mathscr T_{p,q}
={}&-nZ-Z^2\\
&+\alpha^2\left(
-\ddot A+2\dot A^2
+\eta_H^2(p-\beta)
+\eta_V^2(q-\gamma)
\right)\\
&+\beta(-\ddot B+2\dot B^2)
+\gamma(-\ddot C+2\dot C^2),
\end{aligned}
\label{eq:appx-universal-trace}
\end{equation}
where
\[
A=\log\mathfrak a,\qquad B=\log b,\qquad C=\log c,
\]
\[
\eta_H=\frac{\mathfrak a}{b},\qquad
\eta_V=\frac{\mathfrak a}{c},
\]
and
\[
Z=\dot A\,\alpha^2+\dot B\,\beta+\dot C\,\gamma.
\]
We use the abbreviation
\begin{equation}
\delta:=\alpha^2=1-\beta-\gamma.
\label{eq:appx-delta}
\end{equation}
All dots denote differentiation with respect to $y=\log R$.

\subsection{The complex case: derivation of Theorem~2.7}

Here
\[
q=1,\qquad p=n-2.
\]
Since $R=\sinh r$, put
\[
x=\sinh^2r=R^2.
\]
Then
\[
y=\log R=\frac{1}{2}\log x,
\qquad
\frac{d}{dy}=2x\frac{d}{dx}.
\]
The three metric factors are
\[
\mathfrak a=\frac{dr}{dy}=\tanh r=\sqrt{\frac{x}{1+x}},
\qquad
b=\sqrt x,
\qquad
c=\sqrt{x(1+x)}.
\]
Thus
\[
A=\frac{1}{2}\log x-\frac{1}{2}\log(1+x),
\qquad
B=\frac{1}{2}\log x,
\qquad
C=\frac{1}{2}\log x+\frac{1}{2}\log(1+x).
\]
Applying $2x\partial_x$ gives
\begin{equation}
\dot A=\frac{1}{1+x},
\qquad
\ddot A=-\frac{2x}{(1+x)^2},
\label{eq:appx-complex-A}
\end{equation}
\begin{equation}
\dot B=1,
\qquad
\ddot B=0,
\label{eq:appx-complex-B}
\end{equation}
and
\begin{equation}
\dot C=\frac{1+2x}{1+x},
\qquad
\ddot C=\frac{2x}{(1+x)^2}.
\label{eq:appx-complex-C}
\end{equation}
Moreover,
\begin{equation}
\eta_H^2=\frac{1}{1+x},
\qquad
\eta_V^2=\frac{1}{(1+x)^2}.
\label{eq:appx-complex-eta}
\end{equation}

The quantity $Z$ simplifies before the main expansion:
\begin{align}
Z
&=\frac{\delta}{1+x}+\beta+\frac{1+2x}{1+x}\gamma\notag\\
&=\frac{\delta+(1+x)\beta+(1+2x)\gamma}{1+x}\notag\\
&=\frac{1+x\beta+2x\gamma}{1+x},
\label{eq:appx-complex-Z}
\end{align}
where the last step uses $\delta=1-\beta-\gamma$.

It is useful to simplify the three second-order scalar combinations appearing
in \eqref{eq:appx-universal-trace}.  From
\eqref{eq:appx-complex-A}--\eqref{eq:appx-complex-C},
\begin{equation}
-\ddot A+2\dot A^2
=\frac{2x+2}{(1+x)^2}
=\frac{2}{1+x},
\label{eq:appx-complex-SA}
\end{equation}
\begin{equation}
-\ddot B+2\dot B^2=2,
\label{eq:appx-complex-SB}
\end{equation}
and
\begin{equation}
-\ddot C+2\dot C^2
=\frac{2+6x+8x^2}{(1+x)^2}.
\label{eq:appx-complex-SC}
\end{equation}

Substituting \eqref{eq:appx-complex-Z}--\eqref{eq:appx-complex-SC} into
\eqref{eq:appx-universal-trace} gives
\begin{align}
\mathscr T_{n-2,1}
={}&-n\frac{1+x\beta+2x\gamma}{1+x}
-\frac{(1+x\beta+2x\gamma)^2}{(1+x)^2}\notag\\
&+\delta\left[
\frac{2}{1+x}
+\frac{n-2-\beta}{1+x}
+\frac{1-\gamma}{(1+x)^2}
\right]\notag\\
&+2\beta
+\gamma\frac{2+6x+8x^2}{(1+x)^2}.
\label{eq:appx-complex-before-clear}
\end{align}
Multiplying by $(1+x)^2$ yields
\begin{align}
(1+x)^2\mathscr T_{n-2,1}
={}&-n(1+x)(1+x\beta+2x\gamma)\notag\\
&-(1+x\beta+2x\gamma)^2\notag\\
&+\delta\big[(n-\beta)(1+x)+(1-\gamma)\big]\notag\\
&+2\beta(1+x)^2+\gamma(2+6x+8x^2).
\label{eq:appx-complex-cleared}
\end{align}
We now substitute $\delta=1-\beta-\gamma$ and collect the coefficients of
\[
1,\quad \beta,\quad\gamma,\quad\beta^2,\quad\beta\gamma,\quad\gamma^2.
\]
The constant coefficient cancels.  The remaining coefficients of
$-(1+x)^2\mathscr T_{n-2,1}$ are
\begin{align*}
[\beta^2]&=x^2-x-1,\\
[\beta\gamma]&=4x^2-x-2,\\
[\gamma^2]&=4x^2-1,\\
[\beta]&=(n-2)x^2+(2n-1)x+n,\\
[\gamma]&=(2n-8)x^2+(3n-2)x+n.
\end{align*}
Consequently
\begin{equation}
\mathscr T_{n-2,1}
=-\frac{\mathcal R_{n,x}(\beta,\gamma)}{(1+x)^2},
\label{eq:appx-complex-final}
\end{equation}
where
\begin{equation}
\begin{aligned}
\mathcal R_{n,x}(\beta,\gamma)
={}&(x^2-x-1)\beta^2
 +(4x^2-x-2)\beta\gamma
 +(4x^2-1)\gamma^2\\
&+\bigl((n-2)x^2+(2n-1)x+n\bigr)\beta\\
&+\bigl((2n-8)x^2+(3n-2)x+n\bigr)\gamma.
\end{aligned}
\label{eq:appx-complex-R}
\end{equation}
This reproduces Theorem~\ref{thm:intrinsic-trace} exactly.

\subsection{Corollary~2.8 and the reduced-boundary interpretation}

Corollary~\ref{cor:tangent-plane-trace} introduces no new polynomial
identity.  The point is that the formula in Theorem~\ref{thm:intrinsic-trace}
is a pointwise statement for a hyperplane $\Pi=\nu^\perp$, depending on the
normal only through
\[
\nu=\alpha N+v+w,
\qquad
\beta=|v|^2,
\qquad
\gamma=|w|^2.
\]
At $\mathcal H^{n-1}$-almost every point of a reduced boundary $\partial^*E$
there is an approximate tangent hyperplane $\Pi_x=\nu_E(x)^\perp$.  Since the
tangential Jacobian is
\[
J_{\Pi_x}(t,\varphi)=|\cof DF_{t,\varphi}\,\nu_E(x)|,
\]
the same pointwise identity applies there.  Thus Corollary~2.8 is the bridge
from the algebraic hyperplane calculation to the finite-perimeter integral.

\subsection{Theorem~2.9: exact rearrangement and positivity}

Set
\[
\delta=1-\beta-\gamma,
\qquad
L_T=(n-1)+(3n-2)x+(2n-5)x^2.
\]
The claimed rearrangement is
\begin{equation}
\begin{aligned}
\mathcal R_{n,x}
={}&(n-1)(1+x)^2\beta
+\delta(1+x-x^2)\beta
+L_T\gamma\\
&+\delta(1-3x^2)\gamma
+x^2\gamma^2.
\end{aligned}
\label{eq:appx-complex-rearrangement}
\end{equation}
To verify it, expand only the terms containing $\delta$:
\begin{align}
\delta(1+x-x^2)\beta
={}&(1+x-x^2)\beta
-(1+x-x^2)\beta^2\notag\\
&-(1+x-x^2)\beta\gamma,
\label{eq:appx-delta-beta-expand}
\end{align}
and
\begin{align}
\delta(1-3x^2)\gamma
={}&(1-3x^2)\gamma
-(1-3x^2)\beta\gamma\notag\\
&-(1-3x^2)\gamma^2.
\label{eq:appx-delta-gamma-expand}
\end{align}
Hence
\begin{align*}
[\beta^2]&=-(1+x-x^2)=x^2-x-1,\\
[\beta\gamma]&=-(1+x-x^2)-(1-3x^2)=4x^2-x-2,\\
[\gamma^2]&=-(1-3x^2)+x^2=4x^2-1,
\end{align*}
and
\begin{align*}
[\beta]
&=(n-1)(1+x)^2+(1+x-x^2)\\
&=n+(2n-1)x+(n-2)x^2,\\
[\gamma]
&=L_T+(1-3x^2)\\
&=n+(3n-2)x+(2n-8)x^2.
\end{align*}
These are exactly the coefficients in \eqref{eq:appx-complex-R}; therefore
\eqref{eq:appx-complex-rearrangement} is an identity.

For the sign, use the elementary bound
\[
0\le\delta\le1
\quad\Longrightarrow\quad
\delta c\ge\min\{0,c\}
\qquad(c\in\mathbb R).
\]
Thus
\[
\mathcal R_{n,x}
\ge c_H\beta+c_T\gamma+x^2\gamma^2,
\]
where
\[
c_H=(n-1)(1+x)^2+\min\{0,1+x-x^2\},
\]
and
\[
c_T=L_T+\min\{0,1-3x^2\}.
\]
If $1+x-x^2\ge0$, then $c_H\ge(n-1)(1+x)^2>0$.  If
$1+x-x^2<0$, then
\[
c_H=n+(2n-1)x+(n-2)x^2>0.
\]
Similarly, if $1-3x^2\ge0$, then $c_T\ge L_T>0$, while if
$1-3x^2<0$, then
\[
c_T=n+(3n-2)x+(2n-8)x^2>0
\]
for $n\ge4$ and $x\ge0$.  In the borderline case $n=4$ this becomes
$c_T=4+10x>0$.  Hence
\[
\mathcal R_{n,x}\ge0,
\qquad
\mathcal R_{n,x}=0
\quad\Longleftrightarrow\quad
\beta=\gamma=0.
\]
This verifies Theorem~\ref{thm:strict-positivity}.

\subsection{A bookkeeping identity for the $\delta$-basis}

The quaternionic and Cayley formulas are most transparent when written in the
basis
\[
\beta^2,\quad\beta\gamma,\quad\gamma^2,
\quad\beta\delta,\quad\gamma\delta,
\qquad
\delta=1-\beta-\gamma.
\]
Suppose a quadratic polynomial with no constant term is first obtained in the
ordinary basis as
\begin{equation}
Q=q_{20}\beta^2+q_{11}\beta\gamma+q_{02}\gamma^2
+q_{10}\beta+q_{01}\gamma.
\label{eq:appx-raw-basis}
\end{equation}
If we want to rewrite it as
\begin{equation}
Q=A_{20}\beta^2+A_{11}\beta\gamma+A_{02}\gamma^2
+A_{10}\beta\delta+A_{01}\gamma\delta,
\label{eq:appx-delta-basis}
\end{equation}
then expanding $\delta=1-\beta-\gamma$ gives
\begin{align*}
Q
={}&(A_{20}-A_{10})\beta^2
+(A_{11}-A_{10}-A_{01})\beta\gamma\\
&+(A_{02}-A_{01})\gamma^2
+A_{10}\beta+A_{01}\gamma.
\end{align*}
Therefore
\begin{equation}
A_{10}=q_{10},\qquad A_{01}=q_{01},
\label{eq:appx-delta-basis-linear}
\end{equation}
\begin{equation}
A_{20}=q_{20}+q_{10},\qquad
A_{02}=q_{02}+q_{01},
\label{eq:appx-delta-basis-diagonal}
\end{equation}
and
\begin{equation}
A_{11}=q_{11}+q_{10}+q_{01}.
\label{eq:appx-delta-basis-mixed}
\end{equation}
These relations explain exactly how the five positive coefficients displayed
in Theorems~2.10 and~2.11 arise from the ordinary polynomial expansion.

\subsection{The quaternionic case: derivation of Theorem~2.10}

Here
\[
q=3,\qquad p=n-4,
\]
and
\[
R^n=x^{n/2}\left(1+\frac{n}{n+2}x\right).
\]
Set
\[
D:=nx+n+2.
\]
Then
\begin{equation}
R^n=x^{n/2}\frac{D}{n+2}.
\label{eq:appx-quat-R}
\end{equation}
Taking logarithms,
\[
y=\frac{1}{2}\log x+\frac{1}n\log D-\frac{1}n\log(n+2),
\]
so
\begin{align}
\frac{dy}{dx}
&=\frac{1}{2x}+\frac{1}D
=\frac{D+2x}{2xD}\notag\\
&=\frac{(n+2)(1+x)}{2xD},
\label{eq:appx-quat-dydx}
\end{align}
because $D+2x=(n+2)(1+x)$.  Hence
\begin{equation}
\frac{d}{dy}
=\frac{2xD}{(n+2)(1+x)}\frac{d}{dx}.
\label{eq:appx-quat-Dy}
\end{equation}
Since
\[
\frac{dr}{dx}=\frac{1}{2\sqrt{x(1+x)}},
\]
we obtain
\begin{equation}
\mathfrak a
=\frac{\sqrt x\,D}{(n+2)(1+x)^{3/2}},
\qquad
b=\sqrt x,
\qquad
c=\sqrt{x(1+x)}.
\label{eq:appx-quat-abc}
\end{equation}
Thus
\[
A=\frac{1}{2}\log x+\log D-\log(n+2)-\frac{3}{2}\log(1+x),
\]
\[
B=\frac{1}{2}\log x,
\qquad
C=\frac{1}{2}\log x+\frac{1}{2}\log(1+x).
\]
Applying \eqref{eq:appx-quat-Dy} gives
\begin{equation}
\dot A=\frac{nx+n-4x+2}{(n+2)(1+x)^2},
\label{eq:appx-quat-Adot}
\end{equation}
\begin{equation}
\ddot A
=-\frac{2xD(nx+n-4x+8)}{(n+2)^2(1+x)^4},
\label{eq:appx-quat-Addot}
\end{equation}
\begin{equation}
\dot B=\frac{D}{(n+2)(1+x)},
\qquad
\ddot B=-\frac{4xD}{(n+2)^2(1+x)^3},
\label{eq:appx-quat-Bdots}
\end{equation}
and
\begin{equation}
\dot C=\frac{(2x+1)D}{(n+2)(1+x)^2},
\qquad
\ddot C=\frac{2xD(nx+n-4x)}{(n+2)^2(1+x)^4}.
\label{eq:appx-quat-Cdots}
\end{equation}
Furthermore,
\begin{equation}
\eta_H^2=\frac{D^2}{(n+2)^2(1+x)^3},
\qquad
\eta_V^2=\frac{D^2}{(n+2)^2(1+x)^4}.
\label{eq:appx-quat-eta}
\end{equation}
Using $\delta=1-\beta-\gamma$ once, and only once, gives
\begin{equation}
Z=\frac{
(nx+n-4x+2)
+x(nx+n+6)\beta
+2x(nx+n+4)\gamma
}{(n+2)(1+x)^2}.
\label{eq:appx-quat-Z}
\end{equation}

Insert \eqref{eq:appx-quat-Adot}--\eqref{eq:appx-quat-Z} into
\eqref{eq:appx-universal-trace} and clear the common denominator by defining
\begin{equation}
Q_{\mathbb H}
:=-(n+2)^2(1+x)^4\mathscr T_{n-4,3}.
\label{eq:appx-quat-Q}
\end{equation}
The constant term of $Q_{\mathbb H}$ vanishes, and the ordinary-basis
expansion has the form \eqref{eq:appx-raw-basis}.  The two linear coefficients
are
\begin{align}
q_{10}
={}&n^3x^4+(4n^3+5n^2-16n)x^3\notag\\
&+(6n^3+14n^2-16n-48)x^2\notag\\
&+(4n^3+13n^2+4n-12)x+n^3+4n^2+4n,
\label{eq:appx-quat-q10}
\end{align}
and
\begin{align}
q_{01}
={}&(2n^3-4n^2)x^4+(7n^3-48n)x^3\notag\\
&+(9n^3+16n^2-48n-96)x^2\notag\\
&+(5n^3+16n^2+4n-16)x+n^3+4n^2+4n.
\label{eq:appx-quat-q01}
\end{align}
Therefore, by \eqref{eq:appx-delta-basis-linear},
\[
A_{10}=q_{10},\qquad A_{01}=q_{01}.
\]
The remaining three combinations
\eqref{eq:appx-delta-basis-diagonal}--\eqref{eq:appx-delta-basis-mixed}
simplify to
\begin{align}
A_{20}
={}&(1+x)D\Big[(n^2+n)(1+x)^2-6x-2\Big],
\label{eq:appx-quat-A20}
\end{align}
\begin{align}
A_{11}
={}&D\Big[3n^2x^3+(8n^2+6n-24)x^2\notag\\
&\hspace{2.5cm}+(7n^2+8n-16)x+2n^2+2n-4\Big],
\label{eq:appx-quat-A11}
\end{align}
and
\begin{align}
A_{02}
={}&D\Big[2n^2x^3+(5n^2+4n-16)x^2\notag\\
&\hspace{2.5cm}+(4n^2+5n-8)x+n^2+n-2\Big].
\label{eq:appx-quat-A02}
\end{align}
Therefore
\begin{equation}
\mathscr T_{n-4,3}
=-\frac{
A_{20}\beta^2+A_{11}\beta\gamma+A_{02}\gamma^2
+A_{10}\beta\delta+A_{01}\gamma\delta
}{(n+2)^2(1+x)^4}.
\label{eq:appx-quat-final}
\end{equation}
This is precisely Theorem~\ref{thm:quaternionic-trace-sign}.

For completeness, we verify positivity of the coefficients.  The factor
$D=nx+n+2$ is positive.  In $A_{20}$,
\[
(n^2+n)(1+x)^2-6x-2
=(n^2+n-2)+(2n^2+2n-6)x+(n^2+n)x^2,
\]
whose coefficients are positive for $n\ge4$.  The coefficients displayed in
$A_{11}$ and $A_{02}$ are likewise positive for $n\ge4$.  For $A_{10}$ and
$A_{01}$, the only potentially non-obvious coefficients satisfy
\[
4n^3+5n^2-16n=n(4n^2+5n-16)>0,
\]
\[
2n^3-4n^2=2n^2(n-2)>0,
\qquad
7n^3-48n=n(7n^2-48)>0,
\]
and
\[
9n^3+16n^2-48n-96>0
\qquad(n\ge4).
\]
The remaining coefficients are manifestly positive or are already positive at
$n=4$ and increase thereafter.  Hence every $A_{ij}>0$ for $n\ge4$ and
$x\ge0$.  Since $\beta,\gamma,\delta\ge0$,
\[
\mathscr T_{n-4,3}\le0.
\]
If $\beta>0$, then $A_{20}\beta^2>0$; if $\gamma>0$, then
$A_{02}\gamma^2>0$.  Thus
\[
\mathscr T_{n-4,3}=0
\quad\Longleftrightarrow\quad
\beta=\gamma=0.
\]

\subsection{The Cayley case: derivation of Theorem~2.11}

Now
\[
n=16,\qquad p=8,\qquad q=7.
\]
Define
\[
L(x)=120x^3+396x^2+440x+165.
\]
The volume-radius formula can be rewritten as
\begin{equation}
R^{16}=\frac{x^8L(x)}{165}.
\label{eq:appx-cayley-R}
\end{equation}
Taking logarithms,
\[
y=\frac{1}{2}\log x+\frac{1}{16}\log L-\frac{1}{16}\log165,
\]
so
\begin{equation}
\frac{dy}{dx}=\frac{1}{2x}+\frac{L'}{16L}
=\frac{8L+xL'}{16xL}.
\label{eq:appx-cayley-dydx-pre}
\end{equation}
The crucial simplification is
\begin{equation}
8L(x)+xL'(x)=1320(1+x)^3.
\label{eq:appx-cayley-key}
\end{equation}
Indeed,
\[
8L=960x^3+3168x^2+3520x+1320,
\]
\[
xL'=360x^3+792x^2+440x,
\]
and their sum is
\[
1320x^3+3960x^2+3960x+1320=1320(1+x)^3.
\]
Hence
\begin{equation}
\frac{d}{dy}
=\frac{2xL}{165(1+x)^3}\frac{d}{dx}.
\label{eq:appx-cayley-Dy}
\end{equation}
Since $dr/dx=(2\sqrt{x(1+x)})^{-1}$,
\begin{equation}
\mathfrak a=\frac{\sqrt x\,L}{165(1+x)^{7/2}},
\qquad
b=\sqrt x,
\qquad
c=\sqrt{x(1+x)}.
\label{eq:appx-cayley-abc}
\end{equation}
Therefore
\[
A=\frac{1}{2}\log x+\log L-\log165-\frac{7}{2}\log(1+x),
\]
\[
B=\frac{1}{2}\log x,
\qquad
C=\frac{1}{2}\log x+\frac{1}{2}\log(1+x).
\]
Applying \eqref{eq:appx-cayley-Dy} gives
\begin{equation}
\dot A=\frac{48x^3+220x^2+330x+165}{165(1+x)^4},
\label{eq:appx-cayley-Adot}
\end{equation}
\begin{equation}
\ddot A
=-\frac{4x(24x^3+148x^2+275x+165)L}{27225(1+x)^8},
\label{eq:appx-cayley-Addot}
\end{equation}
\begin{equation}
\dot B=\frac{L}{165(1+x)^3},
\qquad
\ddot B=-\frac{2x(36x^2+88x+55)L}{27225(1+x)^7},
\label{eq:appx-cayley-Bdots}
\end{equation}
and
\begin{equation}
\dot C=\frac{(2x+1)L}{165(1+x)^4},
\qquad
\ddot C=\frac{4x(24x^3+92x^2+121x+55)L}{27225(1+x)^8}.
\label{eq:appx-cayley-Cdots}
\end{equation}
Moreover,
\begin{equation}
\eta_H^2=\frac{L^2}{27225(1+x)^7},
\qquad
\eta_V^2=\frac{L^2}{27225(1+x)^8}.
\label{eq:appx-cayley-eta}
\end{equation}
Finally,
\begin{equation}
Z=\frac{Z_0+Z_H\beta+Z_V\gamma}{165(1+x)^4},
\label{eq:appx-cayley-Z}
\end{equation}
where
\begin{align*}
Z_0&=48x^3+220x^2+330x+165,\\
Z_H&=120x^4+468x^3+616x^2+275x,\\
Z_V&=240x^4+864x^3+1056x^2+440x.
\end{align*}

Insert these expressions into \eqref{eq:appx-universal-trace} and clear the
common denominator by setting
\begin{equation}
Q_{\mathbb O}:=-27225(1+x)^8\mathscr T_{8,7}.
\label{eq:appx-cayley-Q}
\end{equation}
Again the constant term vanishes, and $Q_{\mathbb O}$ is quadratic in
$(\beta,\gamma)$.  The ordinary-basis linear coefficients are
\begin{align}
q_{10}
={}&288000x^8+2399040x^7+8774208x^6+18408720x^5\notag\\
&+24235728x^4+20497400x^3+10865800x^2\notag\\
&+3294225x+435600,
\label{eq:appx-cayley-q10}
\end{align}
and
\begin{align}
q_{01}
={}&8\bigl(64800x^8+512640x^7+1774992x^6+3511992x^5\notag\\
&\hspace{0.7cm}+4339236x^4+3421880x^3+1675850x^2\notag\\
&\hspace{0.7cm}+462825x+54450\bigr).
\label{eq:appx-cayley-q01}
\end{align}
Therefore $B_{10}=q_{10}$ and $B_{01}=q_{01}$.  Applying the conversion
relations \eqref{eq:appx-delta-basis-diagonal}--\eqref{eq:appx-delta-basis-mixed}
to the three quadratic coefficients gives
\begin{align}
B_{20}
={}&(1+x)L(x)\notag\\
&\times(2520x^4+9972x^3+14828x^2+9845x+2475),
\label{eq:appx-cayley-B20}
\end{align}
\begin{align}
B_{11}
={}&2L(x)(3600x^5+16944x^4+31752x^3\notag\\
&\hspace{2.2cm}+29568x^2+13640x+2475),
\label{eq:appx-cayley-B11}
\end{align}
and
\begin{align}
B_{02}
={}&L(x)(4800x^5+21792x^4+39144x^3\notag\\
&\hspace{2.2cm}+34628x^2+14960x+2475).
\label{eq:appx-cayley-B02}
\end{align}
Thus
\begin{equation}
\mathscr T_{8,7}
=-\frac{
B_{20}\beta^2+B_{11}\beta\gamma+B_{02}\gamma^2
+B_{10}\beta\delta+B_{01}\gamma\delta
}{27225(1+x)^8}.
\label{eq:appx-cayley-final}
\end{equation}
This is Theorem~\ref{thm:Cayley-trace-sign}.

The sign is immediate from the displayed factorization: $L(x)>0$ for
$x\ge0$, every coefficient in each factor of $B_{20},B_{11},B_{02}$ is
positive, and every coefficient of $B_{10},B_{01}$ is positive.  Since
$\beta,\gamma,\delta\ge0$,
\[
\mathscr T_{8,7}\le0.
\]
If $\beta>0$, then $B_{20}\beta^2>0$; if $\gamma>0$, then
$B_{02}\gamma^2>0$.  Therefore
\[
\mathscr T_{8,7}=0
\quad\Longleftrightarrow\quad
\beta=\gamma=0.
\]

\subsection{What the equality condition means geometrically}

The normal decomposition is
\[
\nu=\alpha N+v+w,
\qquad
\beta=|v|^2,
\qquad
\gamma=|w|^2,
\qquad
\alpha^2+\beta+\gamma=1.
\]
Thus
\[
\beta=\gamma=0
\quad\Longleftrightarrow\quad
v=w=0
\quad\Longleftrightarrow\quad
\alpha^2=1
\quad\Longleftrightarrow\quad
\nu=\pm N.
\]
Accordingly, Theorems~2.7--2.11 establish the pointwise rigidity statement
\[
\text{if the boundary normal has a nonradial component, then }
\mathscr T_{p,q}<0.
\]
For an isoperimetric region, the exact-volume family satisfies
\[
0\le\operatorname{tr}D^2\mathscr P_E(0).
\]
Proposition~\ref{prop:rank-one-exact-volume-trace} gives
\[
\operatorname{tr}D^2\mathscr P_E(0)
=\int_{\partial^*E}\mathscr T_{p,q}\,d\mathcal H^{n-1}.
\]
Since the trace certificates give $\mathscr T_{p,q}\le0$ pointwise,
\[
0
\le
\int_{\partial^*E}\mathscr T_{p,q}\,d\mathcal H^{n-1}
\le0.
\]
Hence $\mathscr T_{p,q}=0$ almost everywhere on $\partial^*E$, and the strict
equality cases proved above imply
\[
\beta=\gamma=0
\quad\text{a.e. on }\partial^*E,
\]
that is,
\[
\nu_E=\pm N
\quad\mathcal H^{n-1}\text{-a.e. on }\partial^*E.
\]
This is the precise role of the algebraic certificates in the global rigidity
argument.

\subsection{Optional exact symbolic cross-check}

The calculations above are completely explicit and do not require computer
algebra.  As an independent reproducibility check, one may nevertheless verify
symbolically that
\[
(1+x)^2\mathscr T_{n-2,1}+\mathcal R_{n,x}\equiv0,
\]
\[
(n+2)^2(1+x)^4\mathscr T_{n-4,3}
+\mathcal P^{\mathbb H}_{n,x}\equiv0,
\]
and
\[
27225(1+x)^8\mathscr T_{8,7}
+\mathcal P^{\mathbb O}_{x}\equiv0.
\]
The symbolic check is supplementary only; all transformations needed to audit
the formulas by hand have been displayed above.

\section*{Statements and declarations}
\noindent\textbf{Competing interests.}
The author has no relevant financial or non-financial interests to disclose.

\noindent\textbf{Declaration of generative AI use}

During the preparation of this manuscript, the author used OpenAI's
ChatGPT (GPT-5.6 Sol) as an auxiliary tool for linguistic and stylistic
editing, exploratory algebra, symbolic consistency checks, development and
exposition of the volume-radius formulation, and literature discovery.  In
particular, it was used to cross-check the $q=1$ specialization of the
universal trace and the explicit polynomial expansions in the quaternionic
and Cayley cases, and to draw attention to standard BV and variational
references including \cite{AFP2000,Maggi2012,BarbosaDoCarmo1984,
BarbosaDoCarmoEschenburg1988,BattagliaMontefalconeMonti2019}.

AI output was treated as an auxiliary drafting and checking aid rather than
as a mathematical authority.  The arguments and formulas used in the paper
are presented explicitly in the manuscript, and the author assumes full
responsibility for their verification, the cited literature, and the final
conclusions.

\noindent\textbf{Data availability.}
No datasets were generated or analysed in the course of this study.
\section*{Acknowledgements}
The author thanks Enkelejd Hashorva and Manuel Ritor\'e for helpful
suggestions that improved this work.
The author gratefully acknowledges financial support from the Ministry
of Education, Science and Innovation of Montenegro through the grants
\emph{``Mathematical Analysis, Optimisation and Machine Learning''} and
\emph{``Complex-analytic and geometric techniques for non-Euclidean machine
learning: theory and applications.''}

\end{document}